\documentclass[11pt]{amsart}
\usepackage{graphicx}
\usepackage[mathcal]{euscript}
\usepackage{float}
\usepackage{cite}
\usepackage{verbatim}

\restylefloat{figure}

\newtheorem{theorem}{Theorem}[section]
\newtheorem{lemma}[theorem]{Lemma}
\newtheorem{corollary}[theorem]{Corollary}
\newtheorem{proposition}[theorem]{Proposition}

\theoremstyle{definition}

\newtheorem{example}[theorem]{Example}

\theoremstyle{remark}

\numberwithin{equation}{section}

\def\Span{\operatorname{span}}
\def\alt{\operatorname{alt}}

\usepackage{tikz}

\usetikzlibrary{arrows,shapes,decorations,backgrounds,patterns}

\pgfdeclarelayer{background}
\pgfdeclarelayer{background2}
\pgfdeclarelayer{background2a}
\pgfdeclarelayer{background2b}
\pgfdeclarelayer{background3}
\pgfdeclarelayer{background4}
\pgfdeclarelayer{background5}
\pgfdeclarelayer{background6}
\pgfdeclarelayer{background7}

\pgfsetlayers{background7,background6,background5,background4,background3,background2b,background2a,background2,background,main}

\definecolor{lsupurple}{RGB}{70,29,124}
\definecolor{lsugold}{RGB}{253,208, 35}

\begin{document}

\title{Turaev genus and alternating decompositions }

\author{Cody W. Armond}
\address{Department of Mathematics\\
The University of Iowa\\
Iowa City, IA}
\email{cody-armond@uiowa.edu}

\author{Adam M. Lowrance}
\address{Department of Mathematics\\
Vassar College\\
Poughkeepsie, NY} 
\email{adlowrance@vassar.edu}

\date{}

\begin{abstract}
We prove that the genus of the Turaev surface of a link diagram is determined by a graph whose vertices correspond to the boundary components of the maximal alternating regions of the link diagram. Furthermore, we use these graphs to classify link diagrams whose Turaev surface has genus one or two, and we prove that similar classification theorems exist for all genera.
\end{abstract}

\maketitle

\section{Introduction}
The discovery of the Jones polynomial \cite{Jones:Polynomial} led to the resolution of the famous Tait conjectures. In particular, Kauffman \cite{Kauffman:Bracket}, Murasugi \cite{Murasugi:JonesConjectures}, and Thistlethwaite \cite{Thistlethwaite:Breadth} use the Jones polynomial to prove that an alternating diagram of a link with no nugatory crossings has the fewest possible number of crossings. In Turaev's \cite{Turaev:SimpleProof} alternate proof of this result, he associates a closed oriented surface to each link diagram $D$, now known as the Turaev surface of $D$. Let $D$ be a diagram of a non-split link $L$ with $c(D)$ crossings, let $V_L(t)$ be the Jones polynomial of $L$, and let $g_T(D)$ be the genus of the Turaev surface of $D$. Turaev shows that
\begin{equation}
\label{eq:TuraevJonesBound}
\Span V_L(t) + g_T(D) \leq c(D).
\end{equation}
In recent years, the Turaev surface has been shown to have further connections to the Jones polynomial \cite{DFKLS:Jones,DFKLS:Determinant}, Khovanov homology \cite{CKS:KhovanovTuraev, DasLow:KhovanovTuraev}, and knot Floer homology \cite{Low:HFKTuraev, DasLow:TuraevConcordance}. 

Thistlethwaite \cite{Thistlethwaite:Breadth} uses a decomposition of a link diagram into maximal alternating pieces to compute a lower bound on crossing number similar to Inequality \eqref{eq:TuraevJonesBound}. Consider a link diagram $D$ as $4$-valent plane graph with over/under decorations at the vertices. An edge or face of $D$ should be understood to refer to an edge or face of the $4$-valent plane graph. An edge of $D$ is called {\em non-alternating} if both of its endpoints are over-strands or both of its endpoints are under-strands. An edge is called {\em alternating} if one of its endpoints is an over-strand and the other is an under-strand. Mark each non-alternating edge of $D$ with two distinct points, and in each face of $D$ connect those marked points with arcs as depicted in Figure \ref{fig:arcs}. This process results in a collection of pairwise disjoint simple closed curves $\{\gamma_1,\dots,\gamma_k\}$. The pair $(D,\{\gamma_1,\dots,\gamma_k\})$ is called the {\em alternating decomposition} of $D$.
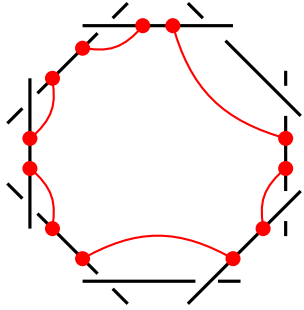
\begin{figure}[h]
$$\begin{tikzpicture}
\begin{scope}[very thick]
	\draw (.7,0) -- (2.2,0);
	\draw (2.5,0) -- (2.8,0);
	\draw (2.1,-.3) -- (3.6,1.2);
	\draw (3.4,1.2) -- (3.4,2.2);
	\draw (3.4, .8) -- (3.4,.6);
	\draw (3.4, 2.6) -- (3.4,2.8);
	\draw (3.6,2.2) -- (2.6,3.2);
	\draw (2.3,3.5) -- (2.1,3.7);
	\draw (2.7,3.4) -- (.7,3.4);
	\draw (.9,3.3) -- (.1,2.5);
	\draw (1.1,3.5) -- (1.3,3.7);
	\draw (-.1,2.3) -- (-.3,2.1);
	\draw (0,2.7) -- (0,.7);
	\draw (0.1,.9) -- (.9,.1);
	\draw (-.1,1.1) -- (-.3,1.3);
	\draw (1.1,-.1) -- (1.3,-.3);
\end{scope}

\node (1) at (.7,.3){};
\node (2) at (.3,.7){};
\node (3) at (0,1.5) {};
\node (4) at (0,1.9) {};
\node (5) at (.3,2.7) {};
\node (6) at (.7,3.1) {};
\node (7) at (1.5,3.4) {};
\node (8) at (1.9, 3.4) {};
\node (9) at (3.4,1.9) {};
\node (10) at (3.4,1.5) {};
\node (11) at (3.1,.7) {};
\node (12) at (2.7,.3) {};

\begin{scope}[line/.style={shorten >=-0.2cm,shorten <=-0.2cm},thick,red]
\fill (1) circle (.1cm);
\fill (2) circle (.1cm);
\fill (3) circle (.1cm);
\fill (4) circle (.1cm);
\fill (5) circle (.1cm);
\fill (6) circle (.1cm);
\fill (7) circle (.1cm);
\fill (8) circle (.1cm);
\fill (9) circle (.1cm);
\fill (10) circle (.1cm);
\fill (11) circle (.1cm);
\fill (12) circle (.1cm);
\path [bend left, line]   (1) edge (12);
\path [bend left, line]   (3) edge (2);
\path [bend left, line]   (5) edge (4);
\path [bend left, line]   (7) edge (6);
\path [bend left, line]   (9) edge (8);
\path [bend left, line]   (11) edge (10);

\end{scope}
\end{tikzpicture}$$
\caption{Each non-alternating edge is marked with two points. Inside of each face, draw arcs that connect marked points that are adjacent on the boundary but do not lie on the same edge of $D$.}
\label{fig:arcs}
\end{figure}

Thistlethwaite associates to $D$ a graph $G$, which we call the {\em alternating decomposition graph} of $D$, as follows. Suppose that $D$ is a connected link diagram, i.e. when $D$ is considered as a graph, it is a connected graph. If $D$ is an alternating diagram, then $G$ is a single vertex with no edges. Otherwise, the vertices of $G$ are in one-to-one correspondence with the curves $\gamma_1,\dots, \gamma_k$ of the alternating decomposition of $D$. The edges of $G$ are in one-to-one correspondence with the non-alternating edges of $D$. Let $v_i$ and $v_j$ be vertices of $G$ corresponding to curves $\gamma_i$ and $\gamma_j$ respectively. An edge of $G$ connects $v_i$ to $v_j$ if and only if the corresponding non-alternating edge of $D$ intersects both $\gamma_i$ and $\gamma_j$. If $D$ is not a connected link diagram, then $G$ is the disjoint union of the alternating decomposition graphs of its connected components. 

The plane embedding of $D$ induces an embedding of each component of $G$ onto a sphere, as described in Section \ref{section:AltDecomp}. Since each component of $G$ can be embedded on a sphere, the graph $G$ is planar. Whenever we refer to $G$ with the sphere embeddings of its components induced by $D$, we use the notation $\mathbb{G}$ and call it the {\em sphere embedding induced by $D$}. We also consider $\mathbb{G}$ as an oriented ribbon graph of genus zero. See Section \ref{section:AltDecomp} for further discussion on oriented ribbon graphs.  Each edge of $G$ can be labeled as ``$+$" or ``$-$" according to whether it corresponds to an over-strand edge of $D$ or an under-strand edge of $D$ respectively. Since the edges in each face of $\mathbb{G}$ rotate between ``$+$" and ``$-$" edges, it follows that every face has an even number of edges in its boundary. Therefore $G$ is bipartite. Also, since every curve $\gamma_i$ encloses a tangle, it follows that every vertex of $G$ has even degree. Proposition \ref{prop:altdecompgraph} below shows that a graph is an alternating decomposition graph if and only if it is planar, bipartite, and each vertex has even degree. See Section \ref{section:AltDecomp} for examples of alternating decompositions of link diagrams and their associated alternating decomposition graphs.

If $D$ has alternating decomposition curves $\{\gamma_1,\dots,\gamma_k\}$, then an {\em alternating region} of $D$ is a component of $S^2-\{\gamma_1,\dots,\gamma_k\}$ that contains crossings of $D$. As the name suggests, if one follows a strand inside of an alternating region of $D$, then the crossings will alternate between over and under. Let $r_{\text{alt}}(D)$ be the number of alternating regions in the alternating decomposition of $D$, and let $e(G)$ be the number of edges in $G$. Note that $e(G)$ is also the number of non-alternating edges in $D$. Thistlethwaite \cite{Thistlethwaite:Breadth} proves that if $D$ is a connected diagram of the link $L$, then
\begin{equation}
\label{eq:thineq}
\Span V_{L}(t)  - r_{\text{alt}}(D) + \frac{1}{2} e(G) +1\leq c(D).
\end{equation}
Bae and Morton \cite{BaeMorton:Spread} use Thistlethwaite's approach to study the extreme terms and the coefficients of the extreme terms in the Jones polynomial. Using combinatorial data from the planar dual of $\mathbb{G}$, a graph they call the {\em non-alternating spine} of $D$, they recover Inequality \eqref{eq:TuraevJonesBound} and show that it is a stronger bound than Inequality \eqref{eq:thineq}.

In this paper, we use Thistlethwaite's alternating decompositions to study the Turaev surface of a link diagram. We show that the genus of the Turaev surface of a link diagram is determined by its alternating decomposition graph. If the Turaev surface is disconnected, then its genus refers to the sum of the genera of its connected components.
\begin{theorem}
\label{thm:graphTuraev}
If $D_1$ and $D_2$ are link diagrams with isomorphic alternating decomposition graphs, then $g_T(D_1) = g_T(D_2)$.
\end{theorem}
Champanerkar and Kofman \cite{CK:Twisting} prove a version of Theorem \ref{thm:graphTuraev} in the case where the two link diagrams are related by a rational tangle replacement. Lowrance \cite{Low:Twisted} uses this special case to compute the Turaev genus of the $(3,q)$-torus links and of many other closed $3$-braids (see also \cite{AbeKish:Dealternating}).

The {\em Turaev genus} of an alternating decomposition graph $G$, denoted $g_T(G)$, is defined to be $g_T(D)$ where $D$ is a link diagram with alternating decomposition graph $G$. Corollary \ref{cor:TuraevAlgorithm} gives a recursive algorithm to compute $g_T(G)$ without any reference to link diagrams. Theorem \ref{thm:graphTuraev} coupled with our algorithm for computing $g_T(G)$ show that the genus of the Turaev surface is determined by how the various alternating regions of $D$ are glued together along the non-alternating edges of $D$. The recursive algorithm is at the core of our classification theorems.

A {\em doubled path} of length $k$ in $G$ is a subgraph of $G$ consisting of distinct vertices $v_0,\dots, v_k$ such that for each $i=1,\dots, k$ there are two distinct edges $e_{i,1}$ and $e_{i,2}$ in $G$ connecting vertices $v_{i-1}$ and $v_i$ and such that $\deg v_i=4$ for $i=1,\dots, k-1$. If $G$ is a graph with a doubled path consisting of vertices $v_0,\dots, v_k$, then let $G'$ be $G/\{e_{i,1}\cup e_{i,2}\}$, the contraction of $e_{i,1}$ and $e_{i,2}$ from $G$ for some $i$ with $1\leq i \leq k$. Then $G'$ is called a {\em doubled path contraction} of $G$. The inverse operation of lengthening a doubled path inside of $G$ is called a {\em doubled path extension} of $G$. Two alternating decomposition graphs $G_1$ and $G_2$ are called {\em doubled path equivalent} if there is a sequence of doubled path contractions and extensions transforming $G_1$ into $G_2$. Doubled path contraction/extension can make a graph non-bipartite (and hence not an alternating decomposition graph), but we do not require every graph in the sequence from $G_1$ to $G_2$ to be bipartite. Proposition \ref{prop:doubledpath} shows that if $G_1$ and $G_2$ are doubled path equivalent, then $g_T(G_1)=g_T(G_2)$.

A graph is {\em $k$-edge connected} for some positive integer $k$ if the graph remains connected whenever fewer than $k$ edges are removed. An alternating decomposition graph $G$ is called {\em reduced} if $G$ is a single vertex or every component of $G$ is $3$-edge connected. In Section \ref{section:AltDecomp}, we study the behavior of alternating decomposition graphs under connected sum. We show that for any link $L$, there exists a diagram $D$ of $L$ with reduced alternating decomposition graph such that $D$ minimizes Turaev genus. The classification theorems characterize all reduced alternating decomposition graphs of a fixed Turaev genus.

Our main theorems give classifications of all reduced alternating decomposition graphs of Turaev genus one and two. A {\em doubled cycle $C_i^2$} of length $i$ is the graph obtained from the cycle $C_i$ of length $i$ by doubling every edge.
\begin{theorem}
\label{thm:genus1}
A reduced alternating decomposition graph $G$ is of Turaev genus one if and only if $G$ is doubled path equivalent to $C_2^2$, that is if and only if $G$ is a doubled cycle of even length.
\end{theorem}
The previous theorem implies that every Turaev genus one link has a diagram $D$ obtained by connecting an even number of alternating two-tangles into a cycle, as in Figure \ref{fig:genus1}.  Dasbach and Lowrance \cite{DasLow:Signature} use Theorem \ref{thm:genus1} to compute the signature of all Turaev genus one knots and to show that either the leading or trailing coefficient of the Jones polynomial of a Turaev genus one link has absolute value one.
\begin{figure}[h]
$$\begin{tikzpicture}[thick]
\draw [bend left] (0,.5) edge (3,.5);
\draw [bend right] (0,-.5) edge (3, -.5);
\draw [bend left] (3,.5) edge (6,.5);
\draw [bend right] (3,-.5) edge (6, -.5);
\draw [bend left] (6,.5) edge (9,.5);
\draw [bend right] (6,-.5) edge (9, -.5);
\draw [bend left] (7.5,.5) edge (10.5,.5);
\draw [bend right] (7.5,-.5) edge (10.5,-.5);

\fill[white] (7.5,1) rectangle (9.1,-1);
\draw (8.25,0) node{\Large{$\dots$}};

\draw (-.8,.6) arc (270:90:.8cm);
\draw (-.8,-.6) arc (90:270:.8cm);
\draw (11.3,.6) arc (-90:90:.8cm);
\draw (11.3,-.6) arc (90:-90:.8cm);
\draw (-.8,2.2) -- (11.3,2.2);
\draw (-.8,-2.2) -- (11.3,-2.2);

\fill[white] (0,0) circle (1cm);
\draw (0,0) node {$T_1$};
\draw (0,0) circle (1cm);
\fill[white] (3,0) circle (1cm);
\draw (3,0) node {$T_2$};
\draw (3,0) circle (1cm);
\fill[white] (6,0) circle (1cm);
\draw (6,0) node {$T_3$};
\draw (6,0) circle (1cm);
\fill[white] (10.5,0) circle (1cm);
\draw (10.5,0) node {$T_{2k}$};
\draw (10.5,0) circle (1cm);

\draw (1.5,1.2) node{$+$};
\draw (1.5,-1.2) node{$-$};
\draw (4.5,1.2) node{$-$};
\draw (4.5,-1.2) node{$+$};
\draw (5.25,2) node{$-$};
\draw (5.25,-2)node{$+$};
\draw (7.2,1.2) node{$+$};
\draw (7.2,-1.2) node{$-$};
\draw (9.3,1.2) node{$+$};
\draw (9.3,-1.2)node{$-$};

\end{tikzpicture}$$
\caption{Every diagram $D$ where $g_T(D)=1$ and $G$ is reduced has alternating decomposition as above. Each two-tangle $T_i$ is alternating. A $\pm$ sign on an edge indicates that it is a non-alternating edge of $D$ with endpoints both over/under crossings respectively. The alternating decomposition graph $G$ associated to such a diagram is a doubled cycle of length $2k$.}
\label{fig:genus1}
\end{figure}
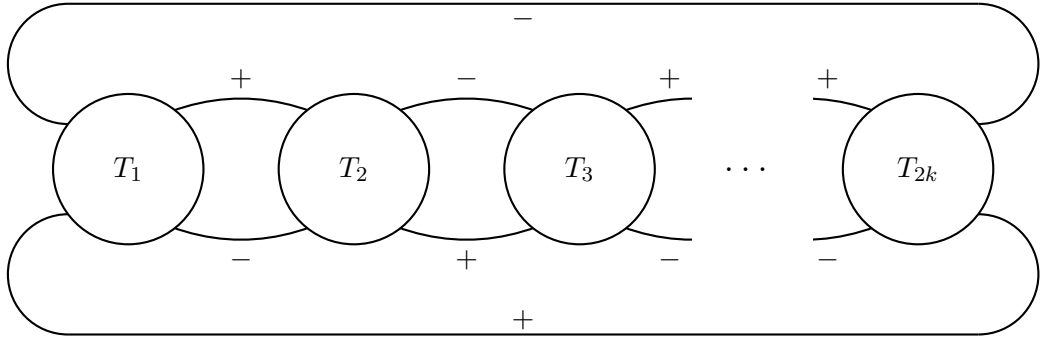

A link is {\em almost-alternating} if it is non-alternating and has a diagram $D$ that can be transformed into an alternating diagram with a single crossing change (see \cite{Adams:AlmostAlternating}). Abe and Kishimoto's work \cite{AbeKish:Dealternating} implies that all almost-alternating links have Turaev genus one. It is unknown whether there is a link with Turaev genus one that is not almost-alternating (see \cite{Low:AltDist}). The following corollary shows another relationship between almost-alternating links and Turaev genus one links.
\begin{corollary}
\label{corollary:mutant}
If $L$ is a link of Turaev genus one, then there is an almost-alternating link $L'$ such that $L$ and $L'$ are mutants of one another.
\end{corollary}

We present a similar classification theorem for reduced alternating decomposition graphs of Turaev genus two. However, instead of only one doubled path equivalence class, now there are five. Let $G_1$ and $G_2$ be two graphs. A {\em one-sum} $G_1\oplus_1 G_2$ is the graph obtained by identifying a vertex of $G_1$ with a vertex of $G_2$. Let $e_1$ be an edge in $G_1$ connecting vertices $v_1$ and $v_2$, and let $e_2$ be an edge in $G_2$ connecting vertices $u_1$ and $u_2$. A {\em two-sum} $G_1\oplus_2 G_2$ is the graph obtained by identifying the triple $(v_1,v_2,e_1)$ with $(u_1,u_2,e_2)$, and then deleting the edge corresponding to $e_1$ and $e_2$. For example the two-sum of two three-cycles $C_3\oplus_2 C_3$ is a four cycle $C_4$. Consider the following five classes of graphs, as depicted in Figure \ref{fig:genus2}.
\begin{enumerate}
\item Let $C_i^2\sqcup C_j^2$ denote the disjoint union of the doubled cycles $C_i^2$ and $C_j^2$.
\item Let $C_i^2\oplus_{1} C_j^2$ be the graph obtained identifying a vertex of the doubled cycle $C^2_i$ with a vertex of $C^2_j$.
\item Let $C_{i,j,k}$ be the graph obtained by identifying two paths of length $k$ in the cycle $C_{i+k}$ of length $i+k$ and the cycle $C_{j+k}$ of length $j+k$. Furthermore, let $C_{i,j,k}^2$ be the graph $C_{i,j,k}$ with each edge doubled.
\item Let $K_4(p,q)$ be the graph obtained by replacing two non-adjacent edges of the complete graph $K_4$ with doubled paths of lengths $p$ and $q$ respectively.
\item Let $K_4(p)$ be the graph $K_4$ with one edge replaced by a doubled path of length $p$. Let $K_4(p) \oplus_2 K_4(q)$ be the two-sum of $K_4(p)$ and $K_4(q)$ taken along the unique edge in each summand that is not contained in or adjacent to the doubled path.
\end{enumerate}
\begin{figure}[h]
$$\begin{tikzpicture}
\draw (.5,0) circle (.6cm);
\draw (.5,0) circle (.4cm);

\fill[black!20!white] (0,0) circle (.2cm);
\draw (0,0) circle (.2cm);
\fill[black!20!white] (1,0) circle (.2cm);
\draw (1,0) circle (.2cm);

\begin{scope}[xshift = 1.8cm]
	\draw (.5,0) circle (.6cm);
	\draw (.5,0) circle (.4cm);

	\fill[black!20!white] (0,0) circle (.2cm);
	\draw (0,0) circle (.2cm);
	\fill[black!20!white] (1,0) circle (.2cm);
	\draw (1,0) circle (.2cm);
\end{scope}

\draw (1.4,-1) node {$C^2_2\sqcup C^2_2$};

\begin{scope}[xshift = 4.5cm]
\draw (.45,0) circle (.6cm);
\draw (.45,0) circle (.4cm);
\draw (1.55,0) circle (.6cm);
\draw (1.55,0) circle (.4cm);

\fill[black!20!white] (-.05,0) circle (.2cm);
\draw (-.05,0) circle (.2cm);
\fill[black!20!white] (1,0) circle (.2cm);
\draw (1,0) circle (.2cm);
\fill[black!20!white] (2.05,0) circle (.2cm);
\draw (2.05,0) circle (.2cm);

\draw (1,-1) node{$C^2_2\oplus_{1} C^2_2$};

\end{scope}

\begin{scope}[xshift = 9cm]
	\draw (.5,0) circle (.6cm);
	\draw (.5,0) circle (.4cm);
	\draw (0,.1) -- (1,.1);
	\draw (0,-.1) -- (1, -.1);

	\fill[black!20!white] (0,0) circle (.2cm);
	\draw (0,0) circle (.2cm);
	\fill[black!20!white] (1,0) circle (.2cm);
	\draw (1,0) circle (.2cm);
	
	\draw (.5,-1) node{$C_{1,1,1}^2$};
\end{scope}

\begin{scope}[yshift = -5cm]
\draw (0,0) -- (2,0) -- (2,2) -- (0,2) -- (0,0);
\draw (0,.1) -- (2,2.1);
\draw (0,-.1) -- (2,1.9);
\draw[bend left] (-.1,2.1) edge (3,3.1);
\draw[bend left] (0,1.9) edge (3,2.9);
\draw [bend left] (2.9, 3) edge (2,0.1);
\draw [bend left] (3.1,3) edge (2.1,-.1);

\fill[black!20!white] (3,3) circle (.2cm);
\draw (3,3) circle (.2cm);

\fill[black!20!white](0,0) circle (.2cm);
\fill[black!20!white](2,0) circle (.2cm);
\fill[black!20!white](2,2) circle (.2cm);
\fill[black!20!white](0,2) circle (.2cm);
\draw (0,0) circle (.2cm);
\draw (2,0) circle (.2cm);
\draw (0,2) circle (.2cm);
\draw (2,2) circle (.2cm);

\fill[black!20!white] (1,1) circle (.2cm);
\draw (1,1) circle (.2cm);

\draw (1,-1) node{$K_4(2,2)$};

\end{scope}

\begin{scope}[xshift = 8cm, yshift=-5.3cm,scale = .8]
	
	\draw (0,0) -- (-3,2) -- (0,4) -- (-1,2) -- (0,0) -- (1,2) -- (0,4) -- (3,2) -- (0,0);
	\draw (-3,2.1) -- (-1,2.1);
	\draw (-3,1.9) -- (-1,1.9);
	\draw (3,2.1) -- (1,2.1);
	\draw (3,1.9) -- (1,1.9);

	\fill[black!20!white] (0,0) circle (.2cm);
	\fill[black!20!white] (0,4) circle (.2cm);
	\draw (0,0) circle (.2cm);
	\draw (0,4) circle (.2cm);

	\fill[black!20!white] (-3,2) circle (.2cm);
	\fill[black!20!white] (-2,2) circle (.2cm);
	\fill[black!20!white] (-1,2) circle (.2cm);
	\fill[black!20!white] (1,2) circle (.2cm);
	\fill[black!20!white] (2,2) circle (.2cm);
	\fill[black!20!white] (3,2) circle (.2cm);
	\draw (-3, 2) circle (.2cm);
	\draw (-2, 2) circle (.2cm);
	\draw (-1, 2) circle (.2cm);
	\draw (1, 2) circle (.2cm);
	\draw (2, 2) circle (.2cm);
	\draw (3, 2) circle (.2cm);

\end{scope}
\draw(8,-6) node{$K_4(2)\oplus_2 K_4(2)$};

\end{tikzpicture}$$
\caption{Representatives of the five doubled path equivalence classes of reduced alternating decompositions graphs of Turaev genus two. Informally, a Turaev genus two link diagram is obtained by inserting appropriate alternating tangles inside of the vertices of these graphs. In the case of $C_2^2\sqcup C_2^2$ one should insert an annular alternating region bounded by two curves that correspond to vertices in distinct components. See Figure \ref{figure:annulus} for an example of a connected link diagram with disconnected alternating decomposition graph.}
\label{fig:genus2}
\end{figure}
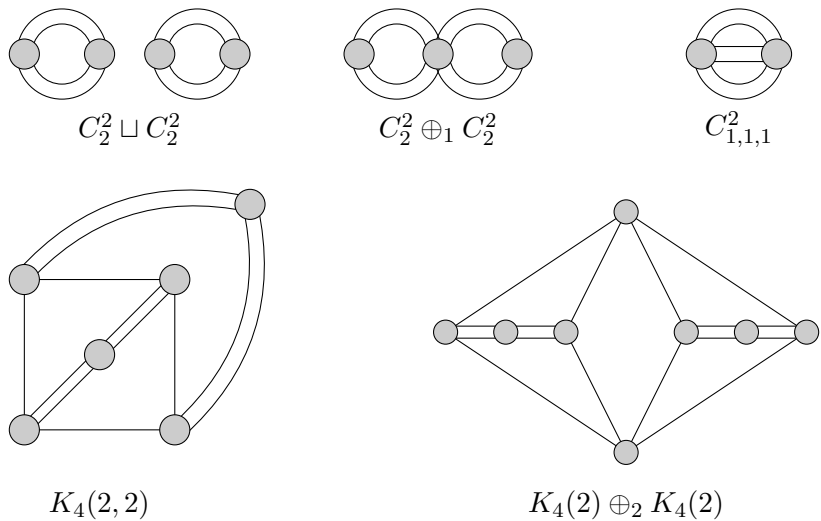
The graphs in the above families are not necessarily bipartite (depending on their parameters). Informally, the subsequent theorem states that a reduced alternating decomposition graph has Turaev genus two if and only if it is in one of the above five families and it is bipartite. The precise statement uses doubled path equivalence.
\begin{theorem}
\label{thm:genus2}
A reduced alternating decomposition graph $G$ is of Turaev genus two if and only if $G$ is doubled path equivalent to one of the following five graphs:
\begin{enumerate}
\item $C_2^2\sqcup C_2^2$,
\item $C_2^2\oplus_1 C_2^2$,
\item $C^2_{1,1,1}$,
\item $K_4(2,2)$, or
\item $K_4(2)\oplus_2 K_4(2)$.
\end{enumerate}
\end{theorem}

Seungwon Kim  \cite{Kim:TuraevClassification} has independently proved versions of Theorems \ref{thm:genus1} and \ref{thm:genus2}. The following theorem shows that for each non-negative integer $k$, there exists a similar classification of reduced alternating decomposition graphs of Turaev genus $k$.
\begin{theorem}
\label{thm:arbgenus}
Let $k$ be a non-negative integer. There are a finite number of doubled path equivalence classes of reduced alternating decomposition graphs $G$ with Turaev genus $k$.
\end{theorem}

This paper is organized as follows. In Section \ref{section:Turaev}, we review background material on the Turaev surface and discuss its connections to other areas of knot theory. In Section \ref{section:AltDecomp}, we give the algorithm to compute $g_T(G)$ and prove Theorem \ref{thm:graphTuraev}. In Section \ref{section:zero}, we classify alternating decomposition graphs of Turaev genus zero and show that all links have a Turaev genus minimizing diagram whose alternating decomposition graph is reduced. In Section \ref{section:classification}, we prove the three main classification theorems (Theorems \ref{thm:genus1}, \ref{thm:genus2}, and \ref{thm:arbgenus}).

The authors thank Sergei Chmutov, Oliver Dasbach, Nathan Druivenga, Charles Frohman, and Thomas Kindred for their helpful comments.

\section{The Turaev surface}
\label{section:Turaev}
In this section, we give the construction of the Turaev surface of a link diagram $D$ and discuss its connections to other link invariants. For a more in depth summary, see Champanerkar and Kofman's recent survey \cite{CK:Survey}.

Each link diagram $D$ has an associated Turaev surface $F(D)$, constructed as follows. Figure \ref{figure:smoothing} shows the $A$ and $B$ resolutions of a crossing in $D$.
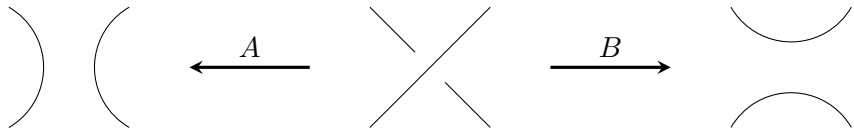
\begin{figure}[h]
$$\begin{tikzpicture}[>=stealth, scale=.8]
\draw (-1,-1) -- (1,1);
\draw (-1,1) -- (-.25,.25);
\draw (.25,-.25) -- (1,-1);
\draw (-3,0) node[above]{$A$};
\draw[->,very thick] (-2,0) -- (-4,0);
\draw (3,0) node[above]{$B$};
\draw[->,very thick] (2,0) -- (4,0);
\draw (-5,1) arc (120:240:1.1547cm);
\draw (-7,-1) arc (-60:60:1.1547cm);
\draw (5,1) arc (210:330:1.1547cm);
\draw (7,-1) arc (30:150:1.1547cm);
\end{tikzpicture}$$
\caption{The $A$ and $B$ resolutions of a crossing}
\label{figure:smoothing}
\end{figure}
The collection of simple closed curves obtained by performing either an $A$-resolution or a $B$-resolution for each crossing of $D$ is a {\em state} of $D$. Performing an $A$-resolution for every crossing results in the {\em all-$A$ state} of $D$. Similarly, performing a $B$-resolution for every crossings results in the {\em all-$B$ state} of $D$. Let $s_A(D)$ and $s_B(D)$ denote the number of components in the all-$A$ and all-$B$ states of $D$ respectively.

To construct the Turaev surface, we take a cobordism from the all-$B$ state of $D$ to the all-$A$ state of $D$ such that the cobordism consists of bands away from the crossings of $D$ and saddles in neighborhoods of the crossing, as depicted in Figure \ref{figure:saddle}. Finally, to obtain $F(D)$, we cap off the boundary components of the cobordism with disks.
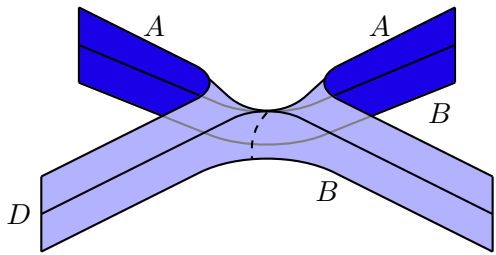
\begin{figure}[h]
$$\begin{tikzpicture}
\begin{scope}[thick]
\draw [rounded corners = 10mm] (0,0) -- (3,1.5) -- (6,0);
\draw (0,0) -- (0,1);
\draw (6,0) -- (6,1);
\draw [rounded corners = 5mm] (0,1) -- (2.5, 2.25) -- (0.5, 3.25);
\draw [rounded corners = 5mm] (6,1) -- (3.5, 2.25) -- (5.5,3.25);
\draw [rounded corners = 5mm] (0,.5) -- (3,2) -- (6,.5);
\draw [rounded corners = 7mm] (2.23, 2.3) -- (3,1.6) -- (3.77,2.3);
\draw (0.5,3.25) -- (0.5, 2.25);
\draw (5.5,3.25) -- (5.5, 2.25);
\end{scope}

\begin{pgfonlayer}{background2}
\fill [blue!30!white]  [rounded corners = 10 mm] (0,0) -- (3,1.5) -- (6,0) -- (6,1) -- (3,2) -- (0,1); 
\fill [blue!30!white] (6,0) -- (6,1) -- (3.9,2.05) -- (4,1);
\fill [blue!30!white] (0,0) -- (0,1) -- (2.1,2.05) -- (2,1);
\fill [blue!30!white] (2.23,2.28) --(3.77,2.28) -- (3.77,1.5) -- (2.23,1.5);

\fill [white, rounded corners = 7mm] (2.23,2.3) -- (3,1.6) -- (3.77,2.3);
\fill [blue!30!white] (2,2) -- (2.3,2.21) -- (2.2, 1.5) -- (2,1.5);
\fill [blue!30!white] (4,2) -- (3.7, 2.21) -- (3.8,1.5) -- (4,1.5);
\end{pgfonlayer}

\begin{pgfonlayer}{background4}
\fill [blue!90!red] (.5,3.25) -- (.5,2.25) -- (3,1.25) -- (2.4,2.2);
\fill [rounded corners = 5mm, blue!90!red] (0.5,3.25) -- (2.5,2.25) -- (2,2);
\fill [blue!90!red] (5.5,3.25) -- (5.5,2.25) -- (3,1.25) -- (3.6,2.2);
\fill [rounded corners = 5mm, blue!90!red] (5.5, 3.25) -- (3.5,2.25) -- (4,2);
\end{pgfonlayer}

\draw [thick] (0.5,2.25) -- (1.6,1.81);
\draw [thick] (5.5,2.25) -- (4.4,1.81);
\draw [thick] (0.5,2.75) -- (2.1,2.08);
\draw [thick] (5.5,2.75) -- (3.9,2.08);

\begin{pgfonlayer}{background}
\draw [black!50!white, rounded corners = 8mm, thick] (0.5, 2.25) -- (3,1.25) -- (5.5,2.25);
\draw [black!50!white, rounded corners = 7mm, thick] (2.13,2.07) -- (3,1.7)  -- (3.87,2.07);
\end{pgfonlayer}
\draw [thick, dashed, rounded corners = 2mm] (3,1.85) -- (2.8,1.6) -- (2.8,1.24);
\draw (0,0.5) node[left]{$D$};
\draw (1.5,3) node{$A$};
\draw (4.5,3) node{$A$};
\draw (3.8,.8) node{$B$};
\draw (5.3, 1.85) node{$B$};
\end{tikzpicture}$$
\caption{In a neighborhood of each crossing of $D$ a saddle surface transitions between the all-$A$ and all-$B$ states.}
\label{figure:saddle}
\end{figure}
The Turaev surface $F(D)$ is oriented, and we denote the genus of the Turaev surface of $D$ by $g_T(D)$. If the Turaev surface (or any oriented surface) is disconnected, then when we refer to its genus, we mean the sum of the genera of its connected components. Let $k(D)$ be the number of split components of the diagram $D$, i.e. the number of graph components of $D$ when $D$ is considered as a $4$-valent graph whose vertices are the crossings. Also, let $c(D)$ be the number of crossings of $D$. It can be shown that 
\begin{equation}
\label{equation:Turaevgenus}
g_T(D) = \frac{1}{2}(2k(D) + c(D) -s_A(D) - s_B(D)).
\end{equation}
The {\em Turaev genus} $g_T(L)$ of a link $L$ is the minimum genus of the Turaev surface of $D$ where $D$ is any diagram of $L$, i.e.
$$g_T(L) = \min\{g_T(D)~|~D~\text{is a diagram of}~L\}.$$

Turaev \cite{Turaev:SimpleProof} constructs his surface in a slightly different, but equivalent way. Turaev's construction allows us to see that a diagram $D$ of the link $L$ can be considered as a $4$-valent graph simultaneously embedded on the sphere and the Turaev surface $F(D)$. First consider $D$ as embedded on a sphere $S$. Then $L$ can be embedded into $S^3$ by replacing crossings of $D$ with suitably small balls where one strand passes over the other as in Figure \ref{figure:CrossingBall}. 

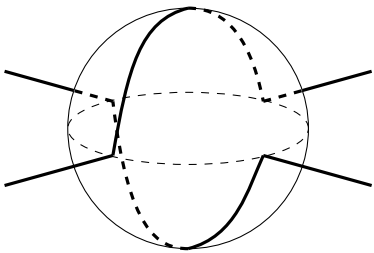
\begin{figure}[h]
$$\begin{tikzpicture}[scale=.8]
\draw (0,0) circle (2cm);
\draw[dashed] (0,0) ellipse (2cm and .6cm);
\draw[very thick] (-1.25,-.45) -- (-3.05,-.95);
\draw[very thick] (0,2) .. controls (-.8,1.8) and (-1,1) .. (-1.25,-.45);
\draw[dashed, very thick] (1.25,.45) -- (1.9,.63);
\draw[very thick] (1.9,.63) -- (3.05,.95);
\draw[dashed, very thick] (0,2) .. controls (.5,2) and (1,1.7) .. (1.25, .45);

\draw[dashed, very thick] (-1.25, .45) -- (-1.9, .63);
\draw[very thick] (1.25,-.45) -- (3.05, -.95);
\draw[very thick] (0,-2) .. controls (.8,-1.8) and (1,-1) .. (1.25,-.45);

\draw[dashed, very thick] (0,-2) .. controls (-.5,-2) and (-1,-1.7) .. (-1.25,.45);
\draw[very thick] (-3.05,.95) -- ( -1.9, .63);
\end{tikzpicture}$$
\caption{A crossing ball shows how $L$ is embedded near a crossing of $D$.}
\label{figure:CrossingBall}
\end{figure}

We construct the Turaev surface of $D$ by first replacing each crossing of $D$ with the disk that is the intersection of the associated crossing ball and $S$. Each alternating edge of $D$ is replaced with an untwisted band that lies completely in the projection sphere $S$. Each non-alternating edge of $D$ is replaced with a twisted band. One arc on the boundary of the twisted band will be an arc in a component of the all-$A$ state of $D$, and one arc on the boundary of the twisted band will be an arc in a component of the all-$B$ state of $D$. The band can be twisted so that the arc corresponding to the all-$A$ state lies in the union of $S$ and its exterior, while the arc corresponding to the all-$B$ state lies in the union of $S$ and its interior. After replacing each crossing of $D$ with a band, the boundary of the resulting surface is the union of the all-$A$ state of $D$ and the all-$B$ state of $D$. Moreover, the boundary components corresponding to the all-$A$ state lie in the union of $S$ and its exterior, and the boundary components corresponding to the all-$B$ state lie in the union of $S$ and its interior. Therefore, the boundary components of this surface can be capped off with disks embedded in $S^3$, and the resulting surface is the Turaev surface $F(D)$. By projecting the link to $S$ in the crossing balls, one can consider the diagram $D$ to be embedded on both $S$ and the Turaev surface $F(D)$. See Figure \ref{figure:TuraevAlternate}.

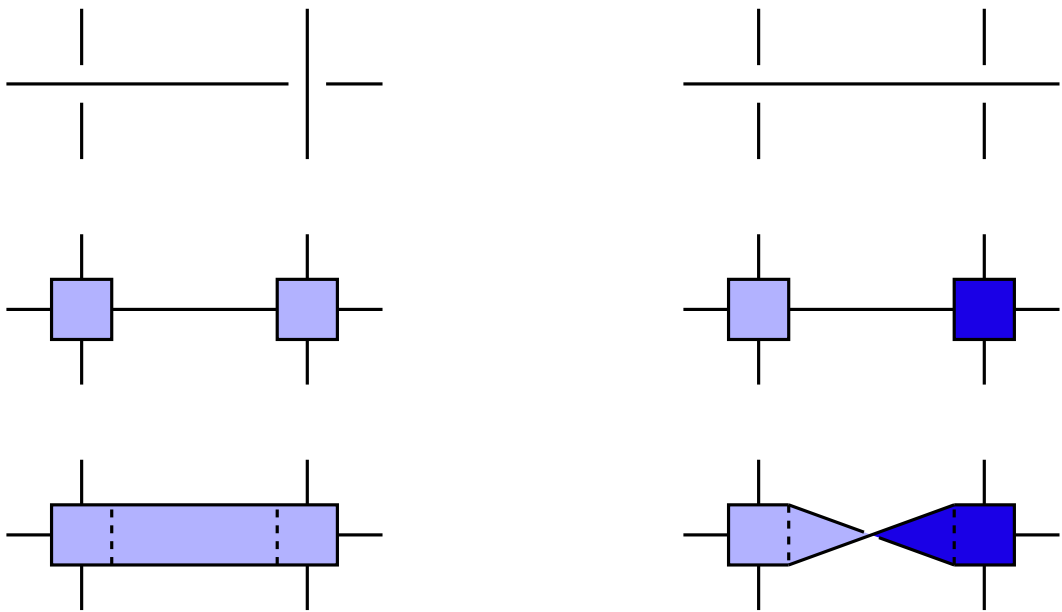
\begin{figure}[h]
$$\begin{tikzpicture}

\draw [very thick] (0,1) -- (3.75,1);
\draw [very thick] (1,2) -- (1,1.25);
\draw [very thick] (1,0) -- (1, .75);
\draw [very thick] (4,2) -- (4,0);
\draw [very thick] (4.25,1) -- (5,1);

\begin{scope}[yshift = -3cm];
	\begin{pgfonlayer}{background}
		\draw [very thick] (0,1) -- (3.75,1);
		\draw [very thick] (1,2) -- (1,1.25);
		\draw [very thick] (1,0) -- (1, .75);
		\draw [very thick] (4,2) -- (4,0);
		\draw [very thick] (4.25,1) -- (5,1);
	\end{pgfonlayer}
	
	\fill[blue!30!white] (.6,.6) rectangle (1.4,1.4);
	\draw [very thick] (.6, .6) rectangle (1.4, 1.4);
	\fill[blue!30!white] (3.6,.6) rectangle (4.4,1.4);
	\draw [very thick] (3.6,.6) rectangle (4.4,1.4);
\end{scope}

\begin{scope}[yshift = -6cm];
	\begin{pgfonlayer}{background}
		\draw [very thick] (0,1) -- (3.75,1);
		\draw [very thick] (1,2) -- (1,1.25);
		\draw [very thick] (1,0) -- (1, .75);
		\draw [very thick] (4,2) -- (4,0);
		\draw [very thick] (4.25,1) -- (5,1);
	\end{pgfonlayer}
	
	\fill[blue!30!white] (.6,.6) rectangle (4.4,1.4);
	\draw[very thick] (.6,.6) rectangle (4.4,1.4);
	\draw [very thick, dashed] (1.4,.6) -- (1.4,1.4);
	\draw [very thick, dashed] (3.6,.6) -- (3.6,1.4);
\end{scope}

\begin{scope}[xshift = 9cm]
	\draw [very thick] (0,1) -- (5,1);
	\draw [very thick] (1,2) -- (1,1.25);
	\draw [very thick] (1,0) -- (1, .75);
	\draw [very thick] (4,2) -- (4,1.25);
	\draw [very thick] (4,.75) -- (4,0);

	\begin{scope}[yshift = -3cm];
		\begin{pgfonlayer}{background}
			\draw [very thick] (0,1) -- (3.75,1);
			\draw [very thick] (1,2) -- (1,1.25);
			\draw [very thick] (1,0) -- (1, .75);
			\draw [very thick] (4,2) -- (4,0);
			\draw [very thick] (4.25,1) -- (5,1);
		\end{pgfonlayer}
	
		\fill[blue!30!white] (.6,.6) rectangle (1.4,1.4);
		\draw [very thick] (.6, .6) rectangle (1.4, 1.4);
		\fill[blue!90!red] (3.6,.6) rectangle (4.4,1.4);
		\draw [very thick] (3.6,.6) rectangle (4.4,1.4);
	\end{scope}

	\begin{scope}[yshift = -6cm];
		\begin{pgfonlayer}{background}
			\draw [very thick] (0,1) -- (1,1);
			\draw [very thick] (1,2) -- (1,1.25);
			\draw [very thick] (1,0) -- (1, .75);
			\draw [very thick] (4,2) -- (4,0);
			\draw [very thick] (4.25,1) -- (5,1);
		\end{pgfonlayer}

		\fill[blue!30!white] (.6,1.4) -- (1.4,1.4) -- (2.5,1) -- (1.4,.6) -- (.6,.6);
		\draw [very thick] (1.4,1.4) -- (.6,1.4) -- (.6,.6) -- (1.4, .6);
		\draw [very thick, dashed] (1.4,1.4) -- (1.4,.6);
		
		\fill[blue!90!red] (4.4,.6) -- (3.6,.6) -- (2.5,1) -- (3.6,1.4) -- (4.4,1.4);
		\draw [very thick] (3.6,.6) -- (4.4,.6) -- (4.4,1.4) -- (3.6,1.4);
		\draw [very thick, dashed] (3.6,.6) -- (3.6, 1.4);

		\draw [very thick] (1.4,.6) -- (3.6,1.4);	
		\draw [very thick] (1.4,1.4) -- (2.4, 1.0364);
		\draw [very thick] (2.6, .96364) -- (3.6,.6);
	\end{scope}
\end{scope}

\end{tikzpicture}$$
\caption{The disks and band associated to an alternating edge are on the left, and the disks and band associated to a non-alternating edge are on the right.}
\label{figure:TuraevAlternate}
\end{figure}

The Turaev surface of a link diagram and the Turaev genus of a link have the following properties. Proofs of these facts can be found in \cite{Turaev:SimpleProof, DFKLS:Jones}.
\begin{enumerate}
\item The Turaev surface $F(D)$ is a Heegaard surface in $S^3$, that is $S^3-F(D)$ is a union of two handelbodies.
\item The diagram $D$ is alternating on $F(D)$.
\item The Turaev surface is a sphere if and only if $D$ is a connected sum of alternating diagrams. Consequently, $g_T(L)=0$ if and only if $L$ is alternating.
\item The complement $F(D)-D$ is a collection of disks.
\end{enumerate}
The above conditions do not completely characterize Turaev surfaces. Let $g_{\alt}(L)$ be the minimal genus of Heegaard surface $F$ in $S^3$ on which the link $L$ has an alternating projection such that the complement of that projection to $F$ is a collection of disks. Adams \cite{Adams:ToroidallyAlternating} studies knots and links where $g_{\alt}(L) =1$, and Balm \cite{Balm:thesis} studies the behavior of $g_{\alt}(L)$ under connected sum. Lowrance \cite{Low:AltDist} constructs a family of links where $g_{\alt}(L)=1$, but the Turaev genus is arbitrarily large. Armond, Druivenga, and Kindred \cite{ADK:HeegaardTuraev} show how to determine whether a surface satisfying the above conditions is a Turaev surface using Heegaard diagrams. Indeed, the Heegaard diagrams corresponding to Turaev surfaces of genus one first inspired Theorem \ref{thm:genus1} and the subsequent work in this paper.

Like many link invariants defined as minimums over all diagrams, there is no algorithm to compute the Turaev genus of a link. Instead, our computations rely on various bounds of Turaev genus. The first bound follows immediately from Inequality \eqref{eq:thineq}. We have
$$g_T(L) \leq c(L) - \Span V_L(t)$$
where $c(L)$ is the minimum crossing number of $L$. Several other bounds on Turaev genus come from link homologies. 

Khovanov \cite{Khovanov:Jones} constructs a categorification $Kh(L)$ of the Jones polynomial, now known as Khovanov homology. Khovanov homology is a bi-graded $\mathbb{Z}$-module with homological grading $i$ and quantum grading $j$, and one may write $Kh(L)$ as a direct sum over its bi-graded summands $Kh(L)=\bigoplus_{i,j} Kh^{i,j}(L)$. Define 
\begin{align*}
\delta_{\min}(Kh(L)) =& \min\{j-2i | Kh^{i,j}(L)\neq 0\}~\text{and}\\
 \delta_{\max}(Kh(L))=&\max\{j-2i | Kh^{i,j}(L)\neq 0\}.
 \end{align*}
Champanerkar, Kofman, and Stoltzfus \cite{CKS:KhovanovTuraev} show that
\begin{equation}
\label{eq:KhovanovTuraev}
\delta_{\max}(Kh(L)) - \delta_{\min}(Kh(L)) -2 \leq 2g_T(L).
\end{equation}

A link diagram $D$ is {\em adequate} if the number of components in the all-$A$ (respectively all-$B$) state is strictly greater than the number of components in every state containing exactly one $B$-resolution (respectively exactly one $A$-resolution). A link is {\em adequate} if it has an adequate diagram. Khovanov \cite{Khovanov:Patterns} studies the Khovanov homology of adequate links, and Abe \cite{Abe:Adequate} proves that Inequality \eqref{eq:KhovanovTuraev} is tight when $L$ is adequate.

Ozsv\'ath and Szab\'o \cite{OS:HFK} and independently Rasmussen \cite{Rasmussen:Thesis} construct a categorification $\widehat{HFK}(K)$ of the Alexander polynomial of a knot $K$, called knot Floer homology. Knot Floer homology is also a bi-graded $\mathbb{Z}$-module with homological (or Maslov) grading $m$ and Alexander grading $s$, and one may write $\widehat{HFK}(K)$ as a direct sum over its bi-graded summands $\widehat{HFK}(K)=\bigoplus_{m,s}\widehat{HFK}_m(K,s)$. Define
\begin{align*}
\delta_{\min}(\widehat{HFK}(K)) = &\min\{s-m | \widehat{HFK}_m(K,s)\neq 0\}~\text{and}\\
\delta_{\max}(\widehat{HFK}(K))=&\max\{s-m | \widehat{HFK}_m(K,s)\neq 0\}.
\end{align*}
Lowrance \cite{Low:HFKTuraev} shows that 
\begin{equation}
\label{eq:HFKTuraev}
\delta_{\max}(\widehat{HFK}(K)) - \delta_{\min}(\widehat{HFK}(K)) - 1 \leq g_T(K).
\end{equation}

Let $\sigma(K)$ be the signature of $K$, let $\tau(K)$ be the Ozsv\'ath-Szab\'o $\tau$-invariant \cite{OS:Tau}, and let $s(K)$ be the Rasmussen $s$-invariant \cite{Rasmussen:Slice}. Dasbach and Lowrance \cite{DasLow:TuraevConcordance} show that
\begin{align}
\left| \tau(K) + \frac{\sigma(K)}{2}\right| \leq & g_T(K),\\
\frac{|s(K) + \sigma(K)|}{2} \leq & g_T(K),~\text{and}\\
\label{eq:taus}
\left| \tau(K) - \frac{s(K)}{2}\right| \leq & g_T(K).
\end{align}
Essentially all known computations of the Turaev genus of a link rely on some inequality among \eqref{eq:KhovanovTuraev} through \eqref{eq:taus}. Finding a new method for computing the Turaev genus remains a challenging open question.

\section{Alternating decomposition graphs}
\label{section:AltDecomp}
Throughout this section, we assume that $D$ is a link diagram, $G$ is the alternating decomposition graph of $D$, and $\mathbb{G}$ is the graph $G$ with the sphere embedding induced by $D$. We begin the section with some examples.

\begin{example}Figure \ref{figure:9_42} shows a diagram $D$ of the knot $9_{42}$ from Rolfsen's table, along with its alternating decomposition curves $\{\gamma_1,\gamma_2\}$. Since the alternating decomposition of $D$ has two curves that both intersect the same four non-alternating edges of $D$, it follows that the alternating decomposition graph of $D$ is $G=C_2^2$, the graph with two vertices and four parallel edges between them. In this example, $g_T(D)=1$ and since $9_{42}$ is non-alternating, it follows that $g_T(L)=1$.
\end{example}
\begin{figure}[h]
$$\begin{tikzpicture}

\fill[black!10!white] (-4,-2) rectangle (4,4.5);
\fill[white] (0,1) ellipse(1.2cm and 2cm);
\fill[black!10!white] (0,1) ellipse(1cm and 1.6cm);

\draw[rounded corners = 5mm] (.25,.25) -- (1,1) -- (0,2);
\draw[rounded corners = 5mm] (-.25,1.75) -- (-1,1) -- (0,0);
\draw[rounded corners = 5mm] (0,2) -- (-2,4) -- (-4,2) -- (-3.25,1.25);
\draw[rounded corners = 5mm] (-1.25,2.75) -- (-4,0) -- (-2,-2) -- (-1.25,-1.25);
\draw (-.75, -.75) -- (-.25, -.25);
\draw[rounded corners = 5mm] (-2.75, .75) -- (0,-2) -- (.75,-1.25);
\draw[rounded corners = 5mm] (0,0) -- (2,-2) -- (4,0) -- (3.25, .75);
\draw[rounded corners = 4mm] (1.25,-.75) -- (4,2) -- (1.5,4.5) -- (1.25,4.25);
\draw[rounded corners = 4mm] (2.75,1.25) -- (.5, 3.5) -- (.75,3.75);
\draw (.25,2.25) -- (.75, 2.75);
\draw[rounded corners = 4mm] (1.25, 3.25) -- (1.5,3.5) -- (.5,4.5) -- (-.75,3.25);

\draw[red, very thick] (0,1) ellipse(1cm and 1.6cm);
\draw[red, very thick] (0,1) ellipse(1.2cm and 2cm);

\draw (0,-2) node[below]{$D$};

\begin{scope}[xshift = 2cm]
\draw (6.5,1) circle (.4 cm);
\draw (6.5,1) circle (.6cm);

\fill[black!20!white] (6,1)circle (.2cm);
\draw (6,1) circle (.2cm);

\fill[black!20!white] (7,1) circle (.2cm);
\draw (7,1) circle (.2cm);

\draw (6.5,-2) node[below]{$G$};
\end{scope}

\end{tikzpicture}$$
\caption{A diagram $D$ of $9_{42}$ with its alternating regions shaded and its alternating decomposition graph $G=C_2^2$.}
\label{figure:9_42}
\end{figure}
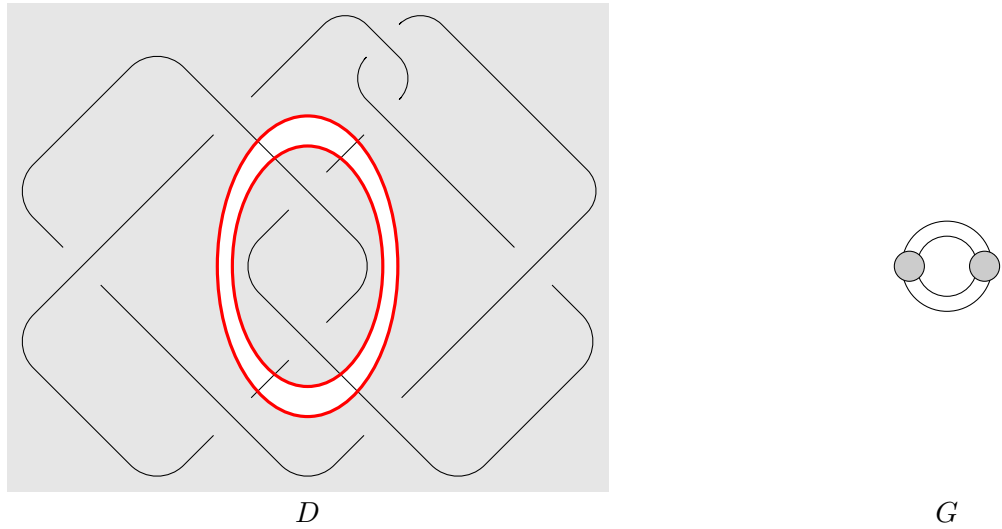

\begin{example}Figure \ref{figure:annulus} shows a connected link diagram $D$ with a disconnected alternating decomposition graph $G$. The alternating decomposition graph $G$ is disconnected when $D$ has an alternating region with more than one boundary component. In this case, the alternating decomposition graph $G$ is $C_2^2\sqcup C_2^2$, the disjoint union of two doubled two cycles. The disjoint union of two copies of the diagram from Figure \ref{figure:9_42} also has $C_2^2\sqcup C_2^2$ as its alternating decomposition graph.
\end{example}
\begin{figure}[h]
$$\begin{tikzpicture}[rounded corners = 5mm, thick]

\fill[black!10!white] (-2.5,-2) rectangle (3.5,5);
\fill[white] (-1.6, 3.6) rectangle (2.6,-1.6);
\fill[black!10!white](-1.3, 3.3) rectangle (2.3,-1.3);
\fill[white](-.6,2.7) rectangle (1.6, -.7);
\fill[black!10!white] (-.3,2.2) rectangle (1.3,-.2);

\draw (.4,.6) -- (0,1) -- (1,2) -- (1,4.7) -- (0,4.7) -- (0,4.1);
\draw (.6,1.4) -- (1,1) -- (0,0) -- (-2,0) -- (-2,4) -- (.9,4);
\draw (.6,.4) -- (1,0) -- (1.9,0);
\draw (2.1,0) -- (3,0) -- (3,4) -- (1.1,4);
\draw (.9,3) -- (-1,3) -- (-1,.1);
\draw (-1,-.1) -- (-1,-1) -- (2,-1) -- (2,3) -- (1.1,3);
\draw (0,3.9) -- (0,3.1);
\draw (0,2.9) -- (0,2) -- (.4,1.6);

\draw[red] (-.3,2.2) rectangle (1.3,-.2);
\draw[red] (-.6,2.7) rectangle (1.6, -.7);
\draw[red] (-1.3, 3.3) rectangle (2.3,-1.3);
\draw[red] (-1.6, 3.6) rectangle (2.6,-1.6);

\draw (0.5,-2) node[below]{$D$};

\draw (6.5,1) circle (.4 cm);
\draw (6.5,1) circle (.6cm);

\fill[black!20!white] (6,1)circle (.2cm);
\draw (6,1) circle (.2cm);

\fill[black!20!white] (7,1) circle (.2cm);
\draw (7,1) circle (.2cm);

\begin{scope}[xshift = 2cm]
\draw (6.5,1) circle (.4 cm);
\draw (6.5,1) circle (.6cm);

\fill[black!20!white] (6,1)circle (.2cm);
\draw (6,1) circle (.2cm);

\fill[black!20!white] (7,1) circle (.2cm);
\draw (7,1) circle (.2cm);
\end{scope}

\draw (7.5,-2) node[below]{$G$};

\end{tikzpicture}$$
\caption{The alternating decomposition of $D$ has an annular alternating region. Hence its alternating decomposition graph $G$ is disconnected.}
\label{figure:annulus}
\end{figure}
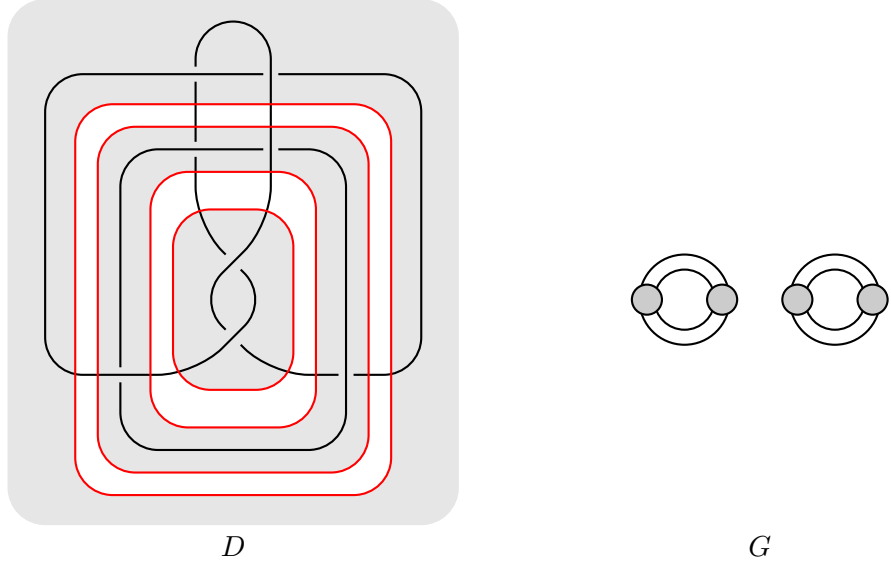

The embedding of $D$ into the plane induces an embedding of each component of the alternating decomposition graph $G$ onto a sphere. Each curve $\gamma_i$ of the alternating decomposition of $D$ is incident to two regions, precisely one of which contains crossings of $D$. In the examples of Figures \ref{figure:9_42} and \ref{figure:annulus}, the alternating regions with crossings are shaded, and the regions without crossings are unshaded. If $\gamma_i$ and $\gamma_j$ are different boundary curves of the same alternating region, then their associated vertices belong to different components of $G$. Let $\gamma_{i_1},\dots, \gamma_{i_k}$ be the curves of the alternating decomposition graph associated to all of the vertices of a particular component of $G$. One may consider the diagram $D$ as being embedded on the sphere $S$, and thus the curves $\gamma_{i_1},\dots,\gamma_{i_k}$ are also embedded on $S$. The embedding of this component of $G$ onto the sphere $S$ is obtained by considering the vertex associated to $\gamma_{i_j}$ to be the disk with boundary $\gamma_{i_j}$ containing the alternating region incident to $\gamma_{i_j}$. This disk may contain other curves from the alternating decomposition of $D$, but these other curves are associated to a different component of $G$. The edges of this component are the segments of the non-alternating edges of $D$ that go between two curves of the alternating decomposition of $D$. Thus each component of $G$ has an induced embedding onto a sphere.

Thistlethwaite \cite{Thistlethwaite:Breadth} proved that if $G$ is an alternating decomposition graph of some link diagram, then $G$ is planar, bipartite, and each vertex of $G$ has even degree. Our first result of this section is the converse.
\begin{proposition}
\label{prop:altdecompgraph}
Let $G$ be a planar, bipartite graph such that each vertex of $G$ has even degree. Then $G$ is the alternating decomposition graph of some link diagram $D$. Moreover, $D$ may be chosen to be adequate.
\end{proposition}
\begin{proof}
Fix a planar embedding for $G$. For each vertex $v_i$ in $G$, choose an alternating tangle $T_i$ with $\deg v_i$ endpoints along the boundary. Each tangle $T_i$ must contain at least one crossing, and each face of the tangle $T_i$ can only meet the boundary circle in at most one arc. Assign to each endpoint the sign ``$+$'' or ``$-$'' based on whether the strand emanating from that point is the over-strand or the under-strand, respectively, of the first crossing it meets. The signs ``$+$" and ``$-$" will alternate around the boundary of $T_i$. Since $G$ is bipartite, the edges of $G$ can also be assigned ``$+$'' or ``$-$'' in such a way that the signs alternate around each vertex in the planar embedding.  Replace $v_i$ with $T_i$ in the planar embedding of $G$ so that each endpoint of an arc in $T_i$ and the edge of $G$ which it gets connected to have the same sign. This produces a link diagram with the property that the non-alternating arcs exactly correspond to the edges of $G$.

To make the link diagram adequate, appropriate tangles must be chosen for the $T_i$. Choosing the tangles shown in Figure~\ref{figure:adequate} will produce an adequate link diagram. This is because the circles in the all-$A$ and all-$B$ resolutions come in two types: Those completely contained in one of the tangles, and those that pass through multiple tangles. Each crossing is either between two distinct circles of the first type, or between a circle of the first type and a circle of the second type. Specifically, each crossing is always between two distinct circles. Thus if one crossing is changed from the $A$-resolution to the $B$-resolution in the all-$A$ state (or vice-versa in the all-$B$ state), then the number of circles will decrease by one.
\end{proof}
\begin{figure}[h]
$$\begin{tikzpicture}[thick]
\draw (0,0) circle (2cm);
\draw (-2.3,0) node {$+$};
\draw (2.3,0) node{$-$};

\draw (0,-2.5) node{$T_2$};

\begin{scope}[rounded corners = 2mm, xshift = -1.5cm, yshift = -.5cm]
	\draw (-.5,.5) -- (0,0) -- (1,1) -- (1.4,.6);
	\draw (2,1.5) -- (.5,1.5) -- (0,1) -- (.4,.6);
	\draw (.6,.4) -- (1,0) -- (2,1) -- (2.4,.6);
	\draw (1.6,.4) -- (2,0) -- (3,1) -- (2.5,1.5) -- (2,1.5);
	\draw (2.6,.4) -- (3,0) -- (3.5,.5);
\end{scope}

\begin{scope}[xshift = 6cm]
\draw (0,0) circle (2cm);

\draw (-1.8,-1.4) node{$+$};
\draw (-1.8, 1.4) node{$-$};
\draw (1.8, 1.4) node{$+$};
\draw (1.8, -1.4) node{$-$};
\draw (0,-2.5) node {$T_4$};

\begin{scope}[rounded corners = 2mm, xshift = -1.5 cm, yshift = -1.32cm]
	\draw (0,0) -- (1,1) -- (1.4,.6);
	\draw (0,1.64) -- (-.32,1.32) -- (0,1) -- (.4,.6);
	\draw (.6,.4) -- (1,0) -- (2,1) -- (2.4,.6);
	\draw (1.6,.4) -- (2,0) -- (3,1) -- (3.32,1.32) -- (3,1.64);
	\draw (2.6,.4) -- (3,0);
\end{scope}

\begin{scope}[rounded corners = 2mm, xshift = -1.5 cm, yshift = .32cm]
	\draw (0,0) -- (1,1) -- (1.4,.6);
	\draw (0,1) -- (.4,.6);
	\draw (.6,.4) -- (1,0) -- (2,1) -- (2.4,.6);
	\draw (1.6,.4) -- (2,0) -- (3,1);
	\draw (2.6,.4) -- (3,0);
\end{scope}

\end{scope}


\begin{scope}[xshift = 11cm, yshift = -1.73 cm]

\begin{scope}[scale=.685, rounded corners = 1.5mm]

	\draw (0,0) -- (1,1) -- (1.4,.6);
	\draw (0,1) -- (.4,.6);
	\draw (.6,.4) -- (1,0) -- (2,1) -- (2.4,.6);
	\draw (1.6,.4) -- (2,0) -- (3,1);
	\draw (2.6,.4) -- (3,0);
	
	\draw (0,1) -- (-1,2) -- (-.6,2.4);
	\draw (-1,1) -- (-.6,1.4);
	\draw (-.4,1.6) -- (0,2) -- (-1,3) -- (-.6,3.4);
	\draw (-.4,2.6) -- (0,3) -- (-1,4);
	\draw (-.4,3.6) -- (0,4);
	
	\begin{scope}[yshift = 4cm]
		\draw (0,0) -- (1,1) -- (1.4,.6);
	\draw (0,1) -- (.4,.6);
	\draw (.6,.4) -- (1,0) -- (2,1) -- (2.4,.6);
	\draw (1.6,.4) -- (2,0) -- (3,1);
	\draw (2.6,.4) -- (3,0);
	\end{scope}
	
	\fill (3.5,2.5) circle (.08 cm);
	\fill (3.5,3) circle (.08 cm);
	\fill (3.5,2) circle (.08cm);

\draw (1.5,2.5) circle (2.92cm);

\end{scope} 

\draw (-.2,-.4) node{$+$};
\draw (-1,.6) node{$-$};
\draw (-1,2.9) node{$+$};
\draw (-.2,3.7) node{$-$};
\draw (2.2,3.7) node{$+$};
\draw (2.2,-.4) node{$-$};

\end{scope}

\draw (12,-2.5) node{$T_{2k}$};

\end{tikzpicture}$$
\caption{Inserting the tangles $T_{2i}$ into an alternating decomposition graph $G$ results in an adequate link diagram $D$ whose alternating decomposition graph is $G$.}
\label{figure:adequate}
\end{figure}
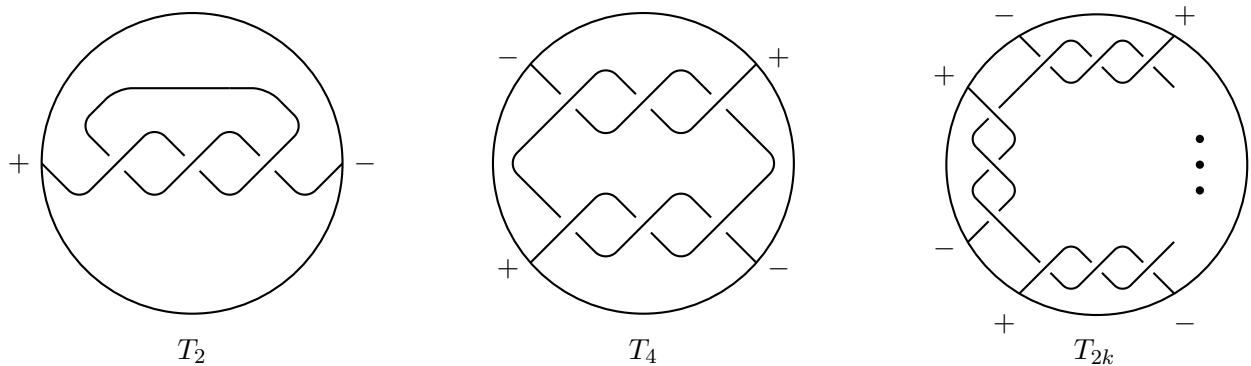

Abe \cite{Abe:Adequate} proves that if $D$ is adequate, then $D$ minimizes Turaev genus, that is $g_T(D)=g_T(L)$. Consequently, we have the following corollary.
\begin{corollary}
Let $G$ be a planar, bipartite graph such that each vertex has even degree. Then there is a link diagram $D$ whose alternating decomposition graph is $G$ such that $g_T(D)=g_T(L)$.
\end{corollary}

An {\em oriented ribbon graph} is a graph $G$ cellularly embedded in an oriented surface $\Sigma$. The genus of an oriented ribbon graph is the genus of $\Sigma$. We often visualize the vertices of an oriented ribbon graph as round disks and the edges of an oriented ribbon graph as rectangular bands attached on opposite ends to the round vertices. The sphere embedding $\mathbb{G}$ of an alternating decomposition graph $G$ is a ribbon graph embedded on a disjoint union of spheres. From $\mathbb{G}$, we construct another ribbon graph $\widetilde{\mathbb{G}}$ such that  the genus of $\widetilde{\mathbb{G}}$ is equal to $g_T(D)$. The ribbon graph $\widetilde{\mathbb{G}}$ has the same vertices and edges as $\mathbb{G}$. To obtain $\widetilde{\mathbb{G}}$ from $\mathbb{G}$ a half-twist is applied to each edge band of $\mathbb{G}$. We say that $\widetilde{\mathbb{G}}$ is the {\em twisted embedding of the alternating decomposition graph $G$}. See Figure \ref{figure:twisted}. The operation of twisting some edges in a ribbon graph has been recently studied by Ellis-Monaghan and Moffatt under the name ``partial petrials" \cite{EMM:RibbonGraphBook}.
\begin{figure}[h]
$$\begin{tikzpicture}
\draw (-2,0) node{$D$};
\draw (.84, .55) arc (34:206:1cm);
\draw (-.84, -.55) arc (214: 386: 1cm);

\begin{scope}[xshift = 6 cm]
	\draw (.84, .55) arc (34:206:1cm);
	\draw (-.84, -.55) arc (214: 386: 1cm);
\end{scope}

\begin{scope}[xshift = 3cm]
	\draw (-.84, .55) arc (146: -26: 1cm);
	\draw (.84, -.55) arc (-34: - 206: 1cm);
\end{scope}

\begin{scope}[xshift = 9cm]
	\draw (-.84, .55) arc (146: -26: 1cm);
	\draw (.84, -.55) arc (-34: - 206: 1cm);
\end{scope}

\draw (-.8,.5) -- (3.8,.5);
\draw (4,.5) -- (5,.5);
\draw (5.2,.5) -- (9.6,.5);
\draw [rounded corners = 3mm] (-1,.5) -- (-1.5,.5) -- (-2,1) -- (-1.5,1.5) -- (10.5,1.5) -- (11,1) -- (10.5,.5) -- (10,.5);
\draw [rounded corners = 3mm] (.8,-.5) -- (-1.5,-.5) -- (-2,-1) -- (-1.5,-1.5) -- (10.5,-1.5) -- (11,-1) -- (10.5,-.5) -- (8.2,-.5);
\draw (1,-.5) -- (2,-.5);
\draw (2.2,-.5) -- (6.8,-.5);
\draw (7,-.5) -- (8, -.5);

\draw[very thick,->] (4.5,-1.75) -- (4.5,-2.75);
\draw[very thick,->, yshift=-4.5cm] (4.5,-1.75) -- (4.5,-2.75);


\begin{scope}[yshift = -4.5cm]
\draw (-2,0) node{$\mathbb{G}$};

\begin{scope}[yshift = -2cm]
\fill[blue!30!white] (-1.25,.4) arc (-90:-270:.6cm);
\fill[blue!30!white] (-1.26,1.6) rectangle (-1.24,1.4);
\fill[blue!30!white] (-1.26,.6) rectangle (-1.24,.4);
\fill[white](-1.25,.6) arc (-90:-270:.4cm);
\fill[white] (-1.26,.6) rectangle (-1.24,1.4);
\draw (-1.25,.6) arc (-90:-270:.4cm);
\draw (-1.25,.4) arc (-90:-270:.6cm);
\end{scope}

\fill[blue!30!white] (-1.25,.4) arc (-90:-270:.6cm);
\fill[blue!30!white] (-1.26,1.6) rectangle (-1.24,1.4);
\fill[blue!30!white] (-1.26,.6) rectangle (-1.24,.4);
\fill[white](-1.25,.6) arc (-90:-270:.4cm);
\fill[white] (-1.26,.6) rectangle (-1.24,1.4);
\draw (-1.25,.6) arc (-90:-270:.4cm);
\draw (-1.25,.4) arc (-90:-270:.6cm);

\fill[blue!30!white] (10.25,.4) arc (-90:90:.6cm);
\fill[blue!30!white] (10.24,.4) rectangle (10.26,.6);
\fill[blue!30!white] (10.24,1.4) rectangle (10.26,1.6);
\fill[white](10.25,.6) arc (-90:90:.4cm);
\fill[white] (10.24,.6) rectangle (10.26,1.4);
\draw (10.25,.4) arc (-90:90:.6cm);
\draw (10.24,.6) arc (-90:90:.4cm);

\begin{scope}[yshift=-2cm]
\fill[blue!30!white] (10.25,.4) arc (-90:90:.6cm);
\fill[blue!30!white] (10.24,.4) rectangle (10.26,.6);
\fill[blue!30!white] (10.24,1.4) rectangle (10.26,1.6);
\fill[white](10.25,.6) arc (-90:90:.4cm);
\fill[white] (10.24,.6) rectangle (10.26,1.4);
\draw (10.25,.4) arc (-90:90:.6cm);
\draw (10.24,.6) arc (-90:90:.4cm);
\end{scope}

\fill[blue!30!white] (-1.25,1.6) -- (10.25,1.6) -- (10.25, 1.4) -- (-1.25,1.4);
\draw (10.25,1.6) -- (-1.25,1.6);
\draw (10.25,1.4) -- (-1.25,1.4);

\begin{scope}[yshift=-3cm]
\fill[blue!30!white] (-1.25,1.6) -- (10.25,1.6) -- (10.25, 1.4) -- (-1.25,1.4);
\draw (10.25,1.6) -- (-1.25,1.6);
\draw (10.25,1.4) -- (-1.25,1.4);

\end{scope}

\fill[blue!30!white] (-1.25,.4) -- (-1.25,.6) -- (10.25,.6) -- (10.25,.4);
\fill[blue!30!white] (-1.25,-.4) -- (-1.25,-.6) -- (10.25,-.6) -- (10.25,-.4);
\draw (-1.25,.4) -- (10.25,.4);
\draw (-1.25,.6) -- (10.25,.6);
\draw (-1.25,-.4) -- (10.25,-.4);
\draw (-1.25,-.6) -- (10.25,-.6);

\fill[blue!30!white] (0,0) circle (1.1cm);
\draw (0,0) circle (1.1cm);

\begin{scope}[xshift = 6 cm]
\fill[blue!30!white] (0,0) circle (1.1cm);
\draw (0,0) circle (1.1cm);
\end{scope}

\begin{scope}[xshift = 3cm]
\fill[blue!30!white] (0,0) circle (1.1cm);
\draw (0,0) circle (1.1cm);
\end{scope}

\begin{scope}[xshift = 9cm]
\fill[blue!30!white] (0,0) circle (1.1cm);
\draw (0,0) circle (1.1cm);
\end{scope}
\end{scope}


\begin{scope}[yshift = -9cm]
\draw (-2,0) node{$\widetilde{\mathbb{G}}$};
\begin{scope}[yshift = -2cm]
\fill[blue!30!white] (-1.25,.4) arc (-90:-270:.6cm);
\fill[blue!30!white] (-1.26,1.6) rectangle (-1.24,1.4);
\fill[blue!30!white] (-1.26,.6) rectangle (-1.24,.4);
\fill[white](-1.25,.6) arc (-90:-270:.4cm);
\fill[white] (-1.26,.6) rectangle (-1.24,1.4);
\draw (-1.25,.6) arc (-90:-270:.4cm);
\draw (-1.25,.4) arc (-90:-270:.6cm);
\end{scope}

\fill[blue!30!white] (-1.25,.4) arc (-90:-270:.6cm);
\fill[blue!30!white] (-1.26,1.6) rectangle (-1.24,1.4);
\fill[blue!30!white] (-1.26,.6) rectangle (-1.24,.4);
\fill[white](-1.25,.6) arc (-90:-270:.4cm);
\fill[white] (-1.26,.6) rectangle (-1.24,1.4);
\draw (-1.25,.6) arc (-90:-270:.4cm);
\draw (-1.25,.4) arc (-90:-270:.6cm);

\fill[blue!90!red] (10.25,.4) arc (-90:90:.6cm);
\fill[blue!90!red] (10.24,.4) rectangle (10.26,.6);
\fill[blue!90!red] (10.24,1.4) rectangle (10.26,1.6);
\fill[white](10.25,.6) arc (-90:90:.4cm);
\fill[white] (10.24,.6) rectangle (10.26,1.4);
\draw (10.25,.4) arc (-90:90:.6cm);
\draw (10.24,.6) arc (-90:90:.4cm);

\begin{scope}[yshift=-2cm]
\fill[blue!90!red] (10.25,.4) arc (-90:90:.6cm);
\fill[blue!90!red] (10.24,.4) rectangle (10.26,.6);
\fill[blue!90!red] (10.24,1.4) rectangle (10.26,1.6);
\fill[white](10.25,.6) arc (-90:90:.4cm);
\fill[white] (10.24,.6) rectangle (10.26,1.4);
\draw (10.25,.4) arc (-90:90:.6cm);
\draw (10.24,.6) arc (-90:90:.4cm);
\end{scope}

\fill[blue!30!white] (-1.25,1.6) -- (4.2,1.6) -- (4.5,1.5) -- (4.2,1.4) -- (-1.25,1.4);
\fill[blue!90!red] (10.25,1.6) -- (4.8,1.6) -- (4.5,1.5) -- (4.8,1.4) -- (10.25,1.4);
\draw (10.25,1.6) -- (4.8,1.6) -- (4.2,1.4) -- (-1.25,1.4);
\draw (10.25,1.4) -- (4.8,1.4) -- (4.2,1.6) -- (-1.25, 1.6);

\begin{scope}[yshift=-3cm]
\fill[blue!30!white] (-1.25,1.6) -- (4.2,1.6) -- (4.5,1.5) -- (4.2,1.4) -- (-1.25,1.4);
\fill[blue!90!red] (10.25,1.6) -- (4.8,1.6) -- (4.5,1.5) -- (4.8,1.4) -- (10.25,1.4);
\draw (10.25,1.6) -- (4.8,1.6) -- (4.2,1.4) -- (-1.25,1.4);
\draw (10.25,1.4) -- (4.8,1.4) -- (4.2,1.6) -- (-1.25, 1.6);

\end{scope}

\fill[blue!30!white] (-1.25,.4) -- (-1.25,.6) -- (1.2,.6) -- (1.5,.5) -- (1.2,.4) -- (-1.25,.4);
\fill[blue!30!white, xshift=6cm] (0,.4) -- (0,.6) -- (1.2,.6) -- (1.5,.5) -- (1.2,.4) -- (0,.4);
\fill[blue!30!white] (6,.6) -- (6,.4) -- (4.8,.4) -- (4.5,.5) -- (4.8,.6) -- (6,.6);
\fill[blue!90!red] (1.5,.5) -- (1.8,.6) -- (4.2,.6) -- (4.5,.5) -- (4.2,.4) -- (1.8,.4) -- (1.5,.5);
\fill[blue!90!red] (7.5,.5) -- (7.8,.6) -- (10.25,.6) -- (10.25,.4) -- (7.8,.4);
\draw (-1.25,.6) -- (1.2,.6) -- (1.8,.4) -- (4.2,.4) -- (4.8,.6) -- (7.2,.6) -- (7.8,.4) -- (10.25,.4);
\draw (-1.25,.4) -- (1.2,.4) -- (1.8,.6) -- (4.2,.6) -- (4.8,.4) -- (7.2,.4) -- (7.8,.6) -- (10.25,.6);

\begin{scope}[yshift = -1cm]

\fill[blue!30!white] (6,.6) -- (6,.4) -- (4.8,.4) -- (4.5,.5) -- (4.8,.6) -- (6,.6);
\fill[blue!30!white] (-1.25,.4) -- (-1.25,.6) -- (1.2,.6) -- (1.5,.5) -- (1.2,.4) -- (-1.25,.4);
\fill[blue!30!white, xshift=6cm] (0,.4) -- (0,.6) -- (1.2,.6) -- (1.5,.5) -- (1.2,.4) -- (0,.4);
\fill[blue!90!red] (1.5,.5) -- (1.8,.6) -- (4.2,.6) -- (4.5,.5) -- (4.2,.4) -- (1.8,.4) -- (1.5,.5);
\fill[blue!90!red] (7.5,.5) -- (7.8,.6) -- (10.25,.6) -- (10.25,.4) -- (7.8,.4);
\draw (-1.25,.6) -- (1.2,.6) -- (1.8,.4) -- (4.2,.4) -- (4.8,.6) -- (7.2,.6) -- (7.8,.4) -- (10.25,.4);
\draw (-1.25,.4) -- (1.2,.4) -- (1.8,.6) -- (4.2,.6) -- (4.8,.4) -- (7.2,.4) -- (7.8,.6) -- (10.25,.6);
\end{scope}

\fill[blue!30!white] (0,0) circle (1.1cm);
\draw (0,0) circle (1.1cm);

\begin{scope}[xshift = 6 cm]
\fill[blue!30!white] (0,0) circle (1.1cm);
\draw (0,0) circle (1.1cm);
\end{scope}

\begin{scope}[xshift = 3cm]
\fill[blue!90!red] (0,0) circle (1.1cm);
\draw (0,0) circle (1.1cm);
\end{scope}

\begin{scope}[xshift = 9cm]
\fill[blue!90!red] (0,0) circle (1.1cm);
\draw (0,0) circle (1.1cm);
\end{scope}
\end{scope}

\end{tikzpicture}$$
\caption{The link diagram $D$, the sphere embedding $\mathbb{G}$ of its alternating decomposition graph $G$, and the twisted embedding $\widetilde{\mathbb{G}}$ of $G$.}
\label{figure:twisted}
\end{figure}
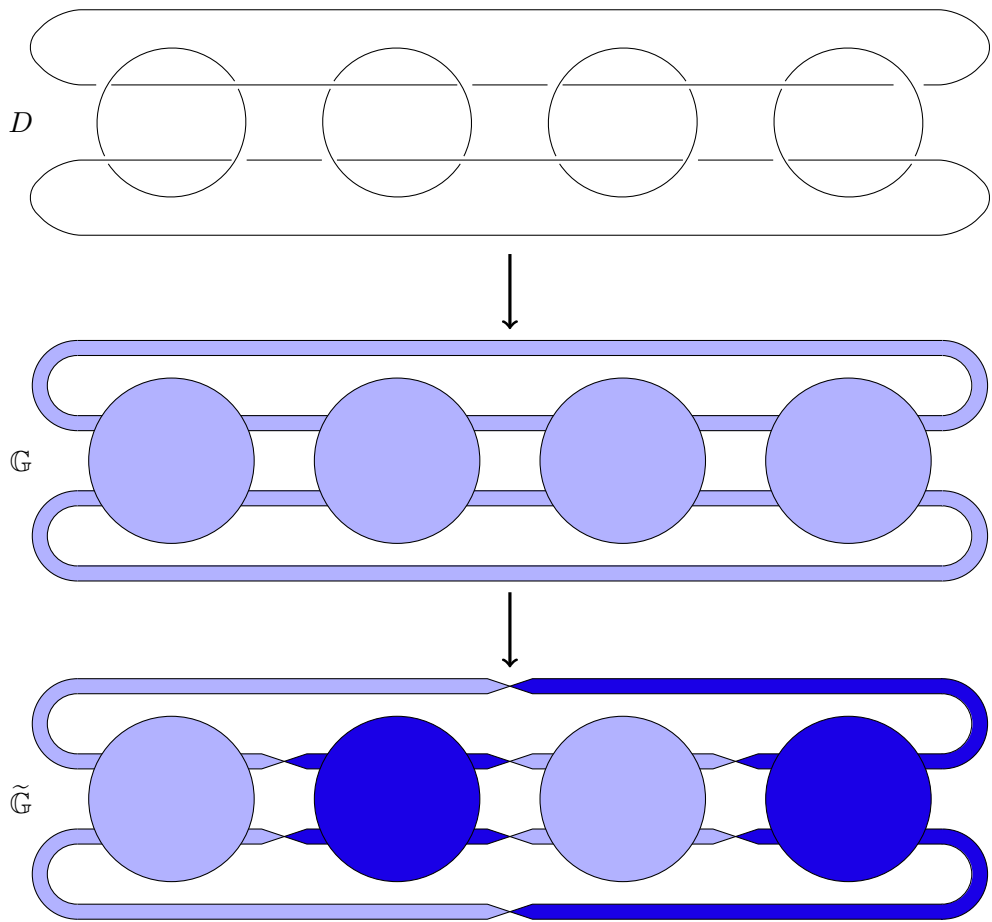

\begin{proposition}
\label{prop:twisted}
Let $\widetilde{\mathbb{G}}$ be the twisted embedding of the alternating decomposition graph of a link diagram $D$. The genus of $\widetilde{\mathbb{G}}$ is $g_T(D)$.
\end{proposition}
\begin{proof}
Each vertex in $\widetilde{\mathbb{G}}$ corresponds to a curve in the alternating decomposition of $D$. Suppose a collection of curves $\{\gamma_{i_1},\dots,\gamma_{i_j}\}$ bound an alternating region $R$ in the alternating decomposition of $D$, and let $v_{i_1},\dots,v_{i_j}$ be their corresponding vertices in $\widetilde{\mathbb{G}}$. The region $R$ is a surface of genus zero with $j$ boundary components. The vertices $v_{i_1},\dots,v_{i_j}$ all lie in different components $\widetilde{\mathbb{G}}_{i_1},\dots, \widetilde{\mathbb{G}}_{i_j}$ of $\widetilde{\mathbb{G}}$. Consider the vertices $v_{i_1},\dots v_{i_j}$ as disks. Form the connected sum $\widetilde{\mathbb{G}}_{i_1}\#\cdots\#\widetilde{\mathbb{G}}_{i_j}$ by identifying disks inside of vertices $v_{i_1},\dots,v_{i_j}$. What was a collection of $j$ disks is now a single planar surface with $j$ boundary components, just like $R$. Repeat this process for each collection of curves that bound an alternating region to form the surface $\Sigma$.

We partially construct the Turaev surface $F(D)$ as follows. Consider $D$ as embedded on a sphere $S$ sitting inside of $S^3$. Replace crossings of $D$ with round disks, and replace all edges of $D$ with either flat or twisted bands according to whether the edge is alternating or non-alternating. The boundary components of the resulting surface correspond to the union of the all-$A$ and all-$B$ states of $D$. If one such boundary component lies completely in $S$ (i.e. each arc in the component contained in an edge band is contained in a flat edge band), then cap that boundary component off with a disk as follows. If the boundary component corresponds to a component of the all-$B$ state, the interior of the disk should be contained inside $S$, and if the boundary component corresponds to a component of the all-$A$ state, the interior of the disk should be contained outside $S$. The resulting surface is $\Sigma$, and so $g(\widetilde{\mathbb{G}})=g(\Sigma)=g_T(D)$.
\end{proof}

Proposition \ref{prop:twisted} implies that the genus of the Turaev surface of $D$ is determined by the sphere embedding $\mathbb{G}$ of its alternating decomposition graph $G$. Hence we define $g_T(\mathbb{G})$ to be $g_T(D)$ for any diagram $D$ with sphere embedding $\mathbb{G}$ of its alternating decomposition graph $G$. We give a recursive algorithm to compute $g_T(\mathbb{G})$ without referring to the link diagram $D$. Our recurrence depends on the following lemma.

\begin{lemma}
\label{lemma:degree2}
Let $\mathbb{G}$ be a sphere embedding of a connected, alternating decomposition graph $G$, and suppose the number of edges in $G$ is nonzero.
\begin{enumerate}
\item Either $\mathbb{G}$ contains a face bounded by exactly two edges or $\mathbb{G}$ contains at least four vertices of degree two.
\item Either $G$ contains a pair of parallel edges or $G$ contains at least four vertices of degree two.
\end{enumerate}
\end{lemma}
\begin{proof}
The degree of a face is defined to be the number of edges in its boundary. Suppose that $\mathbb{G}$ has no face of degree two and three or fewer vertices of degree two. Since every vertex in $\mathbb{G}$ has even degree, it follows that the other vertices of $\mathbb{G}$ have degree at least four. Let $v(\mathbb{G})$, $e(\mathbb{G})$, and $f(\mathbb{G})$ denote the number of vertices, edges, and faces of $\mathbb{G}$ respectively. Also, let $\mathcal{V}(\mathbb{G})$ and $\mathcal{F}(\mathbb{G})$ be the vertex and face sets of $\mathbb{G}$. The handshaking lemma implies
$$4(v(\mathbb{G})-3) + 6 = 4v(\mathbb{G}) - 6\leq \sum_{v\in \mathcal{V}(\mathbb{G})} \deg v = 2e(\mathbb{G}).$$
Thus $v(\mathbb{G})\leq \frac{1}{2} e(\mathbb{G}) + \frac{3}{2}$. Since $\mathbb{G}$ is bipartite, all of its faces have even degree, and since $\mathbb{G}$ has no face of degree two, the handshaking lemma applied to the planar dual of $\mathbb{G}$ implies 
$$4f(\mathbb{G}) \leq \sum_{f\in \mathcal{F}(\mathbb{G})} \deg f = 2e(\mathbb{G}).$$
Thus $f(\mathbb{G})\leq \frac{1}{2} e(\mathbb{G}).$ Now since $\mathbb{G}$ is connected and planar, its Euler characteristic is two. Therefore, we have
$$2=v(\mathbb{G}) - e(\mathbb{G}) + f(\mathbb{G}) \leq \frac{1}{2} e(\mathbb{G}) +\frac{3}{2} - e(\mathbb{G}) + \frac{1}{2} e(\mathbb{G}) = \frac{3}{2},$$
which is a contradiction. Therefore $\mathbb{G}$ must have at least four vertices of degree two. The second statement follows immediately from the first.
\end{proof}

For any graph $\Gamma$ (or oriented ribbon graph), let $k(\Gamma)$ denote the number of connected components in $\Gamma$. If $e$ is an edge in $\Gamma$ incident to vertices $v_1$ and $v_2$, then the contraction of $e$, denoted $\Gamma / e$ is the graph obtained by identifying the vertices $v_1$ and $v_2$ and deleting the edge $e$. Any graph that can be obtained from $\Gamma$ via a sequence of edge contractions and edge or vertex deletions is called a {\em minor} of $\Gamma$. The sphere embedding of a graph induces a sphere embedding on any of its minors. If $\Gamma$ is bipartite, then $\Gamma  - e$ is also bipartite. If $\Gamma$ is bipartite and $k(\Gamma) = k(\Gamma-e) -1$, then $\Gamma / e$ is also bipartite. In the following proposition, whenever a set of edges is deleted or contracted, the induced sphere embedding on the subgraph is assumed. Proposition \ref{prop:PlaneRecursive} gives a recursive algorithm to compute $g_T(\mathbb{G})$.
\begin{proposition}
\label{prop:PlaneRecursive}
Let $\mathbb{G}$ be a sphere embedding of an alternating decomposition graph $G$.
\begin{enumerate}
\item If $\mathbb{G}$ is a collection of isolated vertices, then $g_T(\mathbb{G})=0$.
\item Suppose that $\mathbb{G}$ contains a face bounded by exactly two edges $e_1$ and $e_2$. Let $\mathbb{G}' = \mathbb{G} - \{e_1,e_2\}$, and let $\mathbb{G}''= \mathbb{G}/\{e_1,e_2\}$. If $k(\mathbb{G}') = k(\mathbb{G})$, then $\mathbb{G}'$ is a sphere embedding of an alternating decomposition graph and $g_T(\mathbb{G}') = g_T(\mathbb{G})-1$. If $k(\mathbb{G}') = k(\mathbb{G})+1$, then both $\mathbb{G}'$ and $\mathbb{G}''$ are sphere embeddings of alternating decomposition graphs and $g_T(\mathbb{G}') = g_T(\mathbb{G}'') = g_T(\mathbb{G})$.
\item Suppose that $\mathbb{G}$ contains a vertex $v$ of degree two,  incident to edges $e_1$ and $e_2$. Let $\mathbb{G}'=\mathbb{G}/\{e_1,e_2\}$. Then $\mathbb{G}'$ is a sphere embedding of an alternating decomposition graph, and $g_T(\mathbb{G}') = g_T(\mathbb{G})$.
\end{enumerate}
\end{proposition}
\begin{proof}
(1) Let $D$ be the disjoint union of $m$ alternating diagrams. Then $g_T(D)=0$ and $\mathbb{G}$ is $m$ isolated vertices. Thus $g_T(\mathbb{G})=0$.

(2) Deleting or contracting two edges from a graph embedded on a disjoint union of spheres results in a graph embedded on a disjoint union of spheres. Moreover, since $e_1$ and $e_2$ bound a face, they are incident to the same two vertices. Hence all vertices of $\mathbb{G}'$ and $\mathbb{G}''$ have even degree. Since $\mathbb{G}'$  is obtained from $\mathbb{G}$ by deleting two edges, it follows that $\mathbb{G}'$ is bipartite. Also, since $e_1$ and $e_2$ are parallel, it follows that if the deletion of $e_1$ and $e_2$ increases the number of components in $\mathbb{G}$, then $\mathbb{G}''$ is bipartite. Thus $\mathbb{G}'$ is a sphere embedding of an alternating decomposition graph, and if $k(\mathbb{G}') = k(\mathbb{G})+1$, then $\mathbb{G}''$ is a sphere embedding of an alternating decomposition graph. 

Let $\widetilde{\mathbb{G}}$, $\widetilde{\mathbb{G}}'$, and $\widetilde{\mathbb{G}}''$ be the twisted embeddings of $\mathbb{G}$, $\mathbb{G}'$, and $\mathbb{G}''$ respectively. Define $f(\widetilde{\mathbb{G}})$ to be the number of components of $\Sigma - \widetilde{\mathbb{G}}$ where $\Sigma$ is the surface on which $\mathbb{G}$ is embedded. Note that $f(\widetilde{\mathbb{G}})$ is also the number of boundary components of $\widetilde{\mathbb{G}}$. Similarly define $f(\widetilde{\mathbb{G}}')$ and $f(\widetilde{\mathbb{G}}'')$.

We have $v(\widetilde{\mathbb{G}}')=v(\widetilde{\mathbb{G}})$, $e(\widetilde{\mathbb{G}}') = e(\widetilde{\mathbb{G}})-2$, and $f(\widetilde{\mathbb{G}}') = f(\widetilde{\mathbb{G}})$. If $\mathbb{H}$ is an oriented ribbon graph, then its genus is
$$g(\mathbb{H}) =\frac{1}{2}(2k(\mathbb{H}) -v(\mathbb{H}) + e(\mathbb{H}) - f(\mathbb{H})).$$
Both $\mathbb{G}$ and $\widetilde{\mathbb{G}}$ have the same underlying graph $G$, and so they have the same number of components. A similar statement holds for $\mathbb{G}'$ and $\widetilde{\mathbb{G}}'$. If $k(\mathbb{G}') = k(\mathbb{G}) +1$, then $$g_T(\mathbb{G}') = g(\widetilde{\mathbb{G}}') = g(\widetilde{\mathbb{G}})=g_T(\mathbb{G}),$$ and
if $k(\mathbb{G}') = k(\mathbb{G})$, then
$$g_T(\mathbb{G}') = g(\widetilde{\mathbb{G}}') = g(\widetilde{\mathbb{G}})-1=g_T(\mathbb{G})-1.$$
Also, if $k(\mathbb{G}') = k(\mathbb{G})+1$, then $\widetilde{\mathbb{G}}''$ can be obtained from $\widetilde{\mathbb{G}}'$ by taking a connected sum along the two vertices incident with $e_1$ and $e_2$ in $\widetilde{\mathbb{G}}$. Hence $g_T(\mathbb{G}'') = g_T(\mathbb{G}')$.

(3) As in the previous case, contracting two edges from a graph embedded on a disjoint union of spheres leads to a graph embedded on a disjoint union of spheres. Let $v_1$ and $v_2$ be the two vertices adjacent to $v$, and let $v_{12}$ be the vertex in $\mathbb{G}'$ corresponding to vertices $v_1$ and $v_2$ in $\mathbb{G}$.  If $v_1\neq v_2$, then the degree of $v_{12}$ is $\deg v_1 + \deg v_2 - 2$, which is even. If $v_1=v_2$, then $\deg v_{12} = \deg v_1 - 2$, which is also even. All other vertices in $\mathbb{G}'$ have the same degree as their corresponding vertices in $\mathbb{G}$. Also, the bipartition of the vertices of $\mathbb{G}$ induces a bipartition of the vertices of $\mathbb{G}'$. Thus $\mathbb{G}'$ is a sphere embedding of an alternating decomposition graph.

 Let $\widetilde{\mathbb{G}}$ and $\widetilde{\mathbb{G}}'$ be the twisted embeddings associated to $\mathbb{G}$ and $\mathbb{G}'$, respectively. Then $k(\widetilde{\mathbb{G'}})=k(\widetilde{\mathbb{G}})$ and $e(\widetilde{\mathbb{G}}')=e(\widetilde{\mathbb{G}})-2$. If $v_1\neq v_2$, then
  $v(\widetilde{\mathbb{G'}})= v(\widetilde{\mathbb{G}})-2$ and$f(\widetilde{\mathbb{G}}')=f(\widetilde{\mathbb{G}})$, and if $v_1=v_2$, then 
  $v(\widetilde{\mathbb{G'}})= v(\widetilde{\mathbb{G}})-1$ and 
  $f(\widetilde{\mathbb{G}}') = f(\widetilde{\mathbb{G}})-1$. Hence
$ g_T(\mathbb{G}') =  g_T(\mathbb{G}).$
\end{proof}

As the following theorem shows, the Turaev genus of the sphere embedding $\mathbb{G}$ of the alternating decomposition graph $G$ does not depend on its embedding at all.
\begin{theorem}
\label{thm:EmbeddingIndependent}
Let $\mathbb{G}_1$ and $\mathbb{G}_2$ be sphere embeddings of the same alternating decomposition graph $G$. Then $g_T(\mathbb{G}_1) = g_T(\mathbb{G}_2)$.
\end{theorem}
\begin{proof} We proceed by induction on the number of edges in $G$. If $G$ has no edges, then both $\mathbb{G}_1$ and $\mathbb{G}_2$ are embeddings of a disjoint union of vertices. Hence $g_T(\mathbb{G}_1) = g_T(\mathbb{G}_2) = 0$. 

Suppose that $G$ has $n$ edges and that any two embeddings of an alternating decomposition graph with fewer than $n$ edges have the same Turaev genus. Suppose that $\mathbb{G}_1$ has a vertex $v$ of degree two incident to edges $e_1$ and $e_2$. Since $\mathbb{G}_2$ has the same underlying graph $G$ as $\mathbb{G}_1$, the same statement holds for $\mathbb{G}_2$, that is the vertex $v$ in $\mathbb{G}_2$ has degree two and is incident to edges $e_1$ and $e_2$. Set $\mathbb{G}_1' = \mathbb{G}_1/\{e_1,e_2\}$, $\mathbb{G}_2' =\mathbb{G}_2' / \{e_1,e_2\}$, and $G'=G/\{e_1,e_2\}$. By Proposition \ref{prop:PlaneRecursive}, we have that $g_T(\mathbb{G}_1') = g_T(\mathbb{G}_1)$ and $g_T(\mathbb{G}_2') = g_T(\mathbb{G}_2)$. Since $\mathbb{G}_1'$ and $\mathbb{G}_2'$ are sphere embeddings of the same graph $G'$, the inductive hypothesis implies that $g_T(\mathbb{G}_1') = g_T(\mathbb{G}_2')$. Therefore $g_T(\mathbb{G}_1) = g_T(\mathbb{G}_2)$.

Now suppose that $\mathbb{G}_1$ does not have a vertex of degree two. By Lemma \ref{lemma:degree2}, $\mathbb{G}_1$ has a face bounded by exactly two edges, say $e_1$ and $e_2$. Let $\mathbb{G}_1' = \mathbb{G}_1 -\{e_1,e_2\}$. Then Proposition \ref{prop:PlaneRecursive} implies that if $k(\mathbb{G}_1')=k(\mathbb{G}_1)$, then $g_T(\mathbb{G}_1) = g_T(\mathbb{G}_1')+1$, and if $k(\mathbb{G}_1') = k(\mathbb{G}_1) + 1$, then $g_T(\mathbb{G}_1) = g_T(\mathbb{G}_1')$. Since $\mathbb{G}_1$ and $\mathbb{G}_2$ have the same underlying graph $G$, the edges $e_1$ and $e_2$ are parallel in $\mathbb{G}_2$, but do not necessarily bound a face of degree two. Let $\mathbb{G}_2' = \mathbb{G}_2 - \{e_1,e_2\}$. 

The twisted embedding $\widetilde{\mathbb{G}}_2$ is obtained from $\widetilde{\mathbb{G}}_2'$ by adding the two twisted edges corresponding to $e_1$ and $e_2$. The twisted edges $e_1$ and $e_2$ contain four boundary arcs that are pieces of boundary components of $\widetilde{\mathbb{G}}_2$. Fix one of the boundary arcs and fix an endpoint of that boundary arc. As one travels along the boundary of $\widetilde{\mathbb{G}}_2$ starting from the fixed endpoint, one of the other seven endpoints of boundary arcs of $e_1$ and $e_2$ must be encountered first. The planarity of $\mathbb{G}_2$ lets us rule out four of those endpoints. Furthermore, each edge in $\mathbb{G}_2$ corresponds to a non-alternating edge in some link diagram $D$. The two boundary arcs of that edge correspond to a segment in a component of the all-$A$ state of $D$ and a segment in a component of the all-$B$ state of $D$. In particular, two boundary arcs of the same edge must belong to different components of the boundary of the twisted embedding of the associated alternating decomposition graph. This rules out one more of the endpoints as being the next endpoint encountered. There are two remaining cases, each depicted in Figure \ref{figure:parallel}.

The four boundary arcs of $e_1$ and $e_2$ lie in exactly two components of the boundary of $\widetilde{\mathbb{G}}_2$. Moreover, if the twisted edges $e_1$ and $e_2$ are removed, then the two boundary components containing boundary arcs of $e_1$ and $e_2$ are transformed into two boundary components of the twisted embedding $\widetilde{\mathbb{G}}_2'$. Since no other boundary components of $\widetilde{\mathbb{G}}_2$ are changed by deleting $e_1$ and $e_2$, it follows that $f(\widetilde{\mathbb{G}}_2') = f(\widetilde{\mathbb{G}}_2)$. Since $v(\widetilde{\mathbb{G}}_2')=v(\widetilde{\mathbb{G}}_2)$ and $e(\widetilde{\mathbb{G}}_2') = e(\widetilde{\mathbb{G}}_2) - 2$, it follows that if $k(\mathbb{G}_2') = k(\mathbb{G}_2)$, then $g_T(\mathbb{G}_2)=g_T(\mathbb{G}_2')+1$, and if $k(\mathbb{G}_2') = k(\mathbb{G}_2)+1$, then $g_T(\mathbb{G}_2)=g_T(\mathbb{G}_2')$. The embedded graphs $\mathbb{G}_1'$ and $\mathbb{G}_2'$ have the same underlying graph, and hence the inductive hypothesis implies that $g_T(\mathbb{G}_1')=g_T(\mathbb{G}_2')$.
Deleting $e_1$ and $e_2$ from $\mathbb{G}_1$ increases the number of components if and only if deleting $e_1$ and $e_2$ from $\mathbb{G}_2$ increases the number of components. Therefore $g_T(\mathbb{G}_1) = g_T(\mathbb{G}_2)$, and the desired result is proven.
\end{proof}
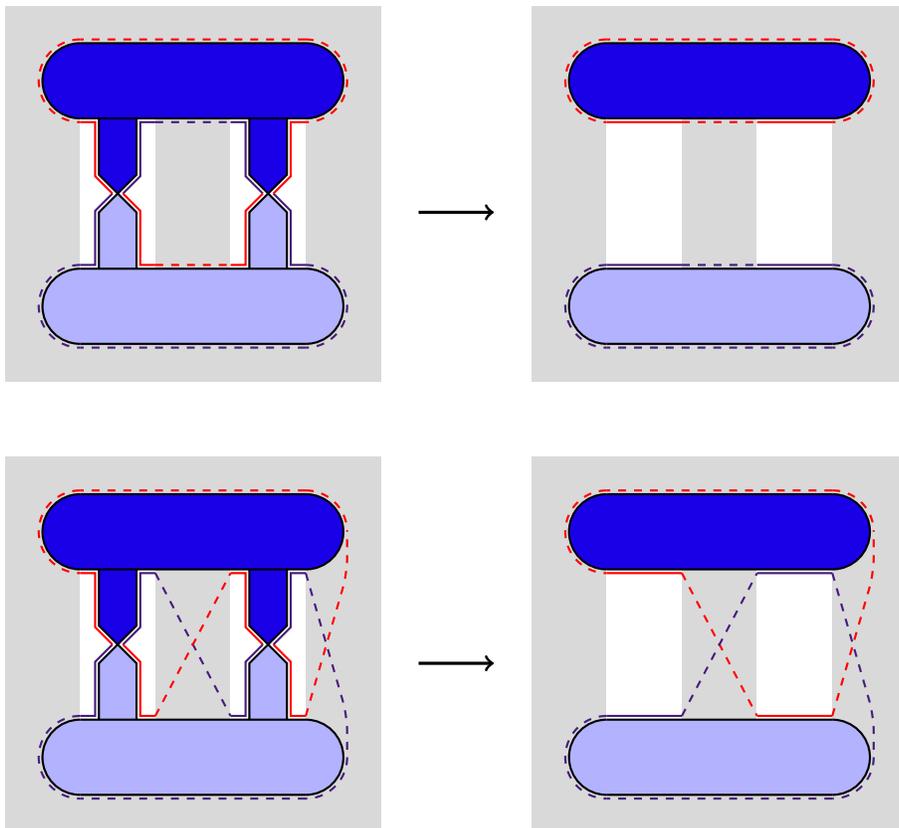
\begin{figure}[h]
$$\begin{tikzpicture}[thick]


\fill [black!15!white] (0,-.5) rectangle (5,4.5);
\fill [white] (1,1) rectangle (2,3);
\fill [white] (4,1) rectangle (3,3);

\fill[blue!30!white] (1,0) arc (270:90:.5cm);
\fill[blue!30!white] (4,0) arc (-90:90:.5cm);
\fill[blue!30!white] (.98,0) rectangle (4.02,1);

\fill [blue!90!red] (1.25,3) -- (1.25,2.25) -- (1.5,2) -- (1.75,2.25) -- (1.75,3);
\fill[blue!30!white] (1.25,1) -- (1.25, 1.75) -- (1.5,2) -- (1.75,1.75) -- (1.75,1);
\draw (1.25,1) -- (1.25,1.75) -- (1.75,2.25) -- (1.75,3);
\draw (1.25,3) -- (1.25, 2.25) -- (1.75, 1.75) -- (1.75,1);

\begin{scope}[xshift = 2cm]
\fill [blue!90!red] (1.25,3) -- (1.25,2.25) -- (1.5,2) -- (1.75,2.25) -- (1.75,3);
\fill[blue!30!white] (1.25,1) -- (1.25, 1.75) -- (1.5,2) -- (1.75,1.75) -- (1.75,1);
\draw (1.25,1) -- (1.25,1.75) -- (1.75,2.25) -- (1.75,3);
\draw (1.25,3) -- (1.25, 2.25) -- (1.75, 1.75) -- (1.75,1);
\end{scope}

\draw (1,0) -- (4,0);
\draw (1,1) -- (4,1);
\draw (1,0) arc (270:90:.5cm);
\draw (4,0) arc (-90:90:.5cm);

\begin{scope}[yshift = 3cm]
\fill[blue!90!red] (1,0) arc (270:90:.5cm);
\fill[blue!90!red] (4,0) arc (-90:90:.5cm);
\fill[blue!90!red] (.98,0) rectangle (4.02,1);
\draw (1,0) -- (4,0);
\draw (1,1) -- (4,1);
\draw (1,0) arc (270:90:.5cm);
\draw (4,0) arc (-90:90:.5cm);
\end{scope}

\draw[lsupurple] (1,1.05) -- (1.2,1.05) -- (1.2,1.77) -- (1.43,2);
\draw[lsupurple] (1.57,2) -- (1.8,2.23) -- (1.8,2.95) -- (2,2.95);
\draw[lsupurple, dashed] (2,2.95) -- (3,2.95);
\draw[lsupurple] (3,2.95) -- (3.2,2.95) -- (3.2, 2.23) -- (3.43,2);
\draw[lsupurple] (3.57,2) -- (3.8,1.77) -- (3.8,1.05) -- (4,1.05);
\draw[lsupurple, dashed] (4,1.05) arc (90:-90: .55cm);
\draw[lsupurple, dashed] (4,-.05) -- (1,-.05);
\draw[lsupurple, dashed] (1,1.05) arc (90:270:.55cm);

\draw[red] (1,2.95) -- (1.2,2.95) -- (1.2, 2.23) -- (1.43,2);
\draw[red] (1.57,2) -- (1.8,1.77) -- (1.8,1.05) -- (2,1.05);
\draw[red] (3,1.05) -- (3.2,1.05) -- (3.2,1.77) -- (3.43,2);
\draw[red] (3.57,2) -- (3.8,2.23) -- (3.8,2.95) -- (4,2.95);
\draw[red, dashed] (2,1.05) -- (3,1.05);
\draw[red, dashed] (4,4.05) -- (1,4.05);
\draw[red, dashed] (4,4.05) arc (90:-90: .55cm);
\draw[red, dashed] (1,4.05) arc (90:270:.55cm);


\begin{scope}[xshift = 7cm]
\fill [black!15!white] (0,-.5) rectangle (5,4.5);
\fill [white] (1,1) rectangle (2,3);
\fill [white] (4,1) rectangle (3,3);

\fill[blue!30!white] (1,0) arc (270:90:.5cm);
\fill[blue!30!white] (4,0) arc (-90:90:.5cm);
\fill[blue!30!white] (.98,0) rectangle (4.02,1);

\draw (1,0) -- (4,0);
\draw (1,1) -- (4,1);
\draw (1,0) arc (270:90:.5cm);
\draw (4,0) arc (-90:90:.5cm);

\begin{scope}[yshift = 3cm]
\fill[blue!90!red] (1,0) arc (270:90:.5cm);
\fill[blue!90!red] (4,0) arc (-90:90:.5cm);
\fill[blue!90!red] (.98,0) rectangle (4.02,1);
\draw (1,0) -- (4,0);
\draw (1,1) -- (4,1);
\draw (1,0) arc (270:90:.5cm);
\draw (4,0) arc (-90:90:.5cm);
\end{scope}

\draw[red, dashed] (4,4.05) -- (1,4.05);
\draw[red, dashed] (4,4.05) arc (90:-90: .55cm);
\draw[red, dashed] (1,4.05) arc (90:270:.55cm);
\draw[red] (1,2.95) -- (2,2.95);
\draw[red, dashed] (2,2.95) -- (3,2.95);
\draw[red] (3,2.95) -- (4,2.95);

\draw[lsupurple, dashed] (4,1.05) arc (90:-90: .55cm);
\draw[lsupurple, dashed] (4,-.05) -- (1,-.05);
\draw[lsupurple, dashed] (1,1.05) arc (90:270:.55cm);
\draw[lsupurple] (1,1.05) -- (2,1.05);
\draw[lsupurple, dashed] (2,1.05) -- (3,1.05);
\draw[lsupurple] (3,1.05) -- (4,1.05);

\end{scope}


\begin{scope}[yshift = -6cm]

\fill [black!15!white] (0,-.5) rectangle (5,4.5);
\fill [white] (1,1) rectangle (2,3);
\fill [white] (4,1) rectangle (3,3);

\fill[blue!30!white] (1,0) arc (270:90:.5cm);
\fill[blue!30!white] (4,0) arc (-90:90:.5cm);
\fill[blue!30!white] (.98,0) rectangle (4.02,1);

\fill [blue!90!red] (1.25,3) -- (1.25,2.25) -- (1.5,2) -- (1.75,2.25) -- (1.75,3);
\fill[blue!30!white] (1.25,1) -- (1.25, 1.75) -- (1.5,2) -- (1.75,1.75) -- (1.75,1);
\draw (1.25,1) -- (1.25,1.75) -- (1.75,2.25) -- (1.75,3);
\draw (1.25,3) -- (1.25, 2.25) -- (1.75, 1.75) -- (1.75,1);

\begin{scope}[xshift = 2cm]
\fill [blue!90!red] (1.25,3) -- (1.25,2.25) -- (1.5,2) -- (1.75,2.25) -- (1.75,3);
\fill[blue!30!white] (1.25,1) -- (1.25, 1.75) -- (1.5,2) -- (1.75,1.75) -- (1.75,1);
\draw (1.25,1) -- (1.25,1.75) -- (1.75,2.25) -- (1.75,3);
\draw (1.25,3) -- (1.25, 2.25) -- (1.75, 1.75) -- (1.75,1);
\end{scope}

\draw (1,0) -- (4,0);
\draw (1,1) -- (4,1);
\draw (1,0) arc (270:90:.5cm);
\draw (4,0) arc (-90:90:.5cm);

\begin{scope}[yshift = 3cm]
\fill[blue!90!red] (1,0) arc (270:90:.5cm);
\fill[blue!90!red] (4,0) arc (-90:90:.5cm);
\fill[blue!90!red] (.98,0) rectangle (4.02,1);
\draw (1,0) -- (4,0);
\draw (1,1) -- (4,1);
\draw (1,0) arc (270:90:.5cm);
\draw (4,0) arc (-90:90:.5cm);
\end{scope}

\draw[lsupurple] (1,1.05) -- (1.2,1.05) -- (1.2,1.77) -- (1.43,2);
\draw[lsupurple] (1.57,2) -- (1.8,2.23) -- (1.8,2.95) -- (2,2.95);
\draw[lsupurple, dashed] (2,2.95) -- (3,1.05);
\draw[red] (3,2.95) -- (3.2,2.95) -- (3.2, 2.23) -- (3.43,2);
\draw[red] (3.57,2) -- (3.8,1.77) -- (3.8,1.05) -- (4,1.05);
\draw[lsupurple, dashed] (4.55,.5) arc (0:-90: .55cm);
\draw[lsupurple, dashed] (4,-.05) -- (1,-.05);
\draw[lsupurple, dashed] (1,1.05) arc (90:270:.55cm);

\draw[red] (1,2.95) -- (1.2,2.95) -- (1.2, 2.23) -- (1.43,2);
\draw[red] (1.57,2) -- (1.8,1.77) -- (1.8,1.05) -- (2,1.05);
\draw[lsupurple] (3,1.05) -- (3.2,1.05) -- (3.2,1.77) -- (3.43,2);
\draw[lsupurple] (3.57,2) -- (3.8,2.23) -- (3.8,2.95) -- (4,2.95);
\draw[red, dashed] (2,1.05) -- (3,2.95);
\draw[red, dashed] (4,4.05) -- (1,4.05);
\draw[red, dashed] (4,4.05) arc (90:0: .55cm);
\draw[red, dashed] (1,4.05) arc (90:270:.55cm);

\draw[red, dashed, rounded corners = 2mm] (4,1.05) -- (4.55,2.95) -- (4.55,3.5);
\draw[lsupurple, dashed, rounded corners = 2mm] (4,2.95) -- (4.55,1.05) -- (4.55,.5);

\end{scope}


\begin{scope}[xshift= 7cm, yshift =-6cm]

\fill [black!15!white] (0,-.5) rectangle (5,4.5);
\fill [white] (1,1) rectangle (2,3);
\fill [white] (4,1) rectangle (3,3);

\fill[blue!30!white] (1,0) arc (270:90:.5cm);
\fill[blue!30!white] (4,0) arc (-90:90:.5cm);
\fill[blue!30!white] (.98,0) rectangle (4.02,1);
\draw (1,0) -- (4,0);
\draw (1,1) -- (4,1);
\draw (1,0) arc (270:90:.5cm);
\draw (4,0) arc (-90:90:.5cm);

\begin{scope}[yshift = 3cm]
\fill[blue!90!red] (1,0) arc (270:90:.5cm);
\fill[blue!90!red] (4,0) arc (-90:90:.5cm);
\fill[blue!90!red] (.98,0) rectangle (4.02,1);
\draw (1,0) -- (4,0);
\draw (1,1) -- (4,1);
\draw (1,0) arc (270:90:.5cm);
\draw (4,0) arc (-90:90:.5cm);
\end{scope}

\draw[lsupurple] (1,1.05) -- (2,1.05);
\draw[red, dashed] (2,2.95) -- (3,1.05);
\draw[lsupurple] (3,2.95) -- (4,2.95);
\draw[red] (3,1.05) -- (4,1.05);
\draw[lsupurple, dashed] (4.55,.5) arc (0:-90: .55cm);
\draw[lsupurple, dashed] (4,-.05) -- (1,-.05);
\draw[lsupurple, dashed] (1,1.05) arc (90:270:.55cm);

\draw[red] (1,2.95) -- (2,2.95);
\draw[lsupurple, dashed] (2,1.05) -- (3,2.95);
\draw[red, dashed] (4,4.05) -- (1,4.05);
\draw[red, dashed] (4,4.05) arc (90:0: .55cm);
\draw[red, dashed] (1,4.05) arc (90:270:.55cm);

\draw[red, dashed, rounded corners = 2mm] (4,1.05) -- (4.55,2.95) -- (4.55,3.5);
\draw[lsupurple, dashed, rounded corners = 2mm] (4,2.95) -- (4.55,1.05) -- (4.55,.5);

\end{scope}

\draw [very thick, ->] (5.5,1.75) -- (6.5,1.75);
\draw [very thick, ->] (5.5, -4.25) -- (6.5,-4.25);

\end{tikzpicture}$$
\caption{The two figures on the left show the boundary components of $\widetilde{\mathbb{G}}_2$ that contain the boundary arcs of $e_1$ and $e_2$, and the two figures on the right show the corresponding boundary components of $\widetilde{\mathbb{G}}_2'$. Other vertices and edges of the graph lie inside the two shaded areas.}
\label{figure:parallel}
\end{figure}

\begin{proof}[Proof of Theorem \ref{thm:graphTuraev}]
Let $D_1$ and $D_2$ be two link diagrams with the same alternating decomposition graph $G$. Let $\mathbb{G}_1$ be the sphere embedding of $G$ induced by $D_1$, and let $\mathbb{G}_2$ be the sphere embedding of $G$ induced by $D_2$. Theorem \ref{thm:EmbeddingIndependent} implies that $g_T(D_1) = g_T(\mathbb{G}_1)=g_T(\mathbb{G}_2)=g_T(D_2)$, as desired.
\end{proof}

Since the Turaev genus of an alternating decomposition graph $G$ does not depend on the sphere embedding of $G$, we can define $g_T(G)$ to be $g_T(D)$ where $D$ is any link diagram with alternating decomposition graph $G$. The recursive algorithm in Proposition \ref{prop:PlaneRecursive} can be restated without reference to embedding.
\begin{corollary}
\label{cor:TuraevAlgorithm}
Let $G$ be an alternating decomposition graph.
\begin{enumerate}
\item If $G$ is a collection of isolated vertices, then $g_T(G)=0$.
\item Suppose that $G$ contains a set of parallel edges $\{e_1,e_2\}$. Let $G'=G-\{e_1,e_2\}$ and let $G''=G/\{e_1,e_2\}$. If $k(G)=k(G')$, then $g_T(G') = g_T(G) - 1$, and if $k(G') = k(G) + 1$, then $g_T(G') = g_T(G'') = g_T(G)$.
\item Suppose that $G$ contains a vertex $v$ of degree two, incident to edges $e_1$ and $e_2$. Let $G' = G/\{e_1,e_2\}$. Then $g_T(G') = g_T(G)$. 
\end{enumerate}
\end{corollary}

\begin{example}
Let $G$ be the alternating decomposition graph on the top left of Figure \ref{figure:AlgExample}. One can apply the algorithm of Corollary \ref{cor:TuraevAlgorithm} to $G$ as follows. First delete four pairs of parallel edges as shown to obtain the graph $G'$. Since $k(G)=k(G')$, it follows that $g_T(G)=g_T(G')+4$. Second, contract the remaining four pairs of parallel edges to obtain $G''$, and note that $g_T(G'') = g_T(G')$. Finally apply operation (3) of Corollary \ref{cor:TuraevAlgorithm} to four degree two vertices of $G''$ to obtain $C_2^2$. Since $g_T(C_2^2)=1$, it follows that $g_T(G)=5$. This example shows that it is not always possible to find $g_T(G)$ pairs of parallel edges in $G$ whose deletion do not increase the number of components.
\end{example}

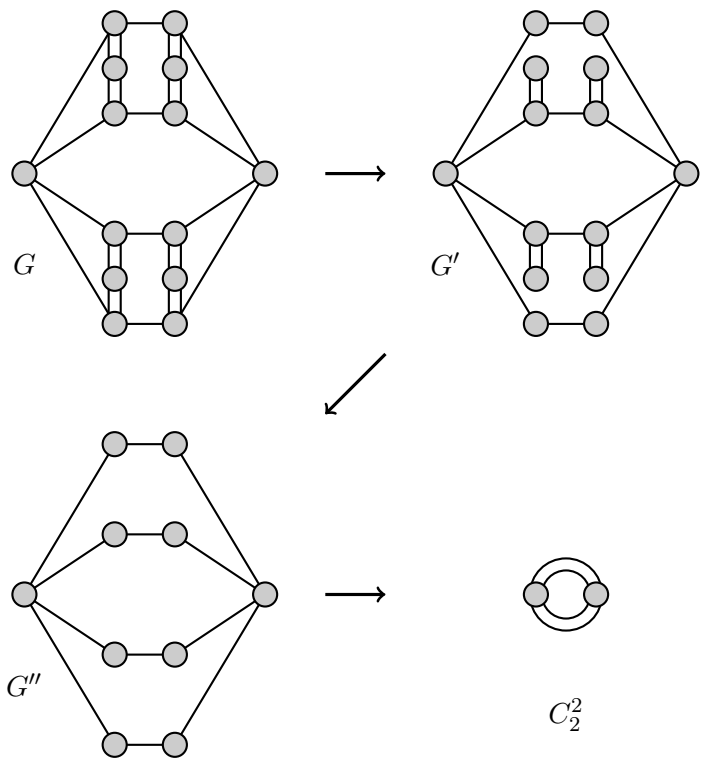
\begin{figure}[h]
$$\begin{tikzpicture}[thick, scale = .8]

\draw (-.5,1) node{$G$};

\draw (-.5,2.5) -- (1,0) -- (2,0) -- (3.5,2.5) -- (2,1.5) -- (1,1.5) -- (-.5,2.5) -- (1,3.5) -- (2,3.5) -- (3.5,2.5) -- (2,5) -- (1,5) -- (-.5,2.5);
\draw (.9,5) -- (.9,4.25);
\draw (1.1,5) -- (1.1,4.25);
\draw (1.9,5) -- (1.9,4.25);
\draw (2.1, 5) -- (2.1,4.25);
\draw (.9,0) -- (.9,.75);
\draw (1.1,0) -- (1.1,.75);
\draw (1.9,0) -- (1.9,.75);
\draw (2.1,0) -- (2.1,.75);

\draw (.9,4.25) -- (.9,3.5);
\draw (1.1, 4.25) -- (1.1, 3.5);
\draw (1.9,4.25) -- (1.9,3.5);
\draw (2.1,4.25) -- (2.1,3.5);
\draw (.9,.75) -- (.9,1.5);
\draw (1.1,.75) -- (1.1,1.5);
\draw (1.9,.75) -- (1.9,1.5);
\draw (2.1,.75) -- (2.1,1.5);

\fill[black!20!white] (-.5,2.5) circle (.2cm);
\fill[black!20!white] (1,0) circle (.2cm);
\fill[black!20!white] (1,.75) circle (.2cm);
\fill[black!20!white] (1,1.5) circle (.2cm);
\fill[black!20!white] (1,3.5) circle (.2cm);
\fill[black!20!white] (1,4.25) circle (.2cm);
\fill[black!20!white] (1,5) circle (.2cm);
\fill[black!20!white] (2,0) circle (.2cm);
\fill[black!20!white] (2,.75) circle (.2cm);
\fill[black!20!white] (2,1.5) circle (.2cm);
\fill[black!20!white] (2,3.5) circle (.2cm);
\fill[black!20!white] (2,4.25) circle (.2cm);
\fill[black!20!white] (2,5) circle (.2cm);
\fill[black!20!white] (3.5,2.5) circle (.2cm);

\draw (-.5,2.5) circle (.2cm);
\draw (1,0) circle (.2cm);
\draw (1,.75) circle (.2cm);
\draw (1,1.5) circle (.2cm);
\draw (1,3.5) circle (.2cm);
\draw (1,4.25) circle (.2cm);
\draw (1,5) circle (.2cm);
\draw (2,0) circle (.2cm);
\draw (2,.75) circle (.2cm);
\draw (2,1.5) circle (.2cm);
\draw (2,3.5) circle (.2cm);
\draw (2,4.25) circle (.2cm);
\draw (2,5) circle (.2cm);
\draw (3.5,2.5) circle (.2cm);

\draw[very thick, ->] (4.5,2.5) -- (5.5,2.5);

\draw[very thick, ->] (5.5,-.5) -- (4.5,-1.5);

\begin{scope}[xshift = 7cm]

\draw (-.5,1) node{$G'$};
\draw (-.5,2.5) -- (1,0) -- (2,0) -- (3.5,2.5) -- (2,1.5) -- (1,1.5) -- (-.5,2.5) -- (1,3.5) -- (2,3.5) -- (3.5,2.5) -- (2,5) -- (1,5) -- (-.5,2.5);

\draw (.9,4.25) -- (.9,3.5);
\draw (1.1, 4.25) -- (1.1, 3.5);
\draw (1.9,4.25) -- (1.9,3.5);
\draw (2.1,4.25) -- (2.1,3.5);
\draw (.9,.75) -- (.9,1.5);
\draw (1.1,.75) -- (1.1,1.5);
\draw (1.9,.75) -- (1.9,1.5);
\draw (2.1,.75) -- (2.1,1.5);

\fill[black!20!white] (-.5,2.5) circle (.2cm);
\fill[black!20!white] (1,0) circle (.2cm);
\fill[black!20!white] (1,.75) circle (.2cm);
\fill[black!20!white] (1,1.5) circle (.2cm);
\fill[black!20!white] (1,3.5) circle (.2cm);
\fill[black!20!white] (1,4.25) circle (.2cm);
\fill[black!20!white] (1,5) circle (.2cm);
\fill[black!20!white] (2,0) circle (.2cm);
\fill[black!20!white] (2,.75) circle (.2cm);
\fill[black!20!white] (2,1.5) circle (.2cm);
\fill[black!20!white] (2,3.5) circle (.2cm);
\fill[black!20!white] (2,4.25) circle (.2cm);
\fill[black!20!white] (2,5) circle (.2cm);
\fill[black!20!white] (3.5,2.5) circle (.2cm);

\draw (-.5,2.5) circle (.2cm);
\draw (1,0) circle (.2cm);
\draw (1,.75) circle (.2cm);
\draw (1,1.5) circle (.2cm);
\draw (1,3.5) circle (.2cm);
\draw (1,4.25) circle (.2cm);
\draw (1,5) circle (.2cm);
\draw (2,0) circle (.2cm);
\draw (2,.75) circle (.2cm);
\draw (2,1.5) circle (.2cm);
\draw (2,3.5) circle (.2cm);
\draw (2,4.25) circle (.2cm);
\draw (2,5) circle (.2cm);
\draw (3.5,2.5) circle (.2cm);
\end{scope}

\begin{scope}[yshift = -7cm]

\draw (-.5,1) node{$G''$};

\draw (-.5,2.5) -- (1,0) -- (2,0) -- (3.5,2.5) -- (2,1.5) -- (1,1.5) -- (-.5,2.5) -- (1,3.5) -- (2,3.5) -- (3.5,2.5) -- (2,5) -- (1,5) -- (-.5,2.5);

\fill[black!20!white] (-.5,2.5) circle (.2cm);
\fill[black!20!white] (1,0) circle (.2cm);
\fill[black!20!white] (1,1.5) circle (.2cm);
\fill[black!20!white] (1,3.5) circle (.2cm);
\fill[black!20!white] (1,5) circle (.2cm);
\fill[black!20!white] (2,0) circle (.2cm);
\fill[black!20!white] (2,1.5) circle (.2cm);
\fill[black!20!white] (2,3.5) circle (.2cm);
\fill[black!20!white] (2,5) circle (.2cm);
\fill[black!20!white] (3.5,2.5) circle (.2cm);

\draw (-.5,2.5) circle (.2cm);
\draw (1,0) circle (.2cm);
\draw (1,1.5) circle (.2cm);
\draw (1,3.5) circle (.2cm);
\draw (1,5) circle (.2cm);
\draw (2,0) circle (.2cm);
\draw (2,1.5) circle (.2cm);
\draw (2,3.5) circle (.2cm);
\draw (2,5) circle (.2cm);
\draw (3.5,2.5) circle (.2cm);

\draw[very thick, ->] (4.5,2.5) -- (5.5,2.5);
\end{scope}

\begin{scope}[xshift = 2cm, yshift = -5.5cm]
\draw (6.5,1) circle (.4 cm);
\draw (6.5,1) circle (.6cm);

\fill[black!20!white] (6,1)circle (.2cm);
\draw (6,1) circle (.2cm);

\fill[black!20!white] (7,1) circle (.2cm);
\draw (7,1) circle (.2cm);

\draw (6.5, -1) node{$C_2^2$};

\end{scope}

\end{tikzpicture}$$
\caption{The graph $G$ is transformed into $C_2^2$ via the algorithm of Corollary \ref{cor:TuraevAlgorithm}. The first step decreases Turaev genus by four, while the second and third steps do not change Turaev genus. Since $g_T(C_2^2)=1$, it follows that $g_T(G)=5$.}
\label{figure:AlgExample}
\end{figure}

\begin{proposition}
\label{prop:doubledpath}
Suppose that $G_1$ and $G_2$ are doubled path equivalent alternating decomposition graphs. Then $g_T(G_1) = g_T(G_2)$.
\end{proposition}
\begin{proof}
Let $G$ be an alternating decomposition graph with sphere embedding $\mathbb{G}$ and twisted embedding $\widetilde{\mathbb{G}}$. A doubled path extension adds one vertex, two edges, and one face to $\widetilde{\mathbb{G}}$, and a doubled path contraction removes one vertex, two edges, and one face from $\widetilde{\mathbb{G}}$. Therefore the Euler characteristic of $\widetilde{\mathbb{G}}$ is unchanged by either doubled path extensions or doubled path contractions. If $G_1$ and $G_2$ are doubled path equivalent alternating decomposition graphs with twisted embeddings $\widetilde{\mathbb{G}}_1$ and $\widetilde{\mathbb{G}}_2$, then the Euler characteristics (and hence genera) of $\widetilde{\mathbb{G}}_1$ and $\widetilde{\mathbb{G}}_2$ agree. Thus $g_T(G_1) = g_T(G_2)$.
\end{proof}

We remind the reader that doubled path extensions and contractions can transform an alternating decomposition graph into a non-bipartite graph whose associated twisted embedding is non-orientable. However, the Euler characteristic argument in the proof of Proposition \ref{prop:doubledpath} applies in both the orientable or non-orientable cases. We also warn the reader that doubled path extensions and contractions only change the length of existing doubled paths. Creating new doubled paths or entirely destroying doubled paths will change the Turaev genus of the graph.

\section{Alternating decomposition graphs of Turaev genus zero}
\label{section:zero}

Turaev \cite{Turaev:SimpleProof} showed that the genus of the Turaev surface of a link diagram $D$ is zero if and only if $D$ is a connected sum of alternating diagrams. In this section, we use Turaev's result to give a classification of alternating decomposition graphs of Turaev genus zero. In order to accomplish this, we will study the behavior of the alternating decomposition graph under certain types of connected sums. 

Suppose that $D$ is a link diagram with $g_T(D)=0$. Hence $D=D_1 \#\cdots \# D_k$ is a connected sum of alternating diagrams $D_1,\dots , D_k$. For $i=1,\dots, k$, let $\widetilde{D}_i = D_1\#\cdots \#D_i.$ Then $D=\widetilde{D}_k$ and $\widetilde{D}_{i+1} = \widetilde{D}_i \# D_{i+1}$. Thus to classify connected sums of alternating diagrams, it suffices to examine the connected sum of a (possibly non-alternating) diagram $\widetilde{D}_i$ and an alternating diagram $D_{i+1}$. See Figure \ref{figure:connectedsum}.

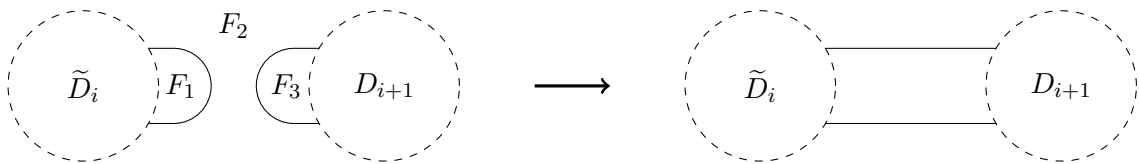
\begin{figure}[h]
$$\begin{tikzpicture}
\begin{scope}[xshift = .2cm]
\draw (.2,.5) -- (1,.5);
\draw (.2, -.5) -- (1,-.5);
\draw (1,.5) arc (90:-90:.5cm);
\end{scope}

\begin{scope}[xshift = -.2cm]

\draw (3.8,.5) -- (3,.5);
\draw (3,-.5) -- (3.8,-.5);
\draw (3,.5) arc (90:270:.5cm);
\end{scope}

\fill[white] (0,0) circle (1cm);
\fill[white] (4,0) circle (1cm);
\draw[dashed] (0,0) circle (1cm);
\draw (0,0) node{$\widetilde{D}_i$};
\draw[dashed] (4,0) circle (1cm);
\draw (4,0) node{$D_{i+1}$};

\draw (1.3,0) node{$F_1$};
\draw (2.7,0) node{$F_3$};
\draw (2,.8) node{$F_2$};

\draw[very thick, ->] (6,0) -- (7,0);

\begin{scope}[xshift=9cm]

	\draw (.6,.5) -- (3.8,.5);
	\draw (.5,.3) -- (.5,.7);
	\draw (.4,.5) -- (.2,.5);
	\draw (3.5,.3) -- (3.5,.4);
	\draw (3.5,.6) -- (3.5,.7);

	\draw (.2, -.5) -- (3.4,-.5);
	\draw (3.5, -.7) -- (3.5,-.3);
	\draw (3.6,-.5) -- (3.8,-.5);
	\draw (.5,-.3) -- (.5,-.4);
	\draw (.5,-.6) -- (.5,-.7);

	\fill[white] (0,0) circle (1cm);
	\fill[white](4,0) circle (1cm);
	\draw[dashed] (0,0) circle (1cm);
	\draw (0,0) node{$\widetilde{D}_i$};
	\draw[dashed] (4,0) circle (1cm);
	\draw (4,0) node{$D_{i+1}$};

\end{scope}

\end{tikzpicture}$$
\caption{On the left is the disjoint union of $\widetilde{D}_i$ and $D_{i+1}$, and on the right is a connected sum of $\widetilde{D}_i$ and $D_{i+1}$. The diagram $D_{i+1}$ is alternating. For $k=1$, $2$, and $3$, let $F_k$ denote the indicated face of $\widetilde{D}_k\sqcup D_{k+1}$.}
\label{figure:connectedsum}
\end{figure}

Let $\widetilde{G}_{i}$ be the alternating decomposition graph of $\widetilde{D}_i$, for each $i=1,\dots, k$. Since $D_{i+1}$ is alternating, its alternating decomposition graph is a single vertex. We examine how $\widetilde{G}_{i+1}$ is obtained from $\widetilde{G}_i$. A face of a link diagram is said to be {\em alternating} if every edge in the boundary of that face is alternating. Otherwise, the face is said to be {\em non-alternating}. Let $e_i$ be the edge of $\widetilde{D}_i$ and let $e_{i+1}$ be the edge of $D_{i+1}$ along which we are taking the connected sum. The edge $e_{i+1}$ is necessarily alternating, but $e_{i}$ can be either alternating or non-alternating. Figure \ref{figure:altdecompsum} shows the alternating decomposition curves in the seven relevant cases, which we describe in detail below.

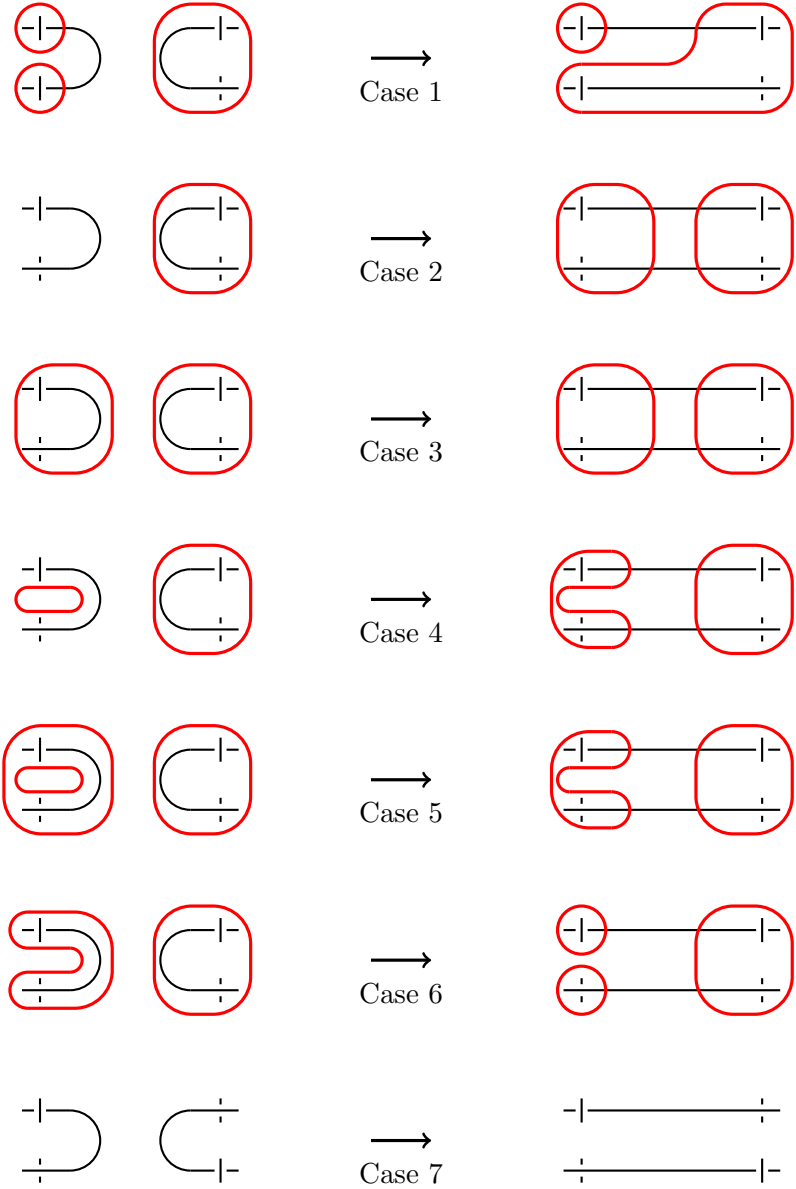
\begin{figure}[h]
$$\begin{tikzpicture}[thick,scale=.8]]

\begin{scope}[yshift = -3 cm, thick]

	\draw (.6,.5) -- (1,.5);
	\draw (3.4,.5) -- (3,.5);
	\draw (3.6,.5) -- (3.8,.5);
	\draw (1,.5) arc (90:-90:.5cm);
	\draw (.5,.3) -- (.5,.7);
	\draw (.4,.5) -- (.2,.5);
	\draw (3.5,.3) -- (3.5,.7);

	\draw (.6, -.5) -- (1,-.5);
	\draw (.2, -.5) -- (.4, -.5);
	\draw (3,-.5) -- (3.8,-.5);
	\draw (3,.5) arc (90:270:.5cm);
	\draw (3.5, -.7) -- (3.5,-.6);
	\draw (3.5, -.4) -- (3.5, -.3);
	\draw (.5,-.3) -- (.5,-.7);
	
	\draw[red, very thick] (.5,.5) circle (.4cm);
	\draw[red, very thick] (.5,-.5) circle (.4cm);
	
	\draw[red, very thick, rounded corners=5mm] (3.2,.9) -- (2.4,.9) -- (2.4,-.9) -- (4,-.9) -- (4,.9) -- (3.2,.9);

	\draw[very thick, ->] (6,0) -- (7,0);
	\draw (6.5,-.2) node[below]{Case 1};

	\begin{scope}[xshift=9cm]

		\draw (.6,.5) -- (3.4,.5);
		\draw (3.6,.5) -- (3.8,.5);
		\draw (.5,.3) -- (.5,.7);
		\draw (.4,.5) -- (.2,.5);
		\draw (3.5,.3) -- (3.5,.7);
		
		\draw (.2, -.5) -- (.4, -.5);
		\draw (.6, -.5) -- (3.8,-.5);
		\draw (3.5, -.4) -- (3.5,-.3);
		\draw (3.5,-.6) -- (3.5,-.7);
		\draw (.5,-.3) -- (.5,-.7);
		
		\draw[red, very thick] (.5,.5) circle (.4cm);
		\draw[red, very thick] (.5,-.9) arc (270:90:.4cm);
		
		\draw[red, very thick, rounded corners=4mm] (.5,-.9) -- (4,-.9) -- (4,.9) -- (2.4,.9) -- (2.4,-.1) -- (.5,-.1);

	\end{scope}
\end{scope}

\begin{scope}[yshift = -6cm, thick]
	\draw (.6,.5) -- (1,.5);
	\draw (3.4,.5) -- (3,.5);
	\draw (3.6,.5) -- (3.8,.5);
	\draw (1,.5) arc (90:-90:.5cm);
	\draw (.5,.3) -- (.5,.7);
	\draw (.4,.5) -- (.2,.5);
	\draw (3.5,.3) -- (3.5,.7);

	\draw (.2, -.5) -- (1,-.5);
	\draw (3,-.5) -- (3.8,-.5);
	\draw (3,.5) arc (90:270:.5cm);
	\draw (3.5, -.7) -- (3.5,-.6);
	\draw (3.5, -.4) -- (3.5, -.3);
	\draw (.5,-.3) -- (.5,-.4);
	\draw (.5,-.6) -- (.5,-.7);
	
	\draw[red, very thick, rounded corners=5mm] (3.2,.9) -- (2.4,.9) -- (2.4,-.9) -- (4,-.9) -- (4,.9) -- (3.2,.9);
	
	\draw[very thick, ->] (6,0) -- (7,0);
	\draw (6.5,-.2) node[below]{Case 2};
	
	\begin{scope}[xshift=9cm]
		\draw (.6,.5) -- (3.4,.5);
		\draw (3.6,.5) -- (3.8,.5);
		\draw (.5,.3) -- (.5,.7);
		\draw (.4,.5) -- (.2,.5);
		\draw (3.5,.3) -- (3.5,.7);
		
		\draw (.2, -.5) -- (3.8,-.5);
		\draw (3.5, -.4) -- (3.5,-.3);
		\draw (3.5,-.6) -- (3.5,-.7);
		\draw (.5,-.3) -- (.5,-.4);
		\draw (.5,-.6) -- (.5,-.7);
		
		\draw[red, very thick, rounded corners=5mm] (3.2,.9) -- (2.4,.9) -- (2.4,-.9) -- (4,-.9) -- (4,.9) -- (3.2,.9);
		\draw[xshift = -2.3cm,red, very thick, rounded corners=5mm] (3.2,.9) -- (2.4,.9) -- (2.4,-.9) -- (4,-.9) -- (4,.9) -- (3.2,.9);
	
	\end{scope}

\end{scope}

\begin{scope}[yshift = -9cm, thick]
	\draw (.6,.5) -- (1,.5);
	\draw (3.4,.5) -- (3,.5);
	\draw (3.6,.5) -- (3.8,.5);
	\draw (1,.5) arc (90:-90:.5cm);
	\draw (.5,.3) -- (.5,.7);
	\draw (.4,.5) -- (.2,.5);
	\draw (3.5,.3) -- (3.5,.7);

	\draw (.2, -.5) -- (1,-.5);
	\draw (3,-.5) -- (3.8,-.5);
	\draw (3,.5) arc (90:270:.5cm);
	\draw (3.5, -.7) -- (3.5,-.6);
	\draw (3.5, -.4) -- (3.5, -.3);
	\draw (.5,-.3) -- (.5,-.4);
	\draw (.5,-.6) -- (.5,-.7);
	\draw[red, very thick, rounded corners=5mm] (3.2,.9) -- (2.4,.9) -- (2.4,-.9) -- (4,-.9) -- (4,.9) -- (3.2,.9);
	\draw[xshift = -2.3cm,red, very thick, rounded corners=5mm] (3.2,.9) -- (2.4,.9) -- (2.4,-.9) -- (4,-.9) -- (4,.9) -- (3.2,.9);

	\draw[very thick, ->] (6,0) -- (7,0);
	\draw (6.5,-.2) node[below]{Case 3};
	
	\begin{scope}[xshift=9cm]
		\draw (.6,.5) -- (3.4,.5);
		\draw (3.6,.5) -- (3.8,.5);
		\draw (.5,.3) -- (.5,.7);
		\draw (.4,.5) -- (.2,.5);
		\draw (3.5,.3) -- (3.5,.7);
		
		\draw (.2, -.5) -- (3.8,-.5);
		\draw (3.5, -.4) -- (3.5,-.3);
		\draw (3.5,-.6) -- (3.5,-.7);
		\draw (.5,-.3) -- (.5,-.4);
		\draw (.5,-.6) -- (.5,-.7);
		
		\draw[red, very thick, rounded corners=5mm] (3.2,.9) -- (2.4,.9) -- (2.4,-.9) -- (4,-.9) -- (4,.9) -- (3.2,.9);
	\draw[xshift = -2.3cm,red, very thick, rounded corners=5mm] (3.2,.9) -- (2.4,.9) -- (2.4,-.9) -- (4,-.9) -- (4,.9) -- (3.2,.9);
	
	\end{scope}

\end{scope}

\begin{scope}[yshift = -12cm, thick]
	\draw (.6,.5) -- (1,.5);
	\draw (3.4,.5) -- (3,.5);
	\draw (3.6,.5) -- (3.8,.5);
	\draw (1,.5) arc (90:-90:.5cm);
	\draw (.5,.3) -- (.5,.7);
	\draw (.4,.5) -- (.2,.5);
	\draw (3.5,.3) -- (3.5,.7);

	\draw (.2, -.5) -- (1,-.5);
	\draw (3,-.5) -- (3.8,-.5);
	\draw (3,.5) arc (90:270:.5cm);
	\draw (3.5, -.7) -- (3.5,-.6);
	\draw (3.5, -.4) -- (3.5, -.3);
	\draw (.5,-.3) -- (.5,-.4);
	\draw (.5,-.6) -- (.5,-.7);
	
	\draw[red, very thick, rounded corners=5mm] (3.2,.9) -- (2.4,.9) -- (2.4,-.9) -- (4,-.9) -- (4,.9) -- (3.2,.9);
	\draw[red, very thick] (.3,.2) -- (1,.2);
	\draw[red, very thick] (.3,-.2) -- (1,-.2);
	\draw[red, very thick] (1,.2) arc (90:-90:.2cm);
	\draw[red, very thick] (.3,.2) arc (90:270:.2cm);
	
	\draw[very thick, ->] (6,0) -- (7,0);
	\draw (6.5,-.2) node[below]{Case 4};
	
	\begin{scope}[xshift=9cm]
		\draw (.6,.5) -- (3.4,.5);
		\draw (3.6,.5) -- (3.8,.5);
		\draw (.5,.3) -- (.5,.7);
		\draw (.4,.5) -- (.2,.5);
		\draw (3.5,.3) -- (3.5,.7);
		
		\draw (.2, -.5) -- (3.8,-.5);
		\draw (3.5, -.4) -- (3.5,-.3);
		\draw (3.5,-.6) -- (3.5,-.7);
		\draw (.5,-.3) -- (.5,-.4);
		\draw (.5,-.6) -- (.5,-.7);
		\draw[red, very thick, rounded corners=5mm] (3.2,.9) -- (2.4,.9) -- (2.4,-.9) -- (4,-.9) -- (4,.9) -- (3.2,.9);
		\draw[red, very thick] (.3,.2) arc (90:270:.2cm);
		\draw[red, very thick] (.3,.2) -- (1,.2);
		\draw[red, very thick] (.3,-.2) -- (1,-.2);
		\draw[red, very thick] (1,.2) arc (-90:90:.3cm);
		\draw[red, very thick] (1,-.2) arc (90:-90:.3cm);
		\draw[red, very thick, rounded corners = 5mm] (1,.8) -- (0,.8) -- (0,-.8) -- (1,-.8);
	
	\end{scope}

\end{scope}

\begin{scope}[yshift = -15cm, thick]
	\draw (.6,.5) -- (1,.5);
	\draw (3.4,.5) -- (3,.5);
	\draw (3.6,.5) -- (3.8,.5);
	\draw (1,.5) arc (90:-90:.5cm);
	\draw (.5,.3) -- (.5,.7);
	\draw (.4,.5) -- (.2,.5);
	\draw (3.5,.3) -- (3.5,.7);

	\draw (.2, -.5) -- (1,-.5);
	\draw (3,-.5) -- (3.8,-.5);
	\draw (3,.5) arc (90:270:.5cm);
	\draw (3.5, -.7) -- (3.5,-.6);
	\draw (3.5, -.4) -- (3.5, -.3);
	\draw (.5,-.3) -- (.5,-.4);
	\draw (.5,-.6) -- (.5,-.7);
	
		\draw[red, very thick, rounded corners=5mm] (3.2,.9) -- (2.4,.9) -- (2.4,-.9) -- (4,-.9) -- (4,.9) -- (3.2,.9);
	\draw[red, very thick] (.3,.2) -- (1,.2);
	\draw[red, very thick] (.3,-.2) -- (1,-.2);
	\draw[red, very thick] (1,.2) arc (90:-90:.2cm);
	\draw[red, very thick] (.3,.2) arc (90:270:.2cm);
	\draw[xshift = -2.3cm,red, very thick, rounded corners=5mm] (3.2,.9) -- (2.2,.9) -- (2.2,-.9) -- (4,-.9) -- (4,.9) -- (3.2,.9);
	
	\draw[very thick, ->] (6,0) -- (7,0);
	\draw (6.5,-.2) node[below]{Case 5};
	
	\begin{scope}[xshift=9cm]
		\draw (.6,.5) -- (3.4,.5);
		\draw (3.6,.5) -- (3.8,.5);
		\draw (.5,.3) -- (.5,.7);
		\draw (.4,.5) -- (.2,.5);
		\draw (3.5,.3) -- (3.5,.7);
		
		\draw (.2, -.5) -- (3.8,-.5);
		\draw (3.5, -.4) -- (3.5,-.3);
		\draw (3.5,-.6) -- (3.5,-.7);
		\draw (.5,-.3) -- (.5,-.4);
		\draw (.5,-.6) -- (.5,-.7);
				\draw[red, very thick, rounded corners=5mm] (3.2,.9) -- (2.4,.9) -- (2.4,-.9) -- (4,-.9) -- (4,.9) -- (3.2,.9);
		\draw[red, very thick] (.3,.2) arc (90:270:.2cm);
		\draw[red, very thick] (.3,.2) -- (1,.2);
		\draw[red, very thick] (.3,-.2) -- (1,-.2);
		\draw[red, very thick] (1,.2) arc (-90:90:.3cm);
		\draw[red, very thick] (1,-.2) arc (90:-90:.3cm);
		\draw[red, very thick, rounded corners = 5mm] (1,.8) -- (0,.8) -- (0,-.8) -- (1,-.8);
	
	\end{scope}

\end{scope}

\begin{scope}[yshift = -18cm, thick]
	\draw (.6,.5) -- (1,.5);
	\draw (3.4,.5) -- (3,.5);
	\draw (3.6,.5) -- (3.8,.5);
	\draw (1,.5) arc (90:-90:.5cm);
	\draw (.5,.3) -- (.5,.7);
	\draw (.4,.5) -- (.2,.5);
	\draw (3.5,.3) -- (3.5,.7);

	\draw (.2, -.5) -- (1,-.5);
	\draw (3,-.5) -- (3.8,-.5);
	\draw (3,.5) arc (90:270:.5cm);
	\draw (3.5, -.7) -- (3.5,-.6);
	\draw (3.5, -.4) -- (3.5, -.3);
	\draw (.5,-.3) -- (.5,-.4);
	\draw (.5,-.6) -- (.5,-.7);
	\draw[red, very thick, rounded corners=5mm] (3.2,.9) -- (2.4,.9) -- (2.4,-.9) -- (4,-.9) -- (4,.9) -- (3.2,.9);
	\draw[red, very thick] (.3,.8) arc (90:270:.3cm);
	\draw[red, very thick] (.3,-.8) arc (270:90:.3cm);
	\draw[red, very thick] (1,.2) arc (90:-90:.2cm);
	\draw[red, very thick] (.3,.2) -- (1,.2);
	\draw[red, very thick] (.3,-.2) -- (1, -.2);
	\draw[red, very thick, rounded corners = 5mm] (.3,.8) -- (1.7, .8) -- (1.7,-.8) -- (.3,-.8);

	\draw[very thick, ->] (6,0) -- (7,0);
	\draw (6.5,-.2) node[below]{Case 6};
	
	\begin{scope}[xshift=9cm]
		\draw (.6,.5) -- (3.4,.5);
		\draw (3.6,.5) -- (3.8,.5);
		\draw (.5,.3) -- (.5,.7);
		\draw (.4,.5) -- (.2,.5);
		\draw (3.5,.3) -- (3.5,.7);
		
		\draw (.2, -.5) -- (3.8,-.5);
		\draw (3.5, -.4) -- (3.5,-.3);
		\draw (3.5,-.6) -- (3.5,-.7);
		\draw (.5,-.3) -- (.5,-.4);
		\draw (.5,-.6) -- (.5,-.7);

		\draw[red, very thick, rounded corners=5mm] (3.2,.9) -- (2.4,.9) -- (2.4,-.9) -- (4,-.9) -- (4,.9) -- (3.2,.9);
			\draw[red, very thick] (.5,.5) circle (.4cm);
	\draw[red, very thick] (.5,-.5) circle (.4cm);
	
	\end{scope}

\end{scope}

\begin{scope}[yshift = -21cm, thick]
	\draw (.6,.5) -- (1,.5);
	\draw (3.8,.5) -- (3,.5);
	\draw (1,.5) arc (90:-90:.5cm);
	\draw (.5,.3) -- (.5,.7);
	\draw (.4,.5) -- (.2,.5);
	\draw (3.5,.3) -- (3.5,.4);
	\draw (3.5,.6) -- (3.5,.7);

	\draw (.2, -.5) -- (1,-.5);
	\draw (3,-.5) -- (3.4,-.5);
	\draw (3.6, -.5) -- (3.8,-.5);
	\draw (3,.5) arc (90:270:.5cm);
	\draw (3.5, -.7) -- (3.5,-.3);
	\draw (.5,-.3) -- (.5,-.4);
	\draw (.5,-.6) -- (.5,-.7);
	
	\draw[very thick, ->] (6,0) -- (7,0);
	\draw (6.5,-.2) node[below]{Case 7};
	
	\begin{scope}[xshift=9cm]
		\draw (.6,.5) -- (3.8,.5);
		\draw (.5,.3) -- (.5,.7);
		\draw (.4,.5) -- (.2,.5);
		\draw (3.5,.3) -- (3.5,.4);
		\draw (3.5,.6) -- (3.5,.7);
		
		\draw (.2, -.5) -- (3.4,-.5);
		\draw (3.6, -.5) -- (3.8,-.5);
		\draw (3.5, -.7) -- (3.5,-.3);
		\draw (.5,-.3) -- (.5,-.4);
		\draw (.5,-.6) -- (.5,-.7);

	\end{scope}

\end{scope}

\end{tikzpicture}$$
\caption{Taking the connected sum of $\widetilde{D}_i=D_1\#\cdots \#D_i$ and the alternating diagram $D_{i+1}$.}
\label{figure:altdecompsum}
\end{figure}

{\bf Case 1:} Suppose that $e_i$ is non-alternating. Figure \ref{figure:altdecompsum} shows the endpoints of $e_i$ passing under the crossing, but the case where the endpoints pass over the crossing is exactly the same. Taking the connected sum merges the curve in the alternating decomposition of $D_{i+1}$ with one of the curves in the alternating decomposition of $\widetilde{D}_i$. Therefore $\widetilde{G}_{i+1} = \widetilde{G}_i$.

{\bf Case 2:} Suppose that $e_i$ is alternating and the connected sum is taken as in Figure \ref{figure:altdecompsum}. Also suppose that both $F_1$ and $F_2$ are alternating faces of $\widetilde{D}_i$. Then there are no alternating decomposition curves of $\widetilde{D}_i$ in either $F_1$ or $F_2$. Hence $\widetilde{G}_{i+1} = \widetilde{G}_i \sqcup C_2$ where $C_2$ is a two cycle.

{\bf Case 3:} Suppose that $e_i$ is alternating and the connected sum is taken as in Figure \ref{figure:altdecompsum}. Also, suppose that $F_1$ is an alternating face of $\widetilde{D}_i$, while $F_2$ is a non-alternating face of $\widetilde{D}_i$. Let $\gamma$ be the alternating decomposition curve in $F_2$ that runs along $e_i$. After performing the connected sum, the curve $\gamma$ transforms into a curve that runs along the same portion of the boundary of $F_2$ and also along all of $F_1$. Thus the connected sum attaches the alternating decomposition curve of $D_{i+1}$ to $\gamma$ by two edges.  Hence $\widetilde{G}_{i+1} = \widetilde{G}_i \oplus_1 C_2$. The transformation $G \mapsto G\oplus_1 C_2$ is called a {\em doubled pendant} move and is depicted in Figure \ref{figure:doubledpendant}.
\begin{figure}[h]
$$\begin{tikzpicture}
\draw (0,0) -- (-.7,.7);
\draw (-.7,0.1) node{$\vdots$};
\draw (0,0) -- (-.7,-.7);

\fill[black!20!white] (0,0) circle (.3cm);
\draw (0,0) circle (.3cm);
\draw (0,0) node{$v_1$};
\draw (-.2,-1) node{$G$};

\draw[very thick, ->] (1,0) -- (2,0);

\begin{scope}[xshift = 3.5cm]

\draw (.6,-1) node{$G\oplus_1 C_2$};

\draw (0,0) -- (-.7,.7);
\draw (-.7,0.1) node{$\vdots$};
\draw (0,0) -- (-.7,-.7);

\draw (0,.15) -- (1.5,.15);
\draw (0,-.15) -- (1.5,-.15);

\fill[black!20!white] (0,0) circle (.3cm);
\draw (0,0) circle (.3cm);
\draw (0,0) node{$v_1$};

\fill[black!20!white] (1.5,0) circle (.3cm);
\draw (1.5,0) circle (.3cm);
\draw (1.5,0) node{$v_2$};

\end{scope}
\end{tikzpicture}$$
\caption{A doubled pendant move on $G$ results in the graph $G\oplus_1 C_2$.}
\label{figure:doubledpendant}
\end{figure}
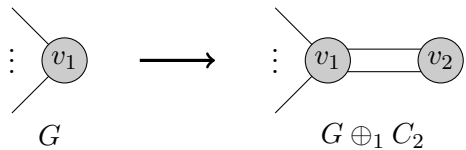

{\bf Case 4:} Suppose that $e_i$ is alternating and the connected sum is taken as in Figure \ref{figure:altdecompsum}. Also, suppose that $F_1$ is a non-alternating face of $\widetilde{D}_i$, while $F_2$ is an alternating face of $\widetilde{D}_i$. Let $\gamma$ be the alternating decomposition curve in $F_1$ that runs along $e_i$. After performing the connected sum, the curve $\gamma$ transforms into a curve that runs along the same portion of the boundary of $F_1$ and also along all of $F_2$. Thus the connected sum attaches the alternating decomposition curve of $D_{i+1}$ to $\gamma$ by two edges. Hence $\widetilde{G}_{i+1} = \widetilde{G}_i \oplus_1 C_2$.

{\bf Case 5:} Suppose that $e_i$ is alternating and the connected sum is taken as in Figure \ref{figure:altdecompsum}. Also, suppose that both $F_1$ and $F_2$ are non-alternating faces of $\widetilde{D}_i$ and that the alternating decomposition curves $\gamma_1$ and $\gamma_2$ that run along $e_i$ are distinct curves. Since the region bounded by $\gamma_1$ and $\gamma_2$ contains crossings of $\widetilde{D}_i$, it follows that the vertices of $\widetilde{G}_i$ corresponding to $\gamma_1$ and $\gamma_2$ lie in different components of $\widetilde{G}_i$. Performing the connected sum operation merges $\gamma_1$ and $\gamma_2$, and connects the alternating decomposition curve of $D_{i+1}$ to the newly merged $\gamma_1$ and $\gamma_2$ with two edges. Therefore, $\widetilde{G}_{i+1}$ is obtained from $\widetilde{G}_i$ by taking a one-sum along two vertices in separate components of $\widetilde{G}_i$ and then an additional one-sum with $C_2$.

{\bf Case 6:} Suppose that $e_i$ is alternating and the connected sum is taken as in Figure \ref{figure:altdecompsum}. Also, suppose that both $F_1$ and $F_2$ are non-alternating faces of $\widetilde{D}_i$ and that there is a single alternating decomposition curve that runs along $e_i$ in both $F_1$ and $F_2$. Performing a connected sum operation splits this alternating decomposition curve into two curves, each of which has a single edge attached to the alternating decomposition curve of $D_{i+1}$. Thus the graph $\widetilde{G}_{i+1}$ is obtained from $\widetilde{G}_i$ by 
\begin{enumerate}
\item picking a vertex $v$ of $\widetilde{G}_i$, 
\item partitioning the edges incident to $v$ into two sets $A$ and $B$ each of odd order, 
\item splitting the vertex $v$ into two new vertices $v_1$ and $v_2$ where the edge set $A$ is incident to $v_1$ and the edge set $B$ is incident to $v_2$, and
\item creating a new vertex $v_3$ of degree two adjacent to both $v_1$ and $v_2$.
\end{enumerate}
See Figure \ref{figure:two-path} for a depiction of this operation, which we call a {\em two-path extension}.
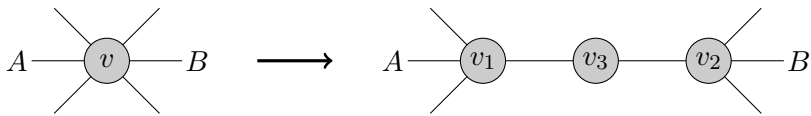
\begin{figure}[h]
$$\begin{tikzpicture}

\draw (0,0) -- (-.7,.7);
\draw (0,0) -- (-1,0);
\draw (0,0) -- (-.7,-.7);
\draw (-1.2,0) node{$A$};

\draw(0,0) -- (1,0);
\draw (0,0) -- (.7,.7);
\draw (0,0) -- (.7,-.7);
\draw (1.2,0) node{$B$};

\fill[black!20!white] (0,0) circle (.3cm);
\draw (0,0) circle (.3cm);
\draw (0,0) node{$v$};

\draw[very thick, ->] (2,0) -- (3,0);

\begin{scope}[xshift = 5cm]
\draw (0,0) -- (-.7,.7);
\draw (0,0) -- (-1,0);
\draw (0,0) -- (-.7,-.7);
\draw (-1.2,0) node{$A$};

\draw(3,0) -- (4,0);
\draw (3,0) -- (3.7,.7);
\draw (3,0) -- (3.7,-.7);
\draw (4.2,0) node{$B$};

\draw (0,0) -- (3,0);

\fill[black!20!white] (0,0) circle (.3cm);
\draw (0,0) circle (.3cm);
\draw (0,0) node{$v_1$};

\fill[black!20!white] (1.5,0) circle (.3cm);
\draw (1.5,0) circle (.3cm);
\draw (1.5,0) node{$v_3$};

\fill[black!20!white] (3,0) circle (.3cm);
\draw (3,0) circle (.3cm);
\draw (3,0) node{$v_2$};

\end{scope}

\end{tikzpicture}$$
\caption{A two-path extension. The edge sets $A$ and $B$ must each be of odd order.}
\label{figure:two-path}
\end{figure}

{\bf Case 7:} Suppose that $e_i$ is alternating and the connected sum is taken as in Figure \ref{figure:altdecompsum}. Note that this connected sum is different than Cases 2 - 6. In this case, it does not matter whether either, neither, or both of $F_1$ and $F_2$ are alternating or non-alternating. In each case, we have $\widetilde{G}_{i+1} = \widetilde{G}_i$. 

\begin{theorem}
\label{theorem:Zero}
Let $G$ be an alternating decomposition graph with $g_T(G)=0$. Then $G$ can be obtained from a collection of isolated vertices via a sequence of doubled pendant moves, two-path extensions, and one-sums along vertices in different components.
\end{theorem}
\begin{proof}
Suppose $D$ is a link diagram with alternating decomposition graph $G$. Then $g_T(D) = g_T(G) = 0$, and hence $D=D_1\# \cdots \#D_k$ is a connected sum of alternating diagrams $D_1, \dots, D_k$. Let $\widetilde{D}_i  = D_1\# \cdots \# D_i$, and let $\widetilde{G}_i$ be the alternating decomposition graph of $\widetilde{G}_i$. Our analysis above shows that there is a sequence $\widetilde{G}_1,\widetilde{G}_2,\dots,\widetilde{G}_k=G$ of alternating decomposition graphs such that $\widetilde{G}_1$ is a collection of isolated vertices and $\widetilde{G}_{i+1}$ can be obtained from $\widetilde{G}_i$ by either doing nothing, a doubled pendant move, a two path extension, a disjoint union with $C_2$, or the multi-step operation of Case 5 (which stipulated that we glue together two components of $\widetilde{G}_i$ along a vertex, and then perform a doubled pendant move to the same vertex). 

We modify the sequence $\widetilde{G}_1,\dots,\widetilde{G}_k=G$ so that it will still begin in a collection of isolated vertices, still end in $G$, and so that each graph can be obtained from the previous one via a doubled pendant move, a two-path extension, or by identifying two vertices in different components. For each $i$ where $\widetilde{G}_{i+1}$ is obtained from $\widetilde{G}_i$ via a disjoint union with $C_2$, we modify $\widetilde{G}_j$ for $j\leq i$ by adding an isolated vertex $v$. Since $G\sqcup C_2 = G \sqcup \{v\}\oplus_1C_2$, we have changed adding a disjoint union of $C_2$ into doubled pendant move. 

For each $i$ where $\widetilde{G}_{i+1}$ is obtained from $\widetilde{G}_i$ via the operation in Case 5, we note that $\widetilde{G}_{i+1}$ is obtained from $\widetilde{G}_i$ by taking a one-sum of vertices in different components and then performing a doubled pendant move. In order to satisfy the conditions in the theorem, these two operations must be completed in separate steps. Thus we modify the sequence by increasing the index by one of each $\widetilde{G}_j$ with $j\geq i+1$. Then we set $\widetilde{G}_{i+1}$ to be the graph obtained from $\widetilde{G}_i$ by taking the prescribed one-sum of vertices in different components, and we then $\widetilde{G}_{i+2}$ can be obtained from $\widetilde{G}_{i+1}$ by a doubled pendant move.
\end{proof}

Recall that an alternating decomposition graph $G$ is reduced if it is a single vertex or if each component of $G$ is $3$-edge connected. In the following proposition, we prove that there exists a Turaev genus minimizing diagram of every non-split link with reduced alternating decomposition graph.

\begin{proposition}
Every non-split link $L$ has a diagram $D$ with alternating decomposition graph $G$ such that $G$ is reduced and such that $g_T(G) = g_T(L)$.
\end{proposition}
\begin{proof}
Equation \eqref{equation:Turaevgenus} implies that for any choice of edge along which to take a connected sum of $D_1$ and $D_2$, we have $g_T(D_1\# D_2) = g_T(D_1) + g_T(D_2)$. Let $D'$ be a diagram of $L$ that minimizes Turaev genus, i.e. such that $g_T(D')=g_T(L)$. Suppose that $D'$ can be written as a connected sum $D_1\# \cdots \# D_k$ where each $D_i$ cannot be realized as a connected sum. Let $G_i$ be the alternating decomposition graph of $D_i$. 

Since each $D_i$ cannot be realized as a connected sum, there is no circle in the plane that intersects $D_i$ exactly twice such that the two one-tangles formed are non-trivial. Therefore, there is no circle in the plane that intersects the alternating decomposition graph of $D_i$ exactly twice in two distinct edges. Hence the alternating decomposition graph $G_i$ is reduced.

However, the alternating decomposition graph $G'$ of $D'$ is not necessarily reduced. We construct another diagram $D$ of $L$ such that $g_T(D)=g_T(D')=g_T(L)$, and such that the alternating decomposition graph $G$ of $D$ is reduced. Suppose the connected sum of two diagrams $\widetilde{D}_1$ and $\widetilde{D}_2$ is formed in the same manner as Case 7 of Figure \ref{figure:altdecompsum}. Let $e_1$ and $e_2$ be the edges along which the connected sum is being taken, and let $F_1$, $F_2$, and $F_3$ be the three faces with $e_1$ and $e_2$ in their boundary (as in Figure \ref{figure:connectedsum}). If at least two of $F_1$, $F_2$, and $F_3$ are alternating faces, then the alternating decomposition graph of $\widetilde{D}_1\#\widetilde{D}_2$ is either the one-sum or disjoint union of the alternating decomposition graphs of $\widetilde{D}_1$ and $\widetilde{D}_2$. Therefore, if the alternating decomposition graphs of $\widetilde{D}_1$ and $\widetilde{D}_2$ are reduced, then the alternating decomposition graph of $\widetilde{D}_1\# \widetilde{D}_2$ is reduced.

For each summand $D_1,\dots, D_k$ in $D=D_1\# \cdots \# D_k$, insert a small twist into the edge on which a connected sum occurs, as in Figure \ref{figure:twist}. Inserting the twist does not change the alternating decomposition graph of each $D_i$, and thus does not change the genus of the associated Turaev surface. Each new twisted edge is an alternating edge, and the face bounded by that single alternating edge is an alternating face. Therefore, if all of the connected sums are taken along these twisted edges, then the alternating decomposition graph $G$ of the resulting diagram $D$ will be reduced. Moreover, since adding the twists does not change the genus of the Turaev surface, we have $g_T(D)=g_T(D')=g_T(L)$.
\end{proof}

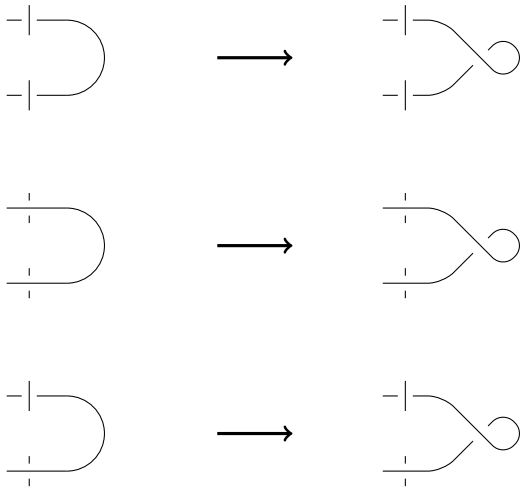
\begin{figure}[h]
$$\begin{tikzpicture}
	\draw (.6,.5) -- (1,.5);

	\draw (1,.5) arc (90:-90:.5cm);
	\draw (.5,.3) -- (.5,.7);
	\draw (.4,.5) -- (.2,.5);

	\draw (.6, -.5) -- (1,-.5);
	\draw (.2, -.5) -- (.4, -.5);
	\draw (.5,-.3) -- (.5,-.7);
	
	\draw[very thick, ->] (3,0) -- (4,0);
	
	\begin{scope}[xshift=5cm]
	
	\draw[rounded corners = 2mm] (.6,.5) -- (1,.5) -- (1.8,-.3) -- (2.1,0) -- (1.8,.3) -- (1.6,.1);

	\draw (.5,.3) -- (.5,.7);
	\draw (.4,.5) -- (.2,.5);

	\draw [rounded corners = 2mm] (.6, -.5) -- (1,-.5) -- (1.4,-.1);
	\draw (.2, -.5) -- (.4, -.5);
	\draw (.5,-.3) -- (.5,-.7);
	\end{scope}
	
\begin{scope}[yshift = -2.5cm]

\draw (.5,.3) -- (.5,.4);
\draw (.5,.6) -- (.5,.7);

\draw (.5,-.3) -- (.5,-.4);
\draw (.5,-.6) -- (.5,-.7);

\draw (.2,.5) -- (1,.5);
\draw (.2,-.5) -- (1,-.5);

\draw (1,.5) arc (90:-90:.5cm);

\draw[very thick, ->] (3,0) -- (4,0);
	
\begin{scope}[xshift=5cm]
	\draw (.5,.3) -- (.5,.4);
	\draw (.5,.6) -- (.5,.7);

	\draw (.5,-.3) -- (.5,-.4);
	\draw (.5,-.6) -- (.5,-.7);
	
	\draw[rounded corners = 2mm] (.2,.5) -- (1,.5) -- (1.8,-.3) -- (2.1,0) -- (1.8,.3) -- (1.6,.1);
	\draw [rounded corners = 2mm] (.2, -.5) -- (1,-.5) -- (1.4,-.1);

	\end{scope}

\end{scope}

\begin{scope}[yshift=-5cm]
	\draw (.6,.5) -- (1,.5);
	\draw (1,.5) arc (90:-90:.5cm);
	\draw (.5,.3) -- (.5,.7);
	\draw (.4,.5) -- (.2,.5);

	\draw (.2, -.5) -- (1,-.5);
	\draw (.5,-.3) -- (.5,-.4);
	\draw (.5,-.6) -- (.5,-.7);

\draw[very thick, ->] (3,0) -- (4,0);
	
\begin{scope}[xshift=5cm]
	\draw[rounded corners = 2mm] (.6,.5) -- (1,.5) -- (1.8,-.3) -- (2.1,0) -- (1.8,.3) -- (1.6,.1);

	\draw (.5,.3) -- (.5,.7);
	\draw (.4,.5) -- (.2,.5);

	\draw[rounded corners = 2mm] (.2, -.5) -- (1,-.5) -- (1.4,-.1);
	\draw (.5,-.3) -- (.5,-.4);
	\draw (.5,-.6) -- (.5,-.7);
\end{scope}
\end{scope}
	
\end{tikzpicture}$$

\caption{Inserting twists into edges where a connected sum is taken makes the resulting diagram have reduced alternating decomposition graph.}
\label{figure:twist}
\end{figure}

\section{Turaev genus classification results}
\label{section:classification}
In this section, we classify all reduced alternating decomposition graphs of Turaev genus one and two. We also show that for any non-negative integer $k$, there are a finite number doubled path equivalence classes of alternating decomposition graphs of Turaev genus $k$. Hence there exists a classification of all reduced alternating decomposition graphs of Turaev genus $k$ for any non-negative integer $k$.

A graph $G$ is called a {\em doubled forest} if it is obtained from a forest by doubling every edge. A {\em doubled tree} is a doubled forest with one component. Let $C_4(p,q,r,s)$ be the graph obtained by attaching doubled paths of lengths $p$, $q$, $r$, and $s$ to the vertices of a four cycle. Also, let $\widetilde{K}_4(p,q)$ be the graph obtained by removing an edge of the complete graph on four vertices $K_4$ and then attaching doubled paths of lengths $p$ and $q$ to the vertices incident to the removed edge. Let $\widetilde{K}_4(p,q)\oplus_2 \widetilde{K}_4(r,s)$ be the two-sum of $\widetilde{K}_4(p,q)$ and $\widetilde{K}_4(r,s)$ taken along the unique edge in each summand that is not contained in nor adjacent to a doubled path. See Figure \ref{figure:TuraevGenusZero}.
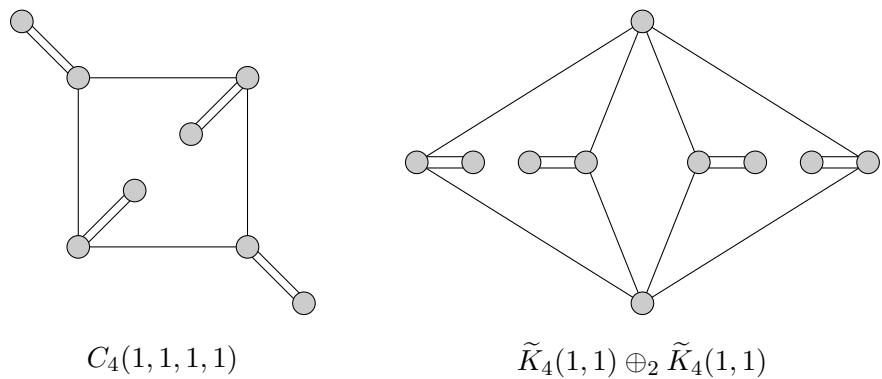
\begin{figure}[h]
$$\begin{tikzpicture}[scale=.75]

\draw (0,0) rectangle (3,3);
\draw (0,.1) -- (1,1.1);
\draw (0,-.1) -- (1,.9);
\draw (3,3.1) -- (2,2.1);
\draw (3,2.9) -- (2,1.9);
\draw (0,3.1) -- (-1,4.1);
\draw (0,2.9) -- (-1,3.9);
\draw (3,.1) -- (4,-.9);
\draw (3,-.1) -- (4,-1.1);

\fill[black!20!white] (0,0) circle (.2cm);
\fill[black!20!white] (3,0) circle (.2cm);
\fill[black!20!white] (3,3) circle (.2cm);
\fill[black!20!white] (0,3) circle (.2cm);
\draw (0,0) circle (.2cm);
\draw (3,0) circle (.2cm);
\draw (3,3) circle (.2cm);
\draw (0,3) circle (.2cm);

\fill[black!20!white] (1,1) circle (.2cm);
\fill[black!20!white] (2,2) circle (.2cm);
\draw (1,1) circle (.2cm);
\draw (2,2) circle (.2cm);

\fill[black!20!white] (-1,4) circle (.2cm);
\draw (-1,4) circle (.2cm);
\fill [black!20!white] (4,-1) circle (.2cm);
\draw (4,-1) circle (.2cm);

\draw (1.5,-2) node{$C_4(1,1,1,1)$};

\begin{scope}[xshift = 10cm]
	
	\draw (0,-1) -- (-4,1.5) -- (0,4) -- (-1,1.5) -- (0,-1) -- (1,1.5) -- (0,4) -- (4,1.5) -- (0,-1);
	\draw (-4,1.6) -- (-3,1.6);
	\draw (-4,1.4) -- (-3,1.4);
	\draw (-2,1.6) -- (-1,1.6);
	\draw (-2,1.4) -- (-1,1.4);
	\draw (4,1.6) -- (3,1.6);
	\draw (4,1.4) -- (3,1.4);
	\draw (2,1.6) -- (1,1.6);
	\draw (2,1.4) -- (1,1.4);

	\fill[black!20!white] (0,-1) circle (.2cm);
	\fill[black!20!white] (0,4) circle (.2cm);
	\draw (0,-1) circle (.2cm);
	\draw (0,4) circle (.2cm);
	
	\fill[black!20!white] (-4,1.5) circle (.2cm);
	\fill[black!20!white] (-3,1.5) circle (.2cm);
	\fill[black!20!white] (-2,1.5) circle (.2cm);
	\fill[black!20!white] (-1,1.5) circle (.2cm);
	\fill[black!20!white] (1,1.5) circle (.2cm);
	\fill[black!20!white] (2,1.5) circle (.2cm);
	\fill[black!20!white] (3,1.5) circle (.2cm);
	\fill[black!20!white] (4,1.5) circle (.2cm);
	\draw (-4, 1.5) circle (.2cm);
	\draw (-3, 1.5) circle (.2cm);
	\draw (-2, 1.5) circle (.2cm);
	\draw (-1, 1.5) circle (.2cm);
	\draw (1, 1.5) circle (.2cm);
	\draw (2, 1.5) circle (.2cm);
	\draw (3, 1.5) circle (.2cm);
	\draw (4, 1.5) circle (.2cm);
	
	\draw (0,-2) node{$\widetilde{K}_4(1,1)\oplus_2 \widetilde{K}_4(1,1)$};

\end{scope}

\end{tikzpicture}$$
\caption{\textcolor{black}{The graphs $C_4(1,1,1,1)$ and $\widetilde{K}_4(1,1)\oplus_2 \widetilde{K}_4(1,1)$.}}
\label{figure:TuraevGenusZero}
\end{figure}

\begin{lemma}
\label{lemma:zero}
Let $H$ be an alternating decomposition graph without isolated vertices such that $g_T(H)=0$ and $H$ has at most four vertices of degree two. Then $H$ is either
\begin{enumerate}
\item a disjoint union of two doubled paths,
\item a doubled tree with two, three, or four leaves,
\item $C_4(p,q,r,s)$ for non-negative integers $p$, $q$, $r$, and $s$, or
\item $\widetilde{K}_4(p,q) \oplus_2 \widetilde{K}_4(r,s)$ for non-negative integers $p$, $q$, $r$, and $s$.
\end{enumerate}
\end{lemma}
\begin{proof}
Each of the above graphs clearly has four or fewer vertices of degree two, and the algorithm of Corollary \ref{cor:TuraevAlgorithm} implies that each of the above graphs is indeed Turaev genus zero. It remains to show that the above list is exhaustive.

Theorem \ref{theorem:Zero} states that every Turaev genus zero alternating decomposition graph can be obtained from a collection of isolated vertices via a sequence of doubled pendant moves, two-path extensions, and one-sums of vertices in distinct components. If $H$ is obtained from a collection of isolated vertices via a sequence of doubled pendant moves and one-sums from distinct components, then $H$ is a doubled forest. Since $H$ has four or fewer vertices of degree two and no isolated vertices, $H$ is either a disjoint union of two doubled paths or a doubled tree with two, three, or four leaves.

If a doubled tree $H$ has a vertex of degree $2d$ for some positive integer $d$, then $H$ contains at least $d$ vertices of degree two. A two-path extension always increases the number of degree two vertices in the graph. Therefore, we can only apply a two-path extension to a vertex of degree two, four, or six. Let $H'$ be obtained from the doubled tree $H$ via a two-path extension applied at a vertex $v$ where the set of edges incident to $v$ is partitioned into sets $A$ and $B$ of odd order, as in Figure \ref{figure:two-path}. Without loss of generality, assume $|A|\geq |B|$.

If the degree of $v$ is two, then $|A|=|B|=1$. Therefore, a two-path extension will add two new vertices of degree two. Hence $H$ must be a doubled path, and $H'$ is $C_4(p,0,0,0)$ for some $p$. If the degree of $v$ is four, then $|A|=3$ and $|B|=1$. A two-path extension will again add two vertices of degree two, and hence $H$ must be a doubled path. Thus $H'$ is $C_4(p,q,0,0)$ for some $p$ and $q$. 

If the degree of $v$ is six, then $H$ already has at least three vertices of degree two. If $|A|=5$ and $|B|=1$, then a two-path extension would create two new vertices of degree two, resulting in at least five vertices of degree two. Therefore $|A|=3$ and $|B|=3$, and $H$ is a doubled tree with three degree two vertices. Let $\mathcal{N}(A)$ (respectively $\mathcal{N}(B)$) be the set of vertices adjacent to $v$ and incident to an edge in $A$ (respectively $B$). There are two cases: either $|\mathcal{N}(A)| = |\mathcal{N}(B)| = 2$ or $|\mathcal{N}(A)| = |\mathcal{N}(B)| = 3$. If $|\mathcal{N}(A)| = |\mathcal{N}(B)| = 2$, then $H'= C_4(p,0,r,0)$ for some $p$ and $r$. If $|\mathcal{N}(A)| = |\mathcal{N}(B)| = 3$, then $H'=\widetilde{K}_4(p,0)\oplus_2\widetilde{K}_4(r,s)$ for some $p$, $r$, and $s$.

In each of the above instances, $H'$ already has four vertices of degree two. Thus the only allowable operation is a doubled pendant move applied to a vertex that is already of degree two. Alternately, one could take a one-sum between $H'$ and a doubled path that identifies two degree two vertices. However, this is the same as a doubled pendant move applied to a vertex of degree two. The only effect his has is changing the parameters in $C_4(p,q,r,s)$ or $\widetilde{K}_4(p,q)\oplus_2\widetilde{K}_4(r,s)$, and hence the result holds.
\end{proof}
Figure \ref{figure:two-pathlemma} shows examples of a two-path extension being applied to a doubled tree with two or three vertices of degree two.

The previous classification of alternating decomposition graphs of Turaev genus zero with at most four vertices of degree two leads directly to the classification reduced alternating decomposition graphs of Turaev genus one and two.

\begin{figure}[h]
$$\begin{tikzpicture}[thick, scale=.8]


\draw (0,.2) -- (4.5,.2);
\draw (0,-.2) -- (4.5, -.2);

\draw [red] (4,0) -- (5,0);

\fill[black!20!white] (0,0) circle (.3cm);
\draw (0,0) circle (.3cm);
\draw (0,0) node {$u_1$};

\fill[black!20!white] (1.5,0) circle (.3cm);
\draw (1.5,0) circle (.3cm);
\draw (1.5,0) node {$u_2$};

\fill[black!20!white] (3,0) circle (.3cm);
\draw (3,0) circle (.3cm);
\draw (3,0) node {$u_3$};

\fill[black!20!white] (4.5,0) circle (.3cm);
\draw (4.5,0) circle (.3cm);
\draw (4.5,0) node {$v$};

\draw[very thick, ->] (6,0) -- (7,0);

\begin{scope}[xshift=8.5cm]
\draw (0,.2) -- (3,.2);
\draw (0,-.2) -- (3, -.2);
\draw (3,0) rectangle (4.5,-1.5);

\fill[black!20!white] (0,0) circle (.3cm);
\draw (0,0) circle (.3cm);
\draw (0,0) node {$u_1$};

\fill[black!20!white] (1.5,0) circle (.3cm);
\draw (1.5,0) circle (.3cm);
\draw (1.5,0) node {$u_2$};

\fill[black!20!white] (3,0) circle (.3cm);
\draw (3,0) circle (.3cm);
\draw (3,0) node {$u_3$};

\fill[black!20!white] (4.5,0) circle (.3cm);
\draw (4.5,0) circle (.3cm);
\draw (4.5,0) node {$v_2$};

\fill[black!20!white] (4.5,-1.5) circle (.3cm);
\draw (4.5,-1.5) circle (.3cm);
\draw (4.5,-1.5) node {$v_3$};

\fill[black!20!white] (3,-1.5) circle (.3cm);
\draw (3,-1.5) circle (.3cm);
\draw (3,-1.5) node {$v_1$};

\end{scope}


\begin{scope}[yshift = -2.5cm]

\draw (0,.2) -- (4.5,.2);
\draw (0,-.2) -- (4.5, -.2);

\draw [red] (2,0) -- (1.5,0) -- (1.5,-.5);

\fill[black!20!white] (0,0) circle (.3cm);
\draw (0,0) circle (.3cm);
\draw (0,0) node {$u_1$};

\fill[black!20!white] (1.5,0) circle (.3cm);
\draw (1.5,0) circle (.3cm);
\draw (1.5,0) node {$v$};

\fill[black!20!white] (3,0) circle (.3cm);
\draw (3,0) circle (.3cm);
\draw (3,0) node {$u_2$};

\fill[black!20!white] (4.5,0) circle (.3cm);
\draw (4.5,0) circle (.3cm);
\draw (4.5,0) node {$u_3$};

\draw[very thick, ->] (6,0) -- (7,0);

\begin{scope}[xshift=8.5cm]
\draw (0,.2) -- (1.5,.2);
\draw (0,-.2) -- (1.5,-.2);
\draw (1.5,0) rectangle (3,-1.5);
\draw (3,.2) -- (4.5,.2);
\draw (3,-.2) -- (4.5,-.2);

\fill[black!20!white] (0,0) circle (.3cm);
\draw (0,0) circle (.3cm);
\draw (0,0) node {$u_1$};

\fill[black!20!white] (1.5,0) circle (.3cm);
\draw (1.5,0) circle (.3cm);
\draw (1.5,0) node {$v_1$};

\fill[black!20!white] (3,0) circle (.3cm);
\draw (3,0) circle (.3cm);
\draw (3,0) node {$u_2$};

\fill[black!20!white] (4.5,0) circle (.3cm);
\draw (4.5,0) circle (.3cm);
\draw (4.5,0) node {$u_3$};

\fill[black!20!white] (1.5,-1.5) circle (.3cm);
\draw (1.5,-1.5) circle (.3cm);
\draw (1.5,-1.5) node {$v_3$};

\fill[black!20!white] (3,-1.5) circle (.3cm);
\draw (3,-1.5) circle (.3cm);
\draw (3,-1.5) node {$v_2$};
\end{scope}

\end{scope}

\begin{scope}[yshift=-5cm]

\draw (.75,.2) -- (3.75,.2);
\draw (.75,-.2) -- (3.75,-.2);
\draw (2.05,0) -- (2.05, -1.5);
\draw (2.45,0) -- (2.45, -1.5);

\draw [red] (1.75,0) -- (2.25,0) -- (2.64,-.354);

\fill[black!20!white] (.75,0) circle (.3cm);
\draw (.75,0) circle (.3cm);
\draw (.75,0) node{$u_1$};

\fill[black!20!white] (2.25,0) circle (.3cm);
\draw (2.25,0) circle (.3cm);
\draw (2.25,0) node{$v$};

\fill[black!20!white] (3.75,0) circle (.3cm);
\draw (3.75,0) circle (.3cm);
\draw (3.75,0) node{$u_2$};

\fill[black!20!white] (2.25,-1.5) circle (.3cm);
\draw (2.25,-1.5) circle (.3cm);
\draw (2.25,-1.5) node{$u_3$};

\draw[very thick, ->] (6,0) -- (7,0);

\begin{scope}[xshift=8.5cm]

\draw (.75,0) rectangle (2.25,-1.5);
\draw (2.25,.2) -- (3.75,.2);
\draw (2.25,-.2) -- (3.75,-.2);
\draw (.55,-1.5) -- (.55,-3);
\draw (.95,-1.5) -- (.95,-3);

\fill[black!20!white] (.75,0) circle (.3cm);
\draw (.75,0) circle (.3cm);
\draw (.75,0) node{$u_1$};

\fill[black!20!white] (2.25,0) circle (.3cm);
\draw (2.25,0) circle (.3cm);
\draw (2.25,0) node{$v_2$};

\fill[black!20!white] (3.75,0) circle (.3cm);
\draw (3.75,0) circle (.3cm);
\draw (3.75,0) node{$u_2$};

\fill[black!20!white] (.75,-1.5) circle (.3cm);
\draw (.75,-1.5) circle (.3cm);
\draw (.75,-1.5) node{$v_1$};

\fill[black!20!white] (.75,-3) circle (.3cm);
\draw (.75,-3) circle (.3cm);
\draw (.75,-3) node{$u_3$};

\fill[black!20!white] (2.25,-1.5) circle (.3cm);
\draw (2.25,-1.5) circle (.3cm);
\draw (2.25,-1.5) node{$v_3$};

\end{scope}

\end{scope}

\begin{scope}[yshift=-9cm]

\draw (0,.2) -- (3,.2);
\draw (0,-.2) -- (3,-.2);
\draw [bend left] (1.5,.2) edge (4.5,.2);
\draw [bend right] (1.5,-.2) edge (4.5,-.2);
\draw[red] (1,0) -- (2,0);

\fill[black!20!white] (0,0) circle (.3cm);
\draw (0,0) circle (.3cm);
\draw (0,0) node {$u_1$};

\fill[black!20!white] (1.5,0) circle (.3cm);
\draw (1.5,0) circle (.3cm);
\draw (1.5,0) node {$v$};

\fill[black!20!white] (3,0) circle (.3cm);
\draw (3,0) circle (.3cm);
\draw (3,0) node {$u_2$};

\fill[black!20!white] (4.5,0) circle (.3cm);
\draw (4.5,0) circle (.3cm);
\draw (4.5,0) node {$u_3$};

\draw[very thick, ->] (6,0) -- (7,0);

\begin{scope}[xshift=8.5cm]

\draw (2.25,0) -- (0,-1) -- (2.25,-2) -- (1.5,-1) -- (2.25,0) -- (3,-1) -- (2.25,-2) -- (4.5,-1) -- (2.25,0);

\fill[black!20!white] (2.25,0) circle (.3cm);
\draw (2.25,0) circle (.3cm);
\draw (2.25,0) node{$v_1$};

\fill[black!20!white] (0,-1) circle (.3cm);
\draw (0,-1) circle (.3cm);
\draw (0,-1) node {$u_1$};

\fill[black!20!white] (1.5,-1) circle (.3cm);
\draw (1.5,-1) circle (.3cm);
\draw (1.5,-1) node {$v_3$};

\fill[black!20!white] (3,-1) circle (.3cm);
\draw (3,-1) circle (.3cm);
\draw (3,-1) node {$u_2$};

\fill[black!20!white] (4.5,-1) circle (.3cm);
\draw (4.5,-1) circle (.3cm);
\draw (4.5,-1) node {$u_3$};

\fill[black!20!white] (2.25,-2) circle (.3cm);
\draw (2.25,-2) circle (.3cm);
\draw (2.25,-2) node{$v_2$};

\end{scope}
\end{scope}

\end{tikzpicture}$$
\caption{Applying two-path extensions to doubled trees. The red line denotes the partition of the edges incident to $v$ into the sets $A$ and $B$.}
\label{figure:two-pathlemma}
\end{figure}
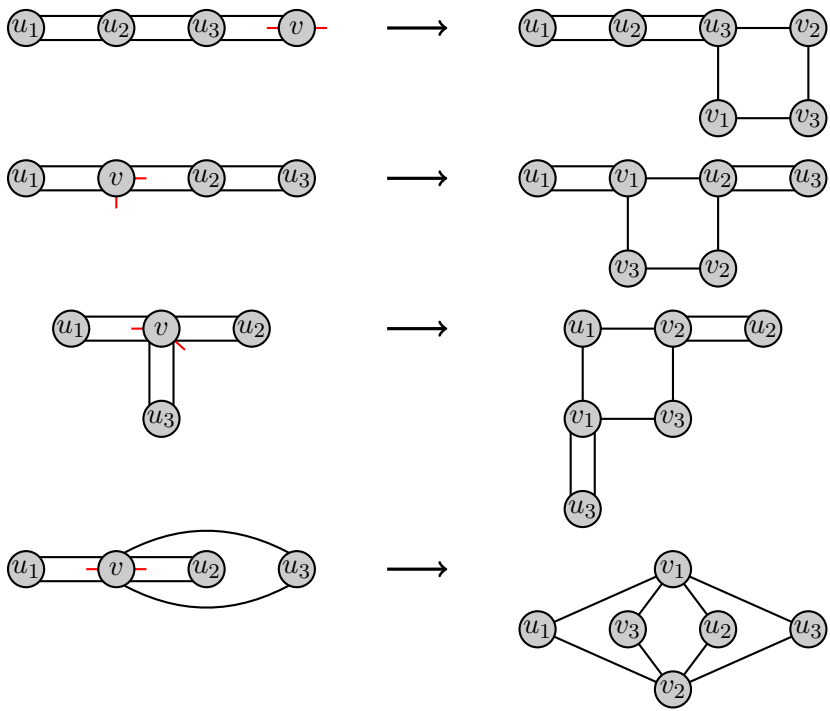

\begin{proof}[Proof of Theorem \ref{thm:genus1}]
If $G$ is a doubled cycle of even length, then it is reduced and Corollary \ref{cor:TuraevAlgorithm} implies that $g_T(G)=1$.

Let $G$ be a reduced alternating decomposition graph with $g_T(G)=1$. Lemma \ref{lemma:degree2} implies $G$ contains a pair of parallel edges $\{e_1,e_2\}$. Let $G'=G-\{e_1,e_2\}$. Since $G$ is reduced $k(G')=k(G)$ and thus $g_T(G')=0$. Because $G$ has no vertices of degree two, it follows that $G'$ has at most two vertices of degree two. Lemma \ref{lemma:zero} implies that $G'$ is a doubled path. Therefore $G$ is a doubled cycle of even length.
\end{proof}

Suppose $L$ is a link containing a two-tangle $T$ inside the ball $B$. A {\em mutation} of $L$ is a link $L'$ obtained by removing the ball $B$, rotating it $180^\circ$ about any of its principle axes, and gluing $B$ back into the link. Two links that are related by a sequence of mutations are said to be mutants of one another.

\begin{proof}[Proof of Corollary \ref{corollary:mutant}]
Since $L$ is Turaev genus one, it has a diagram $D$ as in Figure \ref{fig:genus1}. The alternating decomposition graph of this diagram is $C_{2k}^2$, a doubled cycle of length $2k$. Let $T$ be a the tangle consisting of $T_{i}$ and $T_{i+1}$. Rotating the tangle $T$ by $180^\circ$ in the plane of the diagram results in a new diagram whose alternating decomposition graph is $C_{2k-2}^2$, a doubled cycle of length $2k-2$. See Figure \ref{figure:mutation}. Therefore, through a sequence of mutations, the diagram $D$ can be transformed into a diagram whose alternating decomposition graph is $C_2^2$.
\begin{figure}[h]
$$\begin{tikzpicture}
\draw [bend left] (-2, .8) edge (0,.5);
\draw [bend right] (-2,-.8) edge (0,-.5);
\draw [bend left] (0,.5) edge (3,.5);
\draw [bend right] (0,-.5) edge (3, -.5);
\draw [bend left] (3,.5) edge (6,.5);
\draw [bend right] (3,-.5) edge (6, -.5);
\draw [bend left] (6,.5) edge (9,.5);
\draw [bend right] (6,-.5) edge (9, -.5);
\draw [bend left] (9,.5) edge (11,.8);
\draw [bend right] (9,-.5) edge (11,-.8);

\draw (-1.5,1.2) node{$-$};
\draw (1.5,1.2) node{$+$};
\draw (4.5,1.2) node{$-$};
\draw (7.5,1.2) node{$+$};
\draw (10.5,1.2) node{$-$};

\draw (-1.5,-1.2) node{$+$};
\draw (1.5,-1.2) node{$-$};
\draw (4.5,-1.2) node{$+$};
\draw (7.5,-1.2) node{$-$};
\draw (10.5,-1.2) node{$+$};

\fill[white] (0,0) circle (1cm);
\draw (0,0) node {$T_1$};
\draw (0,0) circle (1cm);
\fill[white] (3,0) circle (1cm);
\draw (3,0) node {$T_2$};
\draw (3,0) circle (1cm);
\fill[white] (6,0) circle (1cm);
\draw (6,0) node {$T_3$};
\draw (6,0) circle (1cm);
\fill[white](9,0) circle (1cm);
\draw (9,0) node {$T_4$};
\draw (9,0) circle (1cm);

\draw[blue, dashed] (4.5,0) ellipse (3 cm and 1.5 cm);

\begin{scope}[yshift = -4cm]
\draw [bend left] (-2, .8) edge (0,.5);
\draw [bend right] (-2,-.8) edge (0,-.5);
\draw [bend left] (0,.5) edge (3,.5);
\draw [bend right] (0,-.5) edge (3, -.5);
\draw [bend left] (3,.5) edge (6,.5);
\draw [bend right] (3,-.5) edge (6, -.5);
\draw [bend left] (6,.5) edge (9,.5);
\draw [bend right] (6,-.5) edge (9, -.5);
\draw [bend left] (9,.5) edge (11,.8);
\draw [bend right] (9,-.5) edge (11,-.8);

\draw (-1.5,1.2) node{$-$};
\draw (4.5,1.2) node{$+$};
\draw (10.5,1.2) node{$-$};

\draw (-1.5,-1.2) node{$+$};
\draw (4.5,-1.2) node{$-$};
\draw (10.5,-1.2) node{$+$};

\fill[white] (0,0) circle (1cm);
\draw (0,0) node {$T_1$};
\draw (0,0) circle (1cm);
\fill[white] (3,0) circle (1cm);
\draw (3,0) node {$\rotatebox[origin=c]{180}{$T_3$}$};
\draw (3,0) circle (1cm);
\fill[white] (6,0) circle (1cm);
\draw (6,0) node {$\rotatebox[origin=c]{180}{$T_2$}$};
\draw (6,0) circle (1cm);
\fill[white](9,0) circle (1cm);
\draw (9,0) node {$T_4$};
\draw (9,0) circle (1cm);

\draw[blue, dashed] (4.5,0) ellipse (3 cm and 1.5 cm);
\end{scope}

\end{tikzpicture}$$
\caption{The $2$-tangle in the upper diagram is rotated $180^\circ$ to obtain the lower diagram. In the lower diagram, the $2$-tangle containing $T_1$ and $T_3$ and the $2$-tangle containing $T_2$ and $T_4$ are alternating.}
\label{figure:mutation}
\end{figure}
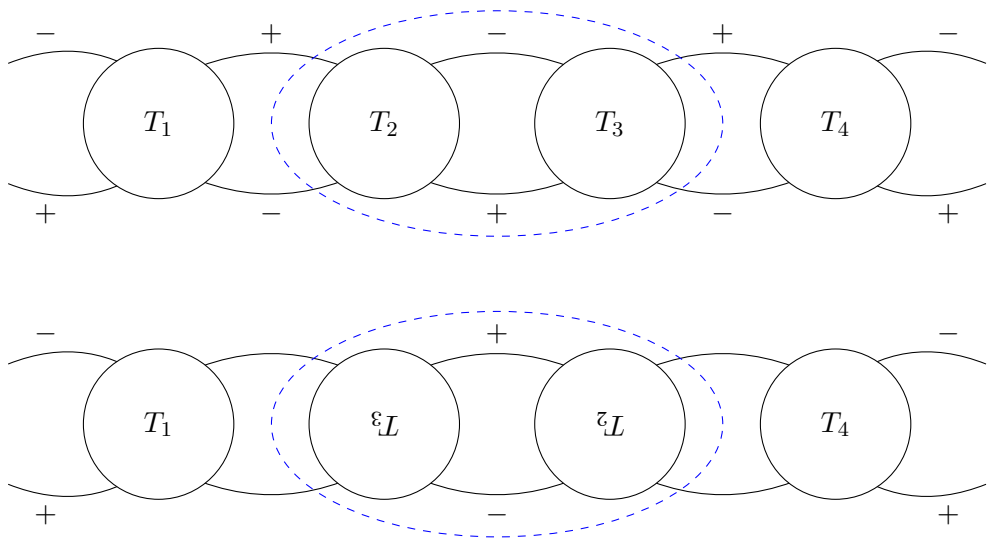

It remains to show that any diagram $D'$ with alternating decomposition graph $C_2^2$ is an almost-alternating link. We may assume that $D'$ consists of two alternating $2$-tangles $T_1$ and $T_2$ connected together by four non-alternating edges. If one of those non-alternating edges is pulled over the tangle $T_1$ as in Figure \ref{figure:almostalternating}, then the resulting diagram is almost-alternating.

\begin{figure}[h]
$$\begin{tikzpicture}
\draw (0,.25) -- (4,.25);
\draw (0,.75) -- (4,.75);
\draw (0,-.25) -- (4,-.25);
\draw (0,-.75) -- (4,-.75);

\fill[white] (0,0) circle (1cm);
\draw (0,0) node {$T_1$};
\draw (0,0) circle (1cm);
\fill[white] (4,0) circle (1cm);
\draw (4,0) node{$T_2$};
\draw (4,0) circle (1cm);

\draw (.3,.75) node{$-$};
\draw (.6,.25) node{$+$};
\draw (.6,-.25) node{$-$};
\draw (.3,-.75) node{$+$};

\draw (3.7,.75) node{$-$};
\draw (3.4, .25) node{$+$};
\draw (3.4, -.25) node{$-$};
\draw (3.7,-.75) node{$+$};

\begin{scope}[xshift = 7cm]

\draw [rounded corners = 2.5mm] (0,.75) -- (1.5,.75) -- (1.5,1.25) -- (-1.25,1.25) -- (-1.25, -1.25) -- (2, -1.25) -- (2,.75) -- (4,.75);

\draw (0,.25) -- (1.9,.25);
\draw (2.1,.25) -- (4,.25);
\draw (0,-.25) -- (1.9,-.25);
\draw (2.1,.-.25) -- (4,-.25);
\draw (0,-.75) -- (1.9,-.75);
\draw (2.1,-.75) -- (4,-.75);

\fill[white] (0,0) circle (1cm);
\draw (0,0) node {$T_1$};
\draw (0,0) circle (1cm);
\fill[white] (4,0) circle (1cm);
\draw (4,0) node{$T_2$};
\draw (4,0) circle (1cm);

\draw (.3,.75) node{$-$};
\draw (.6,.25) node{$+$};
\draw (.6,-.25) node{$-$};
\draw (.3,-.75) node{$+$};

\draw (3.7,.75) node{$-$};
\draw (3.4, .25) node{$+$};
\draw (3.4, -.25) node{$-$};
\draw (3.7,-.75) node{$+$};

\draw [very thick, red] (2,-.25) circle (.3cm);

\end{scope}

\end{tikzpicture}$$
\caption{A diagram with alternating decomposition graph $C_2^2$ is transformed into an almost-alternating diagram by pulling one of the non-alternating edges over one of the tangles. If the crossing inside the red circle is changed, then the diagram will be alternating.}
\label{figure:almostalternating}
\end{figure}
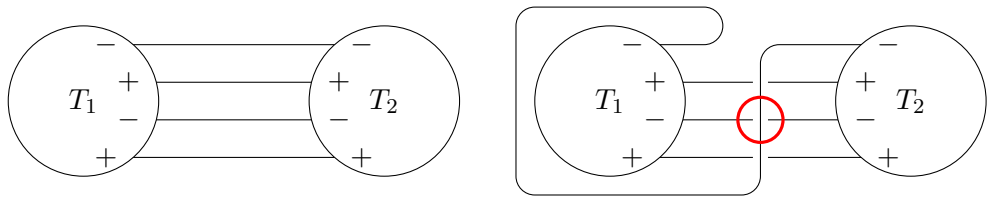
\end{proof}

Many Turaev genus one links are known to be almost-alternating. Kim and Lee \cite{KimLee:Pretzel} show that non-alternating, three-stranded pretzel links are almost-alternating. If each tangle $T_i$ in Figure \ref{fig:genus1} is a rational tangle, then the link $L$ is called a {\em Montesinos link}. In an appendix to \cite{AbeKish:Dealternating}, Jong shows that non-alternating Montesinos links are almost-alternating. Non-alternating Montesinos links include non-alternating pretzel links on arbitrarily many strands. The manipulation of Figure \ref{figure:almostalternating} is a key step in Jong's work. All almost-alternating links are Turaev genus one, but it remains open whether all Turaev genus one links are almost-alternating.

\begin{proof}[Proof of Theorem \ref{thm:genus2}]
Suppose $G\in \{C_2^2\sqcup C_2^2, C_2^2\oplus_1 C_2^2, C^2_{1,1,1}, K_4(2,2), K_4(2)\oplus_2 K_4(2)\}$. Corollary \ref{cor:TuraevAlgorithm} implies that $g_T(G)=2$. Proposition \ref{prop:doubledpath} implies that any alternating decomposition graph that is doubled path equivalent to $G$ also has Turaev genus two. 

Let $G$ be a reduced alternating decomposition graph with $g_T(G)=2$. Since $G$ is reduced and $g_T(G)=2$, it follows that $G$ contains a pair of parallel edges $\{e_1,e_2\}$ such that  $g_T(G')=1$ where $G'=G-\{e_1,e_2\}$. The graph $G'$ has at most two vertices of degree two. Lemma \ref{lemma:degree2} implies that $G'$ contains at least one pair of parallel edges. If the deletion of every pair of parallel edges in $G'$ increased the number of components of $G'$, then every pair could be contracted to obtain the graph $\widetilde{G}'$. Then $g_T(\widetilde{G}')=g_T(G') = 1$, and the graph $\widetilde{G}'$ has at most two vertices of degree two and no pairs of parallel edges. Hence Lemma \ref{lemma:degree2} implies $\widetilde{G}'$ has no edges, which contradicts $g_T(\widetilde{G}')=1$. Thus $G'$ contains a pair of parallel edges $\{e_3,e_4\}$ such that their deletion results in a graph with no more components.

Let $G''=G-\{e_1,e_2,e_3,e_4\}$. Since $G''$ is an alternating decomposition graph of Turaev genus zero with at most four vertices of degree two, it is one of the graphs in Lemma \ref{lemma:zero}. It remains to show that if $G''$ is one of the graphs in Lemma \ref{lemma:zero}, $G$ can be obtained from $G''$ by adding two pairs of parallel edges, and $G$ is a reduced alternating decomposition graph of Turaev genus two, then $G$ is doubled path equivalent to one of the five graphs in the statement of the theorem.

Suppose that $G''$ is a disjoint union of two doubled paths. Then $G''$ has four vertices $v_1$, $v_2$, $v_3$, and $v_4$ of degree two, and thus each pair of parallel edges added to $G''$ must connect two of the degree two vertices. There are two ways to add these parallel edges, one that results in a disjoint union of two doubled cycles and the other that results in a single doubled cycle. However, a doubled cycle only has Turaev genus one, and so $G$ must be a disjoint union of two doubled cycles, i.e. $G$ is doubled path equivalent to $C_2^2\sqcup C_2^2$. See Figure \ref{figure:genus2.1}.
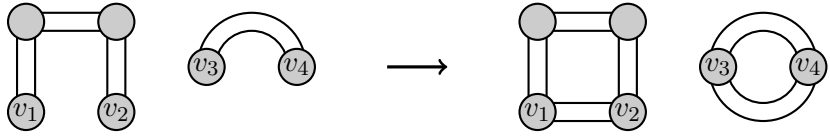
\begin{figure}[h]
$$\begin{tikzpicture}[thick, scale = .8]

\draw (-.15,0) -- (-.15,1.65) -- (1.65,1.65) -- (1.65,0);
\draw (.15,0) -- (.15,1.35) -- (1.35,1.35) -- (1.35,0);

\fill[black!20!white] (0,0) circle (.3cm);
\fill[black!20!white] (0,1.5) circle (.3cm);
\fill[black!20!white] (1.5,1.5) circle (.3cm);
\fill[black!20!white] (1.5,0) circle (.3cm);
\draw (0,0) circle (.3cm);
\draw (0,1.5) circle (.3cm);
\draw (1.5,1.5) circle (.3cm);
\draw (1.5,0) circle (.3cm);
\draw (0,0) node{$v_1$};
\draw (1.5,0) node{$v_2$};

\draw (2.85,.75) arc (180:0:.9cm);
\draw (3.15,.75) arc (180:0:.6cm);

\fill[black!20!white] (3,.75) circle (.3cm);
\fill[black!20!white] (4.5,.75) circle (.3cm);
\draw (3,.75) circle (.3cm);
\draw (4.5,.75) circle (.3cm);
\draw (3,.75) node{$v_3$};
\draw (4.5,.75) node{$v_4$};

\draw[very thick, ->] (6,.75) -- (7,.75);

\begin{scope}[xshift = 8.5cm]
\draw (-.15,-.15) rectangle (1.65,1.65);
\draw (.15,.15) rectangle (1.35,1.35);

\fill[black!20!white] (0,0) circle (.3cm);
\fill[black!20!white] (0,1.5) circle (.3cm);
\fill[black!20!white] (1.5,1.5) circle (.3cm);
\fill[black!20!white] (1.5,0) circle (.3cm);
\draw (0,0) circle (.3cm);
\draw (0,1.5) circle (.3cm);
\draw (1.5,1.5) circle (.3cm);
\draw (1.5,0) circle (.3cm);
\draw (0,0) node{$v_1$};
\draw (1.5,0) node{$v_2$};

\draw (3.75,.75) circle (.9cm);
\draw (3.75,.75) circle (.6cm);

\fill[black!20!white] (3,.75) circle (.3cm);
\fill[black!20!white] (4.5,.75) circle (.3cm);
\draw (3,.75) circle (.3cm);
\draw (4.5,.75) circle (.3cm);
\draw (3,.75) node{$v_3$};
\draw (4.5,.75) node{$v_4$};
\end{scope}
\end{tikzpicture}$$
\caption{If $G''$ is a disjoint union of two doubled paths, then $G$ is a disjoint union of two doubled cycles of even length.}
\label{figure:genus2.1}
\end{figure}

Suppose that $G''$ is a doubled path where $v_1$ and $v_2$ are its degree two vertices. If one adds a pair of parallel edges connecting $v_1$ and $v_2$, then adds a pair of parallel edges anywhere else to obtain $G$, then $G$ is doubled path equivalent to $C_{1,1,1}^2$. If one adds a pair of parallel edges connecting $v_1$ and some other vertex $u_1$ and a pair of parallel edges connecting $v_2$ and some other vertex $u_2$ to obtain $G$, then there are three possibilities for $G$. If $u_1$ is between $v_1$ and $u_2$, then $G$ is not reduced. If $u_1=u_2$, then $G$ is doubled path equivalent to $C_2^2\oplus_1 C_2^2$. If $u_2$ is between $v_1$ and $u_1$, then $G$ is doubled path equivalent to $C_{1,1,1}^2$. See Figure \ref{figure:genus2.2}.
\begin{figure}[h]
$$\begin{tikzpicture}[scale=.8,thick]
\begin{scope}[yshift=-4cm]
\begin{scope}[xshift = .5cm]
\draw (1,.2) -- (0,1.7);
\draw (1,-.2) -- (0,1.3);
\draw (1,2.8) -- (0,1.3);
\draw (1,3.2) -- (0,1.7);
\draw (1,3.1) -- (2.5,3.1);
\draw (1,2.9) -- (2.5,2.9);
\draw (2.5,3.2) -- (3.5,1.7);
\draw (2.5,2.8) -- (3.5,1.3);
\draw (3.5,1.7) -- (2.5,.2);
\draw (3.5,1.3) -- (2.5,-.2);

\fill[black!20!white] (0,1.5) circle (.3cm);
\fill[black!20!white] (1,0) circle (.3cm);
\fill[black!20!white] (1,3) circle (.3cm);
\fill[black!20!white] (2.5,0) circle (.3cm);
\fill[black!20!white] (2.5,3) circle (.3cm);
\fill[black!20!white] (3.5,1.5) circle (.3cm);

\draw (0,1.5) circle (.3cm);
\draw (1,0) circle (.3cm);
\draw (1,3) circle (.3cm);
\draw (2.5,0) circle (.3cm);
\draw (2.5,3) circle (.3cm);
\draw (3.5,1.5) circle (.3cm);

\draw (1,0) node{$v_1$};
\draw (2.5,0) node{$v_2$};

\end{scope}

\draw[very thick, ->] (6,.75) -- (7,.75);

\begin{scope}[xshift = 9cm]

\draw (1,.2) -- (0,1.7);
\draw (1,-.2) -- (0,1.3);
\draw (1,2.8) -- (0,1.3);
\draw (1,3.2) -- (0,1.7);
\draw (1,3.1) -- (2.5,3.1);
\draw (1,2.9) -- (2.5,2.9);
\draw (2.5,3.2) -- (3.5,1.7);
\draw (2.5,2.8) -- (3.5,1.3);
\draw (3.5,1.7) -- (2.5,.2);
\draw (3.5,1.3) -- (2.5,-.2);
\draw (1,.1) -- (2.5,.1);
\draw (1,-.1) -- (2.5,-.1);
\draw (0,1.6) -- (3.5,1.6);
\draw (0,1.4) -- (3.5,1.4);

\fill[black!20!white] (0,1.5) circle (.3cm);
\fill[black!20!white] (1,0) circle (.3cm);
\fill[black!20!white] (1,3) circle (.3cm);
\fill[black!20!white] (2.5,0) circle (.3cm);
\fill[black!20!white] (2.5,3) circle (.3cm);
\fill[black!20!white] (3.5,1.5) circle (.3cm);

\draw (0,1.5) circle (.3cm);
\draw (1,0) circle (.3cm);
\draw (1,3) circle (.3cm);
\draw (2.5,0) circle (.3cm);
\draw (2.5,3) circle (.3cm);
\draw (3.5,1.5) circle (.3cm);

\draw (1,0) node{$v_1$};
\draw (2.5,0) node{$v_2$};
\end{scope}

\end{scope}

\begin{scope}[yshift=-8cm]

\draw (1,3.2) -- (0,1.7) -- (1,.2) -- (2,1.7) -- (3,.2) -- (4,1.7) -- (3,3.2);
\draw (1,2.8) -- (0,1.3) -- (1,-.2) -- (2,1.3) -- (3,-.2) -- (4,1.3) -- (3,2.8);

\fill[black!20!white] (0,1.5) circle (.3cm);
\fill[black!20!white] (1,0) circle (.3cm);
\fill[black!20!white] (1,3) circle (.3cm);
\fill[black!20!white] (2,1.5) circle (.3cm);
\fill[black!20!white] (3,0) circle (.3cm);
\fill[black!20!white] (3,3) circle (.3cm);
\fill[black!20!white] (4,1.5) circle (.3cm);

\draw (0,1.5) circle (.3cm);
\draw (1,0) circle (.3cm);
\draw (1,3) circle (.3cm);
\draw (2,1.5) circle (.3cm);
\draw (3,0) circle (.3cm);
\draw (3,3) circle (.3cm);
\draw (4,1.5) circle (.3cm);

\draw (1,3) node{$v_1$};
\draw (3,3) node{$v_2$};
\draw (2,1.5) node{$u_1$};

\draw[very thick, ->] (6,.75) -- (7,.75);

\begin{scope}[xshift=9cm]
\draw (1,3.2) -- (0,1.7) -- (1,.2) -- (2,1.7) -- (3,.2) -- (4,1.7) -- (3,3.2) -- (2,1.7) -- (1,3.2);
\draw (1,2.8) -- (0,1.3) -- (1,-.2) -- (2,1.3) -- (3,-.2) -- (4,1.3) -- (3,2.8) -- (2,1.3) -- (1,2.8);

\fill[black!20!white] (0,1.5) circle (.3cm);
\fill[black!20!white] (1,0) circle (.3cm);
\fill[black!20!white] (1,3) circle (.3cm);
\fill[black!20!white] (2,1.5) circle (.3cm);
\fill[black!20!white] (3,0) circle (.3cm);
\fill[black!20!white] (3,3) circle (.3cm);
\fill[black!20!white] (4,1.5) circle (.3cm);

\draw (0,1.5) circle (.3cm);
\draw (1,0) circle (.3cm);
\draw (1,3) circle (.3cm);
\draw (2,1.5) circle (.3cm);
\draw (3,0) circle (.3cm);
\draw (3,3) circle (.3cm);
\draw (4,1.5) circle (.3cm);

\draw (1,3) node{$v_1$};
\draw (3,3) node{$v_2$};
\draw (2,1.5) node{$u_1$};
\end{scope}

\end{scope}


\begin{scope}[yshift=-11cm, xshift = 1cm]
\draw (-.1,1.5) -- (-.1,-.1) -- (1.6,-.1) -- (1.6,1.4) -- (2.9,1.4) -- (2.9,.1);
\draw (.1,1.5) -- (.1,.1) -- (1.4,.1) -- (1.4,1.6) -- (3.1,1.6) -- (3.1,-.1);

\fill[black!20!white] (0,0) circle (.3cm);
\fill[black!20!white] (1.5,0) circle (.3cm);
\fill[black!20!white] (0,1.5) circle (.3cm);
\fill[black!20!white] (1.5,1.5) circle (.3cm);
\fill[black!20!white] (3,0) circle (.3cm);
\fill[black!20!white] (3,1.5) circle (.3cm);

\draw (0,0) circle (.3cm);
\draw (1.5,0) circle (.3cm);
\draw (0,1.5) circle (.3cm);
\draw (1.5,1.5) circle (.3cm);
\draw (3,0) circle (.3cm);
\draw (3,1.5) circle (.3cm);

\draw (0,1.5) node{$v_1$};
\draw (1.5,1.5) node{$u_1$};
\draw (1.5,0) node{$u_2$};
\draw (3,0) node{$v_2$};

\draw[very thick, ->] (5,.75) -- (6,.75);

\begin{scope}[xshift = 8 cm]
\draw (1.5,1.6) -- (-.1,1.6) -- (-.1,-.1) -- (1.6,-.1) -- (1.6,1.4) -- (2.9,1.4) -- (2.9,.1) -- (1.5,.1);
\draw (1.5,1.4) -- (.1,1.4) -- (.1,.1) -- (1.4,.1) -- (1.4,1.6) -- (3.1,1.6) -- (3.1,-.1) -- (1.5,-.1);

\fill[black!20!white] (0,0) circle (.3cm);
\fill[black!20!white] (1.5,0) circle (.3cm);
\fill[black!20!white] (0,1.5) circle (.3cm);
\fill[black!20!white] (1.5,1.5) circle (.3cm);
\fill[black!20!white] (3,0) circle (.3cm);
\fill[black!20!white] (3,1.5) circle (.3cm);

\draw (0,0) circle (.3cm);
\draw (1.5,0) circle (.3cm);
\draw (0,1.5) circle (.3cm);
\draw (1.5,1.5) circle (.3cm);
\draw (3,0) circle (.3cm);
\draw (3,1.5) circle (.3cm);

\draw (0,1.5) node{$v_1$};
\draw (1.5,1.5) node{$u_1$};
\draw (1.5,0) node{$u_2$};
\draw (3,0) node{$v_2$};
\end{scope}

\end{scope}
\end{tikzpicture}$$
\caption{If $G''$ is a doubled path, then $G$ is doubled path equivalent to either $C_2^2\oplus_1 C_2^2$ or $C^2_{1,1,1}$.}
\label{figure:genus2.2}
\end{figure}
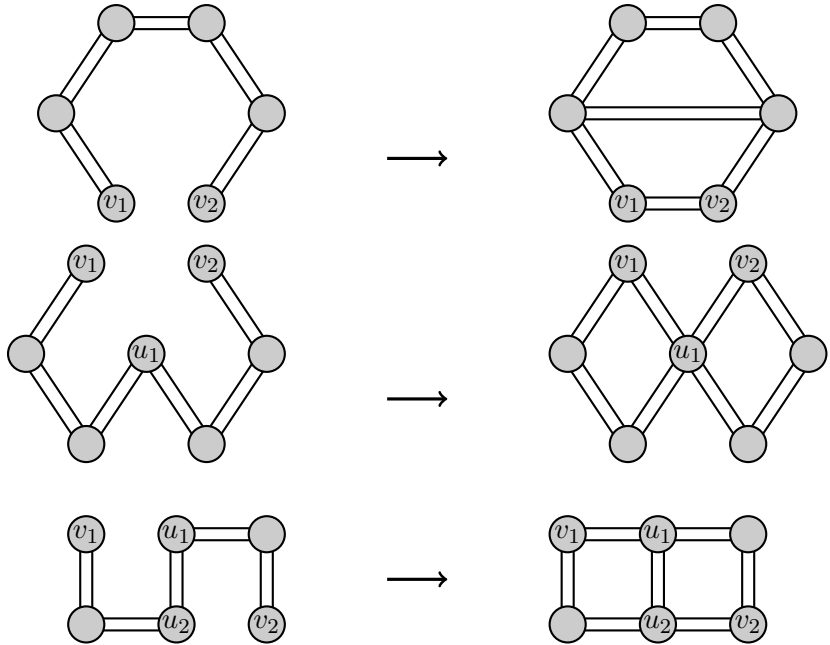

Suppose that $G''$ is a doubled tree with three vertices $v_1$, $v_2$, and $v_3$ of degree two. Let $v$ be the unique vertex in $G''$ of degree six. Since $G''$ contains three vertices of degree two, it follows that two of those vertices must be connected by a pair of parallel edges in $G$. Without loss of generality, assume we add a pair of parallel edges connecting $v_1$ and $v_2$. Also suppose that we add the other pair of parallel edges connecting $v_3$ and some other vertex $u$. If $v$ is between $u$ and $v_3$, then $G$ is doubled path equivalent to $C_{1,1,1}^2$. If $u=v$, then $G$ is doubled path equivalent to $C_2^2\oplus_2 C_2^2$. If $u$ is between $v$ and $v_3$, then $G$ is not reduced. See Figure \ref{figure:genus2.3}.
\begin{figure}[h]
$$\begin{tikzpicture}[scale=.8, thick]
\draw (0,-.1) -- (1.5,-.1); 
\draw (0,.1) -- (1.5,.1);
\draw(1.6,1.4) -- (2.9,1.4) -- (2.9,.1);
\draw (1.4,1.6) -- (3.1,1.6) -- (3.1,-.1);
\draw (1.4,0) -- (1.4,1.4) -- (0,1.4);
\draw (1.6,0) -- (1.6,1.6) -- (0,1.6);

\fill[black!20!white] (0,0) circle (.3cm);
\fill[black!20!white] (1.5,0) circle (.3cm);
\fill[black!20!white] (0,1.5) circle (.3cm);
\fill[black!20!white] (1.5,1.5) circle (.3cm);
\fill[black!20!white] (3,0) circle (.3cm);
\fill[black!20!white] (3,1.5) circle (.3cm);

\draw (0,0) circle (.3cm);
\draw (1.5,0) circle (.3cm);
\draw (0,1.5) circle (.3cm);
\draw (1.5,1.5) circle (.3cm);
\draw (3,0) circle (.3cm);
\draw (3,1.5) circle (.3cm);

\draw (0,0) node{$v_1$};
\draw (0,1.5) node{$v_2$};
\draw (3,0) node{$v_3$};
\draw (1.5,0) node{$u$};

\draw[very thick, ->] (5,.75) -- (6,.75);

\begin{scope}[xshift = 8 cm]
\draw (1.5,1.6) -- (-.1,1.6) -- (-.1,-.1) -- (1.6,-.1) -- (1.6,1.4) -- (2.9,1.4) -- (2.9,.1) -- (1.5,.1);
\draw (1.5,1.4) -- (.1,1.4) -- (.1,.1) -- (1.4,.1) -- (1.4,1.6) -- (3.1,1.6) -- (3.1,-.1) -- (1.5,-.1);

\fill[black!20!white] (0,0) circle (.3cm);
\fill[black!20!white] (1.5,0) circle (.3cm);
\fill[black!20!white] (0,1.5) circle (.3cm);
\fill[black!20!white] (1.5,1.5) circle (.3cm);
\fill[black!20!white] (3,0) circle (.3cm);
\fill[black!20!white] (3,1.5) circle (.3cm);

\draw (0,0) circle (.3cm);
\draw (1.5,0) circle (.3cm);
\draw (0,1.5) circle (.3cm);
\draw (1.5,1.5) circle (.3cm);
\draw (3,0) circle (.3cm);
\draw (3,1.5) circle (.3cm);
\draw (0,0) node{$v_1$};
\draw (0,1.5) node{$v_2$};
\draw (3,0) node{$v_3$};
\draw (1.5,0) node{$u$};
\end{scope}

\begin{scope}[yshift=-4cm,xshift=-1cm]

\draw (0,1.7) -- (1,.2) -- (2,1.7) -- (3,.2) -- (4,1.7) -- (3,3.2);
\draw (0,1.3) -- (1,-.2) -- (2,1.3) -- (3,-.2) -- (4,1.3) -- (3,2.8);
\draw (2,1.7) -- (1,3.2);
\draw (2,1.3) -- (1,2.8);

\fill[black!20!white] (0,1.5) circle (.3cm);
\fill[black!20!white] (1,0) circle (.3cm);
\fill[black!20!white] (1,3) circle (.3cm);
\fill[black!20!white] (2,1.5) circle (.3cm);
\fill[black!20!white] (3,0) circle (.3cm);
\fill[black!20!white] (3,3) circle (.3cm);
\fill[black!20!white] (4,1.5) circle (.3cm);

\draw (0,1.5) circle (.3cm);
\draw (1,0) circle (.3cm);
\draw (1,3) circle (.3cm);
\draw (2,1.5) circle (.3cm);
\draw (3,0) circle (.3cm);
\draw (3,3) circle (.3cm);
\draw (4,1.5) circle (.3cm);

\draw (0,1.5) node{$v_1$};
\draw (1,3) node{$v_2$};
\draw (2,1.5) node{$u$};
\draw (3,3) node{$v_3$};

\draw[very thick, ->] (6,.75) -- (7,.75);

\begin{scope}[xshift=9cm]
\draw (1,3.2) -- (0,1.7) -- (1,.2) -- (2,1.7) -- (3,.2) -- (4,1.7) -- (3,3.2) -- (2,1.7) -- (1,3.2);
\draw (1,2.8) -- (0,1.3) -- (1,-.2) -- (2,1.3) -- (3,-.2) -- (4,1.3) -- (3,2.8) -- (2,1.3) -- (1,2.8);

\fill[black!20!white] (0,1.5) circle (.3cm);
\fill[black!20!white] (1,0) circle (.3cm);
\fill[black!20!white] (1,3) circle (.3cm);
\fill[black!20!white] (2,1.5) circle (.3cm);
\fill[black!20!white] (3,0) circle (.3cm);
\fill[black!20!white] (3,3) circle (.3cm);
\fill[black!20!white] (4,1.5) circle (.3cm);

\draw (0,1.5) circle (.3cm);
\draw (1,0) circle (.3cm);
\draw (1,3) circle (.3cm);
\draw (2,1.5) circle (.3cm);
\draw (3,0) circle (.3cm);
\draw (3,3) circle (.3cm);
\draw (4,1.5) circle (.3cm);
\draw (0,1.5) node{$v_1$};
\draw (1,3) node{$v_2$};
\draw (2,1.5) node{$u$};
\draw (3,3) node{$v_3$};
\end{scope}

\end{scope}

\end{tikzpicture}$$
\caption{If $G''$ is a doubled tree with three vertices of degree two, then $G$ is doubled path equivalent to either $C^2_{1,1,1}$ or $C_2^2\oplus_1 C_2^2$.}
\label{figure:genus2.3}
\end{figure}
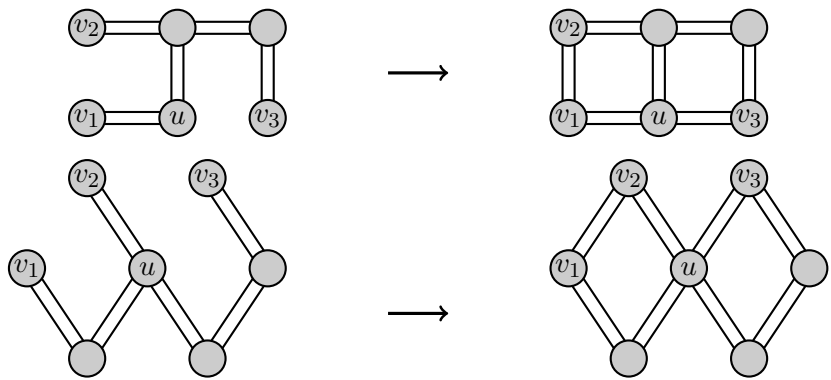

Suppose that $G''$ is a doubled tree with four vertices of degree two. Then one must add one pair of parallel edges connecting two of the degree two vertices and another pair of parallel edges connecting the other two of the degree two vertices. Furthermore $G''$ either contains two vertices of degree six or one vertex of degree eight. If $G''$ contains two vertices of degree six, then $G$ is either not reduced or doubled path equivalent to $C^2_{1,1,1}$. If $G''$ contains a vertex of degree eight, then $G$ is doubled path equivalent to $C_2^2\oplus_1 C_2^2$. See Figure \ref{figure:genus2.4}.

\begin{figure}[h]
$$\begin{tikzpicture}[scale=.8, thick]
\draw (0,-.1) -- (3,-.1);
\draw (0,.1) -- (3,.1);
\draw (0,1.6) -- (3,1.6);
\draw (0,1.4) -- (3,1.4);
\draw (1.4,0) -- (1.4,1.5);
\draw (1.6,0) -- (1.6,1.5);

\fill[black!20!white] (0,0) circle (.3cm);
\fill[black!20!white] (1.5,0) circle (.3cm);
\fill[black!20!white] (0,1.5) circle (.3cm);
\fill[black!20!white] (1.5,1.5) circle (.3cm);
\fill[black!20!white] (3,0) circle (.3cm);
\fill[black!20!white] (3,1.5) circle (.3cm);

\draw (0,0) circle (.3cm);
\draw (1.5,0) circle (.3cm);
\draw (0,1.5) circle (.3cm);
\draw (1.5,1.5) circle (.3cm);
\draw (3,0) circle (.3cm);
\draw (3,1.5) circle (.3cm);

\draw (0,0) node{$v_1$};
\draw (0,1.5) node{$v_2$};
\draw (3,0) node{$v_3$};
\draw (3,1.5) node{$v_4$};

\draw[very thick, ->] (5,.75) -- (6,.75);

\begin{scope}[xshift = 8 cm]
\draw (1.5,1.6) -- (-.1,1.6) -- (-.1,-.1) -- (1.6,-.1) -- (1.6,1.4) -- (2.9,1.4) -- (2.9,.1) -- (1.5,.1);
\draw (1.5,1.4) -- (.1,1.4) -- (.1,.1) -- (1.4,.1) -- (1.4,1.6) -- (3.1,1.6) -- (3.1,-.1) -- (1.5,-.1);

\fill[black!20!white] (0,0) circle (.3cm);
\fill[black!20!white] (1.5,0) circle (.3cm);
\fill[black!20!white] (0,1.5) circle (.3cm);
\fill[black!20!white] (1.5,1.5) circle (.3cm);
\fill[black!20!white] (3,0) circle (.3cm);
\fill[black!20!white] (3,1.5) circle (.3cm);

\draw (0,0) circle (.3cm);
\draw (1.5,0) circle (.3cm);
\draw (0,1.5) circle (.3cm);
\draw (1.5,1.5) circle (.3cm);
\draw (3,0) circle (.3cm);
\draw (3,1.5) circle (.3cm);
\draw (0,0) node{$v_1$};
\draw (0,1.5) node{$v_2$};
\draw (3,0) node{$v_3$};
\draw (3,1.5) node{$v_4$};
\end{scope}

\begin{scope}[yshift=-4cm,xshift=-1cm]

\draw (0,1.7) -- (1,.2) -- (2,1.7) -- (3,.2) -- (4,1.7);
\draw (0,1.3) -- (1,-.2) -- (2,1.3) -- (3,-.2) -- (4,1.3);
\draw (2,1.7) -- (1,3.2);
\draw (2,1.3) -- (1,2.8);
\draw (2,1.7) -- (3,3.2);
\draw (2,1.3) -- (3,2.8);

\fill[black!20!white] (0,1.5) circle (.3cm);
\fill[black!20!white] (1,0) circle (.3cm);
\fill[black!20!white] (1,3) circle (.3cm);
\fill[black!20!white] (2,1.5) circle (.3cm);
\fill[black!20!white] (3,0) circle (.3cm);
\fill[black!20!white] (3,3) circle (.3cm);
\fill[black!20!white] (4,1.5) circle (.3cm);

\draw (0,1.5) circle (.3cm);
\draw (1,0) circle (.3cm);
\draw (1,3) circle (.3cm);
\draw (2,1.5) circle (.3cm);
\draw (3,0) circle (.3cm);
\draw (3,3) circle (.3cm);
\draw (4,1.5) circle (.3cm);

\draw (0,1.5) node{$v_1$};
\draw (1,3) node{$v_2$};
\draw (3,3) node{$v_3$};
\draw (4,1.5) node{$v_4$};

\draw[very thick, ->] (6,.75) -- (7,.75);

\begin{scope}[xshift=9cm]
\draw (1,3.2) -- (0,1.7) -- (1,.2) -- (2,1.7) -- (3,.2) -- (4,1.7) -- (3,3.2) -- (2,1.7) -- (1,3.2);
\draw (1,2.8) -- (0,1.3) -- (1,-.2) -- (2,1.3) -- (3,-.2) -- (4,1.3) -- (3,2.8) -- (2,1.3) -- (1,2.8);

\fill[black!20!white] (0,1.5) circle (.3cm);
\fill[black!20!white] (1,0) circle (.3cm);
\fill[black!20!white] (1,3) circle (.3cm);
\fill[black!20!white] (2,1.5) circle (.3cm);
\fill[black!20!white] (3,0) circle (.3cm);
\fill[black!20!white] (3,3) circle (.3cm);
\fill[black!20!white] (4,1.5) circle (.3cm);

\draw (0,1.5) circle (.3cm);
\draw (1,0) circle (.3cm);
\draw (1,3) circle (.3cm);
\draw (2,1.5) circle (.3cm);
\draw (3,0) circle (.3cm);
\draw (3,3) circle (.3cm);
\draw (4,1.5) circle (.3cm);

\draw (0,1.5) node{$v_1$};
\draw (1,3) node{$v_2$};
\draw (3,3) node{$v_3$};
\draw (4,1.5) node{$v_4$};

\end{scope}

\end{scope}

\end{tikzpicture}$$
\caption{If $G''$ is a doubled tree with four vertices of degree two, then $G$ is doubled path equivalent to either $C^2_{1,1,1}$ or $C_2^2\oplus_1 C_2^2$.}
\label{figure:genus2.4}
\end{figure}
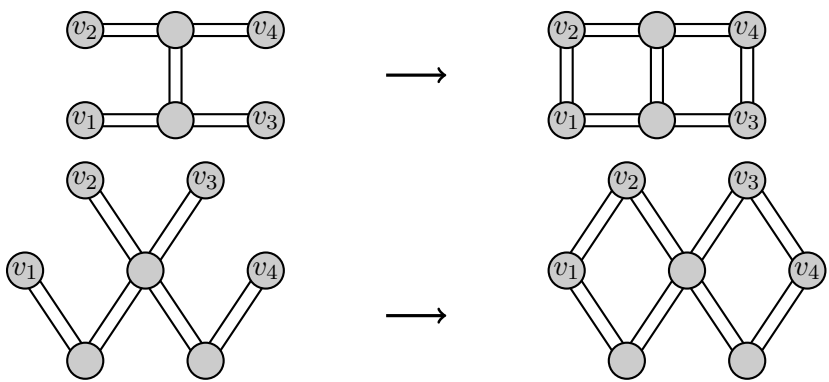

Suppose that $G''=C_4(p,q,r,s)$ for some non-negative integers $p$, $q$, $r$, and $s$. Since $G''$ has four vertices of degree two, each pair of parallel edges added to $G''$ must connect two of the degree two vertices. The resulting graph is $K_4(\widetilde{p},\widetilde{q})$ for some values of $\widetilde{p}$ and $\widetilde{q}$. Thus $G$ is doubled path equivalent to $K_4(2,2)$.

Suppose that $G''=\widetilde{K}_4(p,q)\oplus_2\widetilde{K}_4 (r,s)$ for some non-negative integers $p$, $q$, $r$, and $s$. Since $G''$ has four vertices of degree two, each pair of parallel edges added to $G''$ must connect two of the degree two vertices. The resulting graph is $K_4(\widetilde{p})\oplus_2 K_4(\widetilde{q})$ for some values of $\widetilde{p}$ and $\widetilde{q}$. Thus $G$ is doubled path equivalent to $K_4(2)\oplus_2 K_4(2)$.

Hence if $G$ is a reduced alternating decomposition graph with $g_T(G)=2$, then $G$ is doubled path equivalent to one of $C_2^2\sqcup C_2^2$,  $C_2^2\oplus_1 C_2^2$, $C^2_{1,1,1}$, $K_4(2,2)$, or $K_4(2)\oplus_2 K_4(2)$.
\end{proof}

Suppose that $G$ has $v(G)$ vertices, $e(G)$ edges, and $k(G)$ components. The {\em nullity} $n(G)$ of $G$ is defined as 
$$n(G) = e(G) - v(G) + k(G).$$
One can equivalently define the nullity of $G$ to be the nullity of the incidence matrix of $G$ or to be the number of edges not in a maximal spanning forest of $G$. The simplification $\operatorname{si}(G)$ of the graph $G$ is the  graph obtained from $G$ by deleting loops and replacing each set of multiple edges connecting two distinct vertices $v_1$ and $v_2$ with a single edge connecting $v_1$ and $v_2$. As long as an alternating decomposition graph $G$ does not have any vertices of degree two, its Turaev genus is bounded below by the nullity of the simplification of $G$ in the following manner.

\begin{proposition}
Let $G$ be an alternating decomposition graph, and let $\operatorname{si}(G)$ be the simplification of $G$. If $G$ contains no vertices of degree two, then $3g_T(G) \geq n(\operatorname{si}(G))$.
\end{proposition}
\begin{proof}
Since $G$ is assumed to have no vertices of degree two, the base case is $G=C_2^2$, a doubled cycle of length two, i.e. $G$ contains two vertices with four parallel edges between them. In this case $g_T(G)=1$ and $n(\operatorname{si}(G))=0$, and so the result holds.

Now suppose that the desired inequality holds for all alternating decomposition graphs with no vertices of degree two that have fewer edges than $G$. Since $G$ does not contain any vertices of degree two, Lemma \ref{lemma:degree2} implies that $G$ contains a pair of parallel edges $e_1$ and $e_2$. Let $G'=G-\{e_1,e_2\}$, and let $e_{12}$ be the edge in $\operatorname{si}(G)$ corresponding to $e_1$ and $e_2$.

Suppose that $k(G') = k(G)+1$. Then $g_T(G') = g_T(G)$. The edge $e_{12}$ is a bridge in $\operatorname{si}(G)$, and thus $n(\operatorname{si}(G'))=n(\operatorname{si}(G))$. By induction, $3g_T(G')\geq n(\operatorname{si}(G'))$, and hence $3g_T(G)\geq n(\operatorname{si}(G))$.

Suppose that $k(G')=k(G)$. Then $g_T(G) = g_T(G')+1$ and $n(\operatorname{si}(G)) \leq n(\operatorname{si}(G'))+1$. Let $v_1$ and $v_2$ be the two vertices incident to $e_1$ and $e_2$ in $G$. For $i=1$ or $2$, we consider three cases:
\begin{enumerate}
\item the degree of $v_i$ is greater than two, 
\item the vertex $v_i$ has degree two and two distinct neighbors, or
\item the vertex $v_i$ has degree two and only one distinct neighbor.
\end{enumerate}

In order to apply our inductive hypothesis, we eliminate all vertices of degree two in $G'$ as follows. If $\deg v_i >2$, then nothing needs to be done. If $\deg v_i=2$ and $v_i$ has two distinct neighbors, then perform a two-path contraction at $v_i$. A two-path contraction does not change the Turaev genus of the graph but could decrease the nullity of the simplification of the graph by one. Suppose that $\deg v_i=2$ and $v_i$ has only one neighbor. Let $P_i$ be the maximal doubled path embedded in $G'$ with endpoints $v_i$ and $u_i$ such that every interior vertex of $P_i$ has exactly two neighbors. If every edge in $P_i$ is contracted, then both the Turaev genus and the nullity of the simplification of the resulting graph remain unchanged. 

Let $G''$ be the graph obtained from $G'$ by performing the above operations on $v_1$ and $v_2$. Then $G''$ has no vertices of degree two.  We have $g_T(G'') = g_T(G')$ and $n(\operatorname{si}(G''))+2\geq n(\operatorname{si}(G'))$. Therefore $g_T(G) = g_T(G'')+1$ and $n(\operatorname{si}(G)) \leq n(\operatorname{si}(G''))+3$. By the inductive hypothesis, we have $n(\operatorname{si}(G'')) \leq 3g_T(G'')$. Therefore
$$n(\operatorname{si}(G))\leq n(\operatorname{si}(G''))+3 \leq 3g_T(G'') + 3 = 3g_T(G).$$
\end{proof}

We will use the following lemma in the proof of Theorem \ref{thm:arbgenus}.
\begin{lemma}
\label{lemma:finite}
Let $n_1$ and $n_2$ be non-negative integers. There are a finite number of graphs $G$ such that $n(G)=n_1$ and such that $G$ contains  $n_2$ vertices of degree two.
\end{lemma}
\begin{proof}
Because nullity is additive with respect to disjoint union, it suffices to show the above statement for connected graphs. Let $T$ be a tree, and let $d_{12}(T)$ be the number of degree one or degree two vertices in $T$. Suppose that $T$ is a spanning tree of a graph $G$ with $n(G)=n_1$ where $G$ contains $n_2$ vertices of degree two. Hence $G$ is obtained from $T$ by adding $n_1$ edges. Each of the $n_1$ edges added to $T$ can make at most two of the vertices of degree one or two in $T$ have degree larger than two in $G$. Also, every degree two vertex in $G$ is either a degree one or a degree two vertex in $T$. Therefore $d_{12}(T) \leq 2n_1 + n_2$.

Every tree can be obtained from a single vertex by repeatedly adding pendant edges. Each pendant edge addition increases $d_{12}(T)$, and for a given tree, there are only finitely many ways to add a pendant edge. Thus the number of trees $T$ with $d_{12}(T) \leq 2n_1 + n_2$ is finite. There are only a finite number of ways to add $n_1$ edges to such a tree, and hence there exists a finite number of graphs $G$ with nullity $n_1$ that contain $n_2$ vertices of degree two. 
\end{proof}

We end the paper with the proof of Theorem \ref{thm:arbgenus}.
\begin{proof}[Proof of Theorem \ref{thm:arbgenus}]
For each doubled path equivalence class $c$ of reduced alternating decomposition graphs $G$ with $g_T(G)=k$, let $G_c$ be a representative such that no other representative of $c$ can be obtained from $G_c$ via a sequence of doubled path contractions. Let $V'$ be the set of vertices $v$ in $G$ such that $\deg v = 4$, each $v$ has exactly two distinct neighbors $u$ and $w$, there are two edges incident to both $u$ and $v$, and there are two edges incident to both $w$ and $v$. 

For each vertex $v\in V'$, there are two pairs of parallel edges incident to $v$, say parallel edges $e_{v,1}$ and $e_{v,2}$ and parallel edges $e_{v,3}$ and $e_{v,4}$. Let $E'$ be a set of edges containing exactly one pair of these parallel edges for each  $v\in V'$, that is
$E' = \{e_{v,1}, e_{v,2}~|~v\in V'\}.$
We claim that the graph $G_c - E'$, i.e. the graph obtained by deleting the edges set $E'$ from $G_c$, has the same number of components as $G_c$.

By way of contradiction, suppose that $G_c-E'$ has more components than $G_c$. Then there exists a minimal subset $E''$ of $E'$ such that $G_c-E''$ has one more component than $G_c$, but $G_c-S$ has the same number of components as $G_c$ for any proper subset $S$ of $E''$. Note that if an edge $e_{v,1}$ is in $E''$, then its parallel edge $e_{v,2}$ is also in $E''$. Therefore if $G_c'' = G_c / E''$, i.e. the contraction of the edges in $E''$ from $G_c$, then $G_c''$ is obtained from $G_c$ via a sequence of doubled path contractions.

Let $C''$ be a cycle in $G_c''$. Then there is a cycle $C$ in $G_c$ such that $C'' = C / (C\cap E'')$. Since $G_c$ is bipartite, it follows that $C$ has an even number of edges. Since adding any single edge of $E''$ to $G_c-E''$ connects two components of $G_c$, it follows that $C\cap E''$ has an even number of edges. Therefore, $C''$ has an even number of edges. Because each cycle of $G_c''$ has an even number of edges, the graph $G_c''$ is bipartite. Thus $G_c''$ is an alternating decomposition graph, which contradicts that no other representative of $c$ can be obtained from $G_c$ via a sequence of doubled path contractions.

Therefore $G_c - E'$ has the same number of components as $G_c$. Hence deleting each pair of parallel edges in $E'$ from $G_c$ decreases the Turaev genus by one, which implies that $|E'| \leq 2k$ and $|V'|\leq k$. Each vertex $v\in V'$ has degree two in the simplification $\operatorname{si}(G_c)$.

Any other vertex of degree two in $\operatorname{si}(G_c)$ arises from a vertex $v$ in $G_c$ with two distinct neighbors $v_1$ and $v_2$ such that there are $r$ edges between $v$ and $v_1$ and $s$ edges between $v$ and $v_2$ where $r+s$ is even and $\max \{r,s\}>2$. For each such vertex, there are two parallel edges whose removal decreases Turaev genus by one and does not change the simplification $\operatorname{si}(G_c)$. Because pairs of such vertices could be adjacent, there are at most $2k$ in $G_c$. 

Therefore $\operatorname{si}(G_c)$ has at most $3k$ vertices of degree two. Moreover $3k=3g_T(G_c) \geq n(\operatorname{si}(G_c))$. Because the nullity and the number of degree two vertices are bounded, Lemma \ref{lemma:finite} implies that there are only a finite number of candidates for the graph $\operatorname{si}(G_c)$. Because adding arbitrarily many parallel edges to an alternating decomposition graph increases its Turaev genus without bound, there are only a finite number of alternating decomposition graphs of a fixed Turaev genus whose simplification is a given graph. Therefore, there are only finitely many doubled path equivalence classes of alternating decomposition graphs of Turaev genus $k$.
\end{proof}

\bibliography{AlternatingDecompositions}{}
\bibliographystyle {amsalpha}

\end{document}